\documentclass[12]{amsart}

\usepackage{amsmath}
\usepackage{amsfonts}
\usepackage{rotating}
\usepackage{bbm}
\usepackage[all,cmtip]{xy}
\usepackage{subcaption}

\usepackage{amssymb}
\usepackage[latin1]{inputenc}
\usepackage{graphicx}
\usepackage{dsfont}
\usepackage{color}
\usepackage{hyperref}
\usepackage{enumitem}   

\renewcommand{\phi}{\varphi}
\newcommand{\C}{{\mathbb{C}}}
\newcommand{\R}{{\mathbb{R}}}

\newcommand{\Z}{{\mathbb{Z}}}
\newcommand{\N}{{\mathbb{N}}}

\advance\hoffset by -0.5 truecm
\usepackage{amssymb}
\newtheorem{Theorem}{Theorem}[section]
\newtheorem{Lemma}[Theorem]{Lemma}
\newtheorem{Corollary}[Theorem]{Corollary}

\newtheorem{Proposition}[Theorem]{Proposition}

\newtheorem{Definition}[Theorem]{Definition}
\newtheorem{Remark}[Theorem]{Remark}

\def\Z{{\mathbb Z}}
\def\R{{\mathbb R}}

\def \C{{\mathbb C}}

\def \0{\lambda_{0}}

\def \J{{\mathcal J}}
\usepackage{enumitem}
\newtheorem{Example}[Theorem]{Example}
\begin{document}

\title[$J^+$-like invariants of  the restricted three-body problem]{$J^+$-like invariants of periodic orbits of the second kind in the restricted three-body problem}

\author{Joontae Kim and Seongchan Kim}
 \address{School of Mathematics, Korea Institute for Advanced Study, 85 Hoegiro, Dongdaemun-gu, Seoul 02455, Republic of Korea}
 \email {joontae@kias.re.kr}
 \address{Universit\"at Augsburg, Universit\"atsstrasse 14, D-86159 Augsburg, Germany}
 \email {seongchan.kim@math.uni-augsburg.de}
\date{\today}



\begin{abstract}
We determine  three invariants:   Arnold's $J^+$-invariant as well as $\mathcal{J}_1$ and $\mathcal{J}_2$ invariants,  which were introduced by Cieliebak-Frauenfelder-van Koert,   of periodic orbits of the second kind near the heavier primary in the restricted three-body problem, provided that the mass ratio is sufficiently small. 
\end{abstract}

\maketitle

\setcounter{tocdepth}{1}
\tableofcontents

\section{Introduction}

The planar circular restricted three-body problem (PCR3BP) describes motions of a massless body under the influence of a Newtonian potential with two primaries.  Because this system is not completely integrable, finding periodic orbits is not easy. One approach to attack this problem  is finding families of periodic orbits. This approach was taken   for example by Hill \cite{Hill}, Darwin \cite{Darwin1, Darwin2}, Poincar\'e \cite{Poin}, Moulton \cite{Moulton}, Ljapunov \cite{Ljapunov}, and also  H\'enon \cite{Henon} and Bruno \cite{Bruno} for numerical investigations.

To find a starting point of a family of periodic orbits, one may vary dynamical systems. For example, if the mass ratio of the two primaries goes to either zero or one, the PCR3BP becomes in the limit   the rotating Kepler problem which is completely integrable. Also, if one switches off the rotating term, then the PCR3BP becomes the Euler problem of two fixed centers which is also completely integrable.  Since a periodic orbit  in completely integrable systems is well known,  one can take such a periodic orbit as a starting point of a family. 

In this paper, we consider families of periodic orbits starting from the ones in the rotating Kepler problem by following the approach  whose origin goes back to  Poincar\'e. For such families, the parameter will be the mass ratio of the two primaries. Fixing the energy and varying the mass ratio, Poincar\'e  obtained   periodic orbits in the PCR3BP from periodic orbits in the rotating Kepler problem \cite{Poin}. In the rotating Kepler problem, there are two kinds of periodic orbits: circular orbits and $T_{k,l}$-type orbits (including collision orbits). A circular orbit with a positive or a negative angular momentum is called the \textit{retrograde circular orbit} or the \textit{direct circular orbit}, respectively.  We remark that this definition depends on the sign of the magnetic term in the Hamiltonian, see Section \ref{sectionRKP}. Those circular orbits come from the retrograde and direct circular orbits in the inertial Kepler problem.  On the other hand, $T_{k,l}$-type orbits   come from   elliptic orbits in the inertial Kepler problem, more precisely they are  $k$-fold Kepler ellipses in an $l$-fold rotating coordinate system. The family of $T_{k.l}$-type orbits is called the \textit{$T_{k,l}$-torus family}. This family bifurcates from a  $|k-l|$-fold covered direct circular orbit and ends at a $(k+l)$-fold covered retrograde circular orbit, see Proposition \ref{probifurcation}. Note that the $T_{k,l}$-torus family is in fact a 2-parameter family: one variable represents the rotational symmetry and the other one represents      eccentricity  or, equivalently,  the energy.  Since the rotation of periodic orbits does not affect   the shape, without loss of generality we may think of the torus family as a 1-parameter family whose parameter is  eccentricity (or energy).  Note that in each $T_{k,l}$-torus family there are two types of $T_{k,l}$-orbits: one is direct and the other is retrograde, see Section \ref{sec:bif}.

Poincar\'e defined two classes of periodic orbits in the PCR3BP obtained from the ones in the rotating Kepler problem by varying the mass ratio: \textit{periodic orbits of the first kind} which  are obtained from the circular orbits and  \textit{periodic orbits of the second kind} which are obtained from  $T_{k,l}$-type orbits. The existence of periodic orbits of the first kind is discussed in detail by Birkhoff \cite{Birkhoff} and that of periodic orbits of the second kind by    Arenstorf \cite{Aren} and Barrar \cite{Barrar}. Note that in the $T_{k,l}$-torus family, there exists a unique torus consisting of collision orbits. In \cite{Gia}, Giacaglia showed that even collision orbits   can be continued to periodic orbits of the second kind.  Moreover, if $k > 2 \sqrt{2} l$,  for a sufficiently small mass ratio, a certain subfamily of the $T_{k,l}$-torus family in the rotating Kepler problem continuates to a 1-parameter family of periodic orbits of the second kind in the PCR3BP  which bifurcates from   a direct periodic orbit of the first kind and dies at a retrograde periodic orbit of the first kind, see \cite{Schmidt}.

In order to apply  holomorphic curve techniques in symplectic geometry to attack the dynamics, periodic orbits of the first kind are essential.  Indeed,  below the critical Jacobi energy direct and  retrograde circular orbits bound (disk-like) global surfaces of section in the rotating Kepler problem \cite{RKP} and in the PCR3BP for small mass ratios \cite{McGehee}.   Moreover, the Conley-Zehnder indices of $T_{k,l}$-type orbits can be computed from the Conley-Zehnder indices of the circular orbits, see \cite{RKP}. 

In this paper, however, we are interested in periodic orbits of the second kind and  determine their invariants. Since they are planar curves, a natural candidate for an invariant is   Arnold's $J^+$-invariant \cite{Arnold}: an  invariant of families of planar immersions which does not change under crossings through triple intersections and inverse self-tangencies. Hence, it provides an invariant for families of periodic orbits with finitely many triple intersections and inverse self-tangencies. Note that this invariant can change under direct self-tangencies, see Section \ref{subsection2.2}.

However, in a family of periodic orbits in the PCR3BP the particle can collide with one of the primaries or touch the boundary of the Hill's region. For the former disaster, denoted by $(I_0)$, birth or death of loops around the primary happens. For the latter one, birth or death of interior or exterior loops (depending on the connected component of the Hill's region in which the periodic orbit lies) happens. To see this picture more precisely, recall that if the energy is less than the first critical value, the Hill's region consists of three connected components: two bounded components and one unbounded component. If the latter disaster happens in one of the bounded components, then birth or death of exterior loops through cusps at the boundary of the Hill's region occurs, which we denote by $(I_{\infty})$. If the disaster happens in the unbounded component, then we  see a similar picture with interior loops instead of exterior loops, which we denote by $(I_{-\infty})$. This phenomenon occurs in more general systems, i.e.,  Stark-Zeeman systems, see  Section \ref{seubsection2.3}. Note that Arnold's $J^+$-invariant does not change under $(I_{\infty})$, but it does change under $(I_0)$ and $(I_{-\infty})$. This implies that the $J^+$-invariant is not a suitable invariant for families of periodic orbits in the PCR3BP. 

Recently,   Cieliebak-Frauenfelder-van Koert introduced two invariants \cite{invariant}: $\mathcal{J}_1$ and $\mathcal{J}_2$ invariants,  for families of periodic orbits in Stark-Zeeman systems   which are also invariant under $(I_0)$ (but not under $(I_{-\infty})$). Hence, they are suitable to study periodic orbits in the bounded Hill's region.  These two invariants are based on   Arnold's $J^+$-invariant, more precisely     the $\mathcal{J}_1$ invariant is defined by a combination of the  $J^+$-invariant and the winding number of a periodic orbit which plays a role in correcting errors arising from birth or death of additional loops. On the other hand, the $\mathcal{J}_2$ invariant is defined as the $J^+$-invariant of the pulled back image of a periodic orbit under the Levi-Civita embedding.  Note that all three invariants  $J^+$, $\mathcal{J}_1$ and $\mathcal{J}_2$  do not depend on the choice of the orientation of an orbit.

The main result of this paper is to determine the invariants of families of periodic orbits of the second kind in the PCR3BP. Before stating the theorem, we need to clarify the possible cases for $k$ and $l$.  Note that if the greatest common divisor of $k$ and $l$ is $m$, then every $T_{k,l}$-type orbit is $m$-fold covered. In this paper, we consider only simple  periodic orbits and hence we always assume that $k$ and $l$ are coprime. We exclude the case $k=1=l$ whose corresponding torus orbits are the Kepler ellipses. Note that the winding number of $T_{k,l}$-type orbits (with respect to the origin) is odd if and only if $k$ and $l$ have different parity.

We are now in position to state the main theorem.

\begin{Theorem}\label{maintheorem} 
Let $k>l$ with $k,l$ coprime. All  $T_{k,l}$-type orbits have the same $\mathcal{J}_1$ and $\mathcal{J}_2$ invariants, which are given by
\begin{eqnarray*}
\mathcal{J}_1 &=&  1 - k + \frac{k^2}{2} - \frac{l^2}{2}  \\
\mathcal{J}_2 &=& \begin{cases} (k-1)^2  -l^2     & \text{ if  $k$ and $l$ have different parity} ,   \\    1 - k + \frac{k^2}{4} - \frac{l^2}{4}     &  \text{ if $k$ and $l$ are both odd}  .     \end{cases}
\end{eqnarray*}
Consequently, there exists a small $\mu_{k,l}>0$ such that for each $\mu < \mu_{k,l}$ all periodic orbits   of the second kind near the heavier primary in the restricted three-body problem, which are obtained from $T_{k,l}$-type orbits, have the same invariants as above.
\end{Theorem}

\begin{Remark} \rm The previous theorem is a special case of Proposition \ref{jplusformeusladddd}. In the proposition,  we also provide formulas for  the two invariants $\mathcal{J}_1$ and $\mathcal{J}_2$ for $k<l$ and the $J^+$-invariant for any $k$ and $l$.
\end{Remark}

\begin{Remark}\rm  The two invariants $\mathcal{J}_1$ and $\mathcal{J}_2$ are not independent if the winding number is odd. Indeed, in this case they are related by $\mathcal{J}_2 = 2 \mathcal{J}_1   -1$, see \cite[Proposition 6]{invariant}. Observe that this is confirmed by the formula above.
\end{Remark}

\begin{Remark}\rm  Recall that switching off the rotating terms from the Hamiltonian of the PCR3BP gives rise to another integrable system called the Euler problem of two fixed centers.  The invariants for periodic orbits in the Euler problem will be determined  in the forthcoming paper \cite{Kim}. Once such invariants are determined, comparing those with the invariants for the rotating Kepler problem, given in this paper, one can see whether periodic orbits in the two integrable systems can be connected by Stark-Zeeman homotopies. 
\end{Remark}

\;\;\;

 {Acknowledgements:} We would like to express our deepest gratitude to  Urs Frauenfelder and Otto van Koert for  continued support and encouragement. We also would like to thank  the Institute for Mathematics of University of Augsburg for providing a supportive research environment and  the unknown referee for  very valuable comments.  J. Kim is partially supported by the National Research Foundation of Korea grant NRF-2016R1C1B2007662 funded by the Korean government and S. Kim by Deutsche Forschungsgemeinschaft grants CI 45/8-1 and FR 2637/2-1.

\section{Invariants of  periodic orbits}\label{sectioninvariant}

In this section, we briefly recall some relevant information on  the three invariants $J^+$, $\mathcal{J}_1$ and $\mathcal{J}_2$ of plane curves. We also introduce  Viro's formula for the $J^+$-invariant, which is useful for the computational purpose.

\subsection{Immersions in the plane}
Since we are interested in (simple covered) planar periodic orbits, throughout the paper a curve always means a map $\gamma : S^1 \rightarrow \C$ from the circle into the complex plane. By an \textit{immersion}, we mean an immersed curve   $\gamma:S^1\to \C$ which is considered up to orientation preserving reparametrization. By abuse of  notation, an immersion is identified with its image $K:=\gamma(S^1)\subset \C$. An immersion is said to be   \textit{generic} if it has only transverse double points. One might expect that generically, a homotopy of immersions consists only of generic immersions. However, this is not the case. Indeed, during a homotopy of immersions, triple points might appear. A more serious situation is the appearance of \textit{new} double points at which  the tangent vectors are parallel. This event is called a \textit{direct self-tangency} (or an \textit{inverse self-tangency}) if the two tangent vectors point in the same direction (or in  opposite directions).   By a  \textit{generic homotopy} $(K^s)_{s\in [0,1]}$, we mean a homotopy of generic immersions except at finitely many $s\in (0,1)$ at which the aforementioned  three disasters can occur, see Figure \ref{threedisasters}. Note that in particular under the crossing through a self-tangency the number of double points changes. 
\begin{figure}[h]
\begin{subfigure}{0.8\textwidth}
   \centering
  \includegraphics[width=1\linewidth]{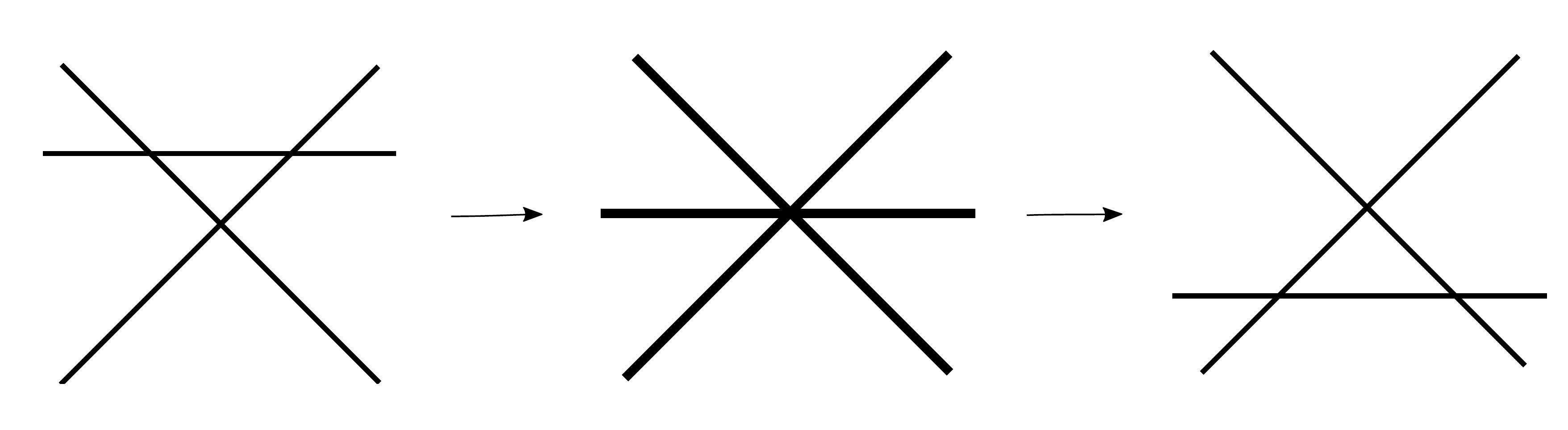}
  \caption{ A triple point}
  \label{  }
\end{subfigure}
\begin{subfigure}{0.8\textwidth}
  \centering
  \includegraphics[width=1\linewidth]{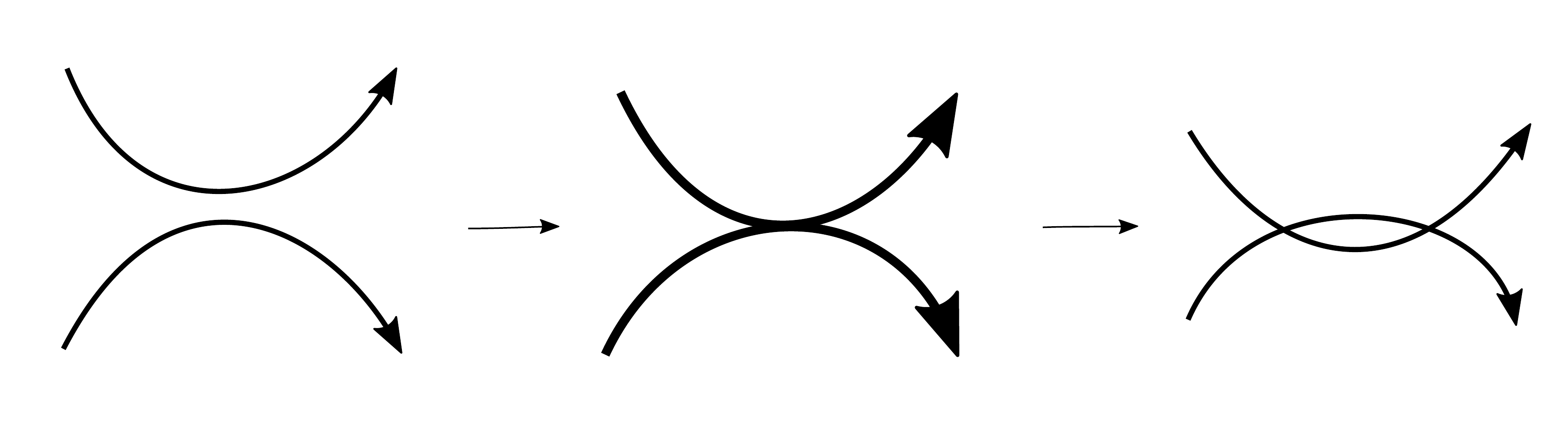}
  \caption{ A direct self-tangency}
 \label{directself}
\end{subfigure}
\begin{subfigure}{0.8\textwidth}
  \centering
  \includegraphics[width=1\linewidth]{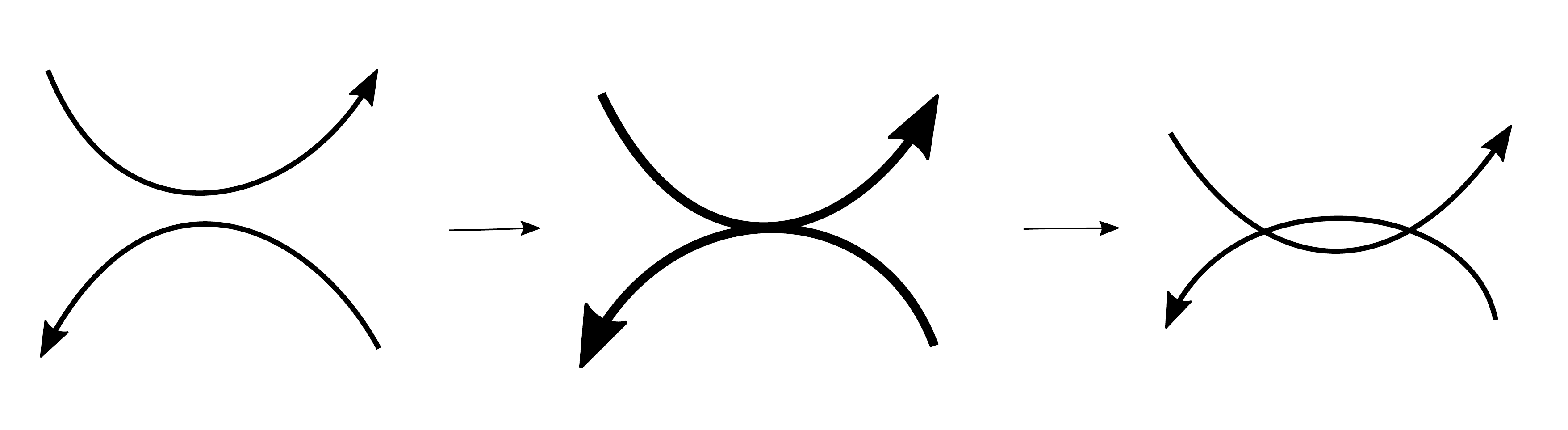}
  \caption{ An inverse self-tangency}
\end{subfigure}
\caption{ Three disasters which can occur during a generic homotopy}
 \label{threedisasters}
\end{figure}

\subsection{Arnold's $J^+$-invariant}\label{subsection2.2}
Arnold's $J^+$-invariant, denoted by $J^+(K)\in 2\Z$ for a generic immersion $K\subset \C$, is an invariant that is characterized by the following properties:
\begin{enumerate}[label=(\roman*)]
\item it is independent of the orientation;  \vspace{1mm}

\item it does not change under crossings through an inverse self-tangency or a triple point;  \vspace{1mm}

\item it increases and decreases by two under a positive crossing and a negative crossing through a direct self-tangency, respectively. By a positive or negative crossing, we mean a direct self-tangency through which the number of double points increases or decreases, respectively, see Figure \ref{directself}; \vspace{1mm}

\item  define standard curves $K_j$, $j \in \Z_{\geq 0}$,  as follows: $K_0$ is the figure eight and for each $j \neq 0$, the  curve $K_j$ is the circle with $(j-1)$ interior loops whose rotation number equals $j$, where the rotation number of a curve is defined to be the winding number of the tangent vector. In Figure \ref{standardcurves} we illustrate   some examples of the standard curves. We then have 
\begin{equation*}
J^+(K_j)=\begin{cases}
		2-2j & \text{$j\ne 0$,} \\
		0 & \text{$j=0$.}
	\end{cases}  
\end{equation*}
\end{enumerate}
Suppose that $K$ generically homotopes to some standard curve $K_j$ through $n_1$ positive crossings and $n_2$ negative crossings. Then the $J^+$-invariant of $K$ is determined by the homotopy via the above four rules:  $J^+(K)=J^+(K_j) - 2n_1 + 2n_2.$ Note that the $J^+$-invariant exists and is uniquely defined by the above properties. Moreover, it is additive under connected sums, see \cite[Chapter 1]{Arnold}.	
\begin{figure}[h]
  \centering
  \includegraphics[width=1.0\linewidth]{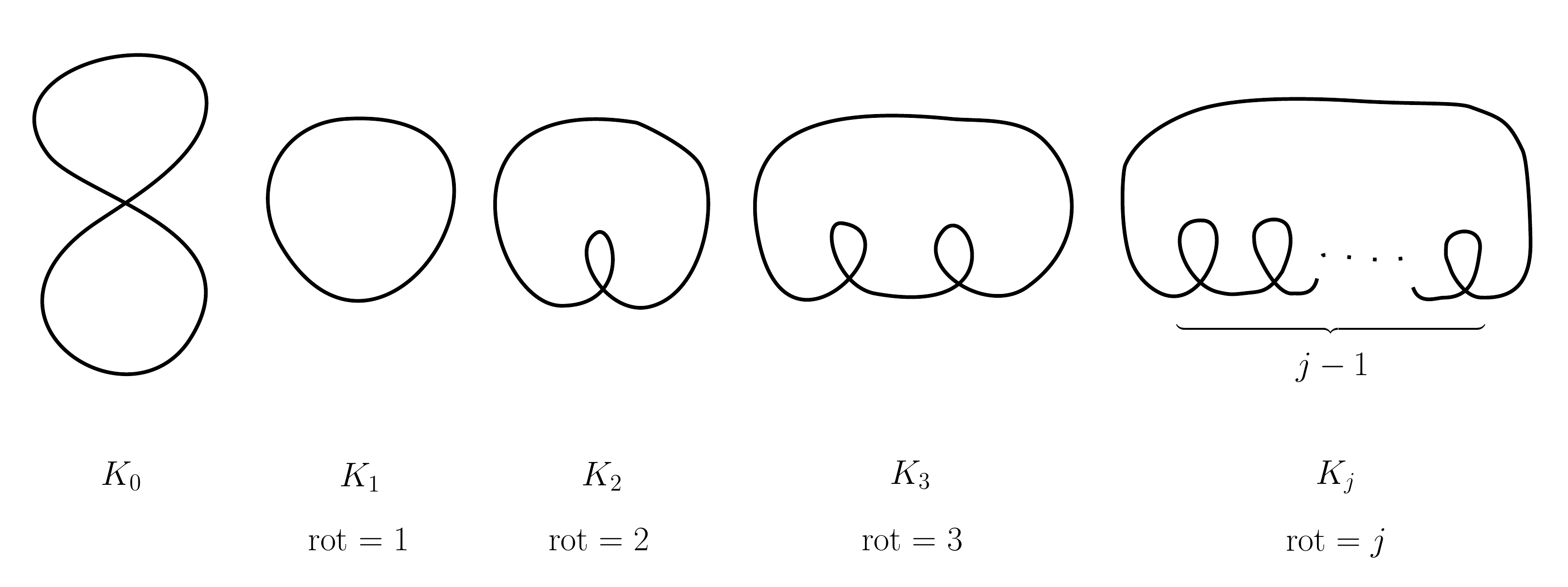}
  \caption{The standard curves}
  \label{standardcurves}
\end{figure}

\subsection{Viro's formula}
 For a generic immersion $K\subset \C$, we let $\Lambda_K$ be the  {set of all connected components of the complement} $\C\setminus K$   and let $D_K$ be the  {set of all double points of $K$}.  The  {winding number of $K$ around a component $C\in \Lambda_K$}, denoted by $w_C(K)$, is defined to be the winding number of $K$ around an interior point in $C$. Note that this number does not depend on the choice of the interior point, so $w_C(K)$ is well-defined. In order to define the index of a double point, we observe  that every double point is adjacent to four  components (one component may be counted twice).   The  {index of $K$ at the double point $p\in D_K$}, denoted by $\text{ind}_p(K)$, is then defined to be the arithmetic mean of the winding numbers of the adjacent  four   components.  We now state  Viro's formula.

\begin{Theorem}\label{virotheorem}
Let $K$ be a generic immersion with $n$ double points. Then   Arnold's $J^+$-invariant of $K$ is given by
$$J^+(K)=1+n-\sum_{C \in \Lambda_K}(w_C(K))^2+\sum_{p \in D_K} (\text{ind}_p(K))^2.$$	
\end{Theorem} 
This was proven in \cite[Corollary 3.1.B and Lemma 3.2.A]{Viro}. This formula tells us that the information of all double points and the winding number around each component of the complement completely determines the $J^+$-invariant.

\begin{Example} \rm
We compute the $J^+$-invariant of a generic immersion $K\subset \C$ given in Figure \ref{example_viro}. We assume that $K$ is oriented  counterclockwise. The complement $\C\setminus K$ has 5 components (including the unbounded component) and the number in the interior of each component means its winding number. The immersion has three double points, and the index at each double point is given by $$\frac{1}{4}\cdot (0+1+1+2)=1.$$By  Viro's formula, we obtain
$$J^+(K)=1+3-(3\cdot 1^2+1\cdot 2^2+1\cdot 0^2)+(3\cdot 1^2)=0.$$

\begin{figure}[h]
  \centering
  \includegraphics[width=0.3\linewidth]{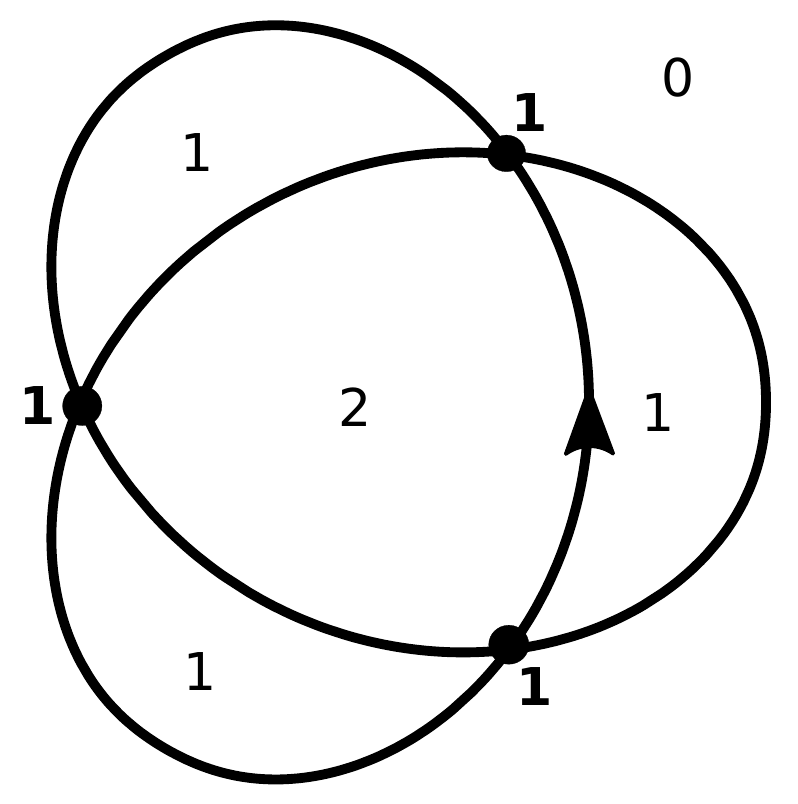}
  \caption{ A generic immersion $K\subset \C$} 
  \label{example_viro}
\end{figure}

\end{Example}

\begin{Example} \rm
	In Figure \ref{example_virodd}, there is a double point which is adjacent to four components, where the component of winding number 1 is counted twice. The index of the double point is then given by
	$$
	\frac{1}{4}\cdot(0+1+1+2)=1.
	$$
	Hence, we obtain
	$$
	J^+(K)=1+1-(1^2+2^2)+1^2=-2.
	$$

\begin{figure}[h]
  \centering
  \includegraphics[width=0.3\linewidth]{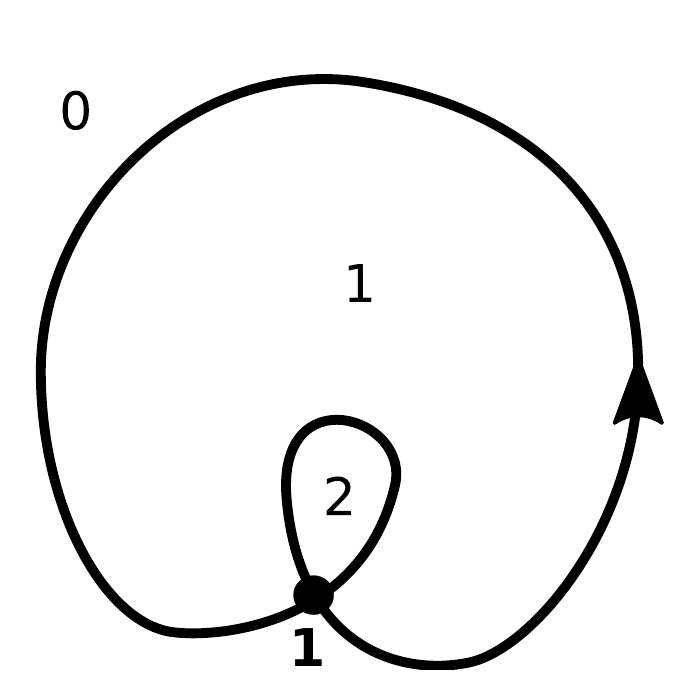}
  \caption{ A generic immersion $K\subset \C$} 
  \label{example_virodd}
\end{figure}

\end{Example}

It is worth mentioning that the winding number around the unbounded component is always zero for any generic immersion.

\subsection{The $\mathcal{J}_1$ and $\mathcal{J}_2$ invariants}\label{seubsection2.3}
In the paper \cite{invariant}, Cieliebak, Frauenfelder and van Koert defined the two $J^+$-like invariants $\mathcal{J}_1$ and $\mathcal{J}_2$ of generic   1-parameter families of simple periodic orbits in (planar) Stark-Zeeman systems. A \textit{Stark-Zeeman system} is a Hamiltonian system whose Hamiltonian is given by the Hamiltonian of the Kepler problem to which might be added some potential and  magnetic terms, where a potential term depends only on the position and a magnetic term also depends on the momentum. For example, the restricted three-body problem and the rotating Kepler problem, which we treat in this paper, are Stark-Zeeman systems, see \cite[Section 3.2]{invariant}.   In what follows, we assume that the magnetic term is nonzero and   has no singularities.

Assume that  $c_1 \in \R \cup \left \{ \infty \right \}$ has the following properties:
\begin{itemize}
\item $c_1$ is a regular value of the Hamiltonian;
\item for any $c<c_1$, the \textit{Hill's region} $\mathcal{K}_c$, which is by definition the projection of the energy hypersurface to the configuration space,  contains a unique bounded component $\mathcal{K}_c^b$ which is diffeomorphic to a closed disk with the origin removed;
\end{itemize}
By $\infty$ being a regular value of the Hamiltonian, we mean that any real number is a regular value. In other words, the Hamiltonian has no critical values.

 Physically, the Hill's region is a region in the configuration space in which the massless body for fixed energy can lie. In particular, velocity of the particle vanishes if and only if the particle lies on the boundary of the Hill's region. Let us fix  $c<c_1$.   We now regularize the system so that periodic orbits are allowed to pass through the origin.  Note that the regularized bounded component of the Hill's region, which we again denote by $\mathcal{K}_c^b$, is diffeomorphic to a closed disk. We denote by $\mathcal{K}_c^u$ the unbounded component of the Hill's region.

Consider an orbit $q : \R \rightarrow \R^2$ in $\mathcal{K}_c^b$ and fix a  point $q_0=q(t_0)$. Then it satisfies one of the following:
\begin{enumerate}[label=(\roman*)]
\item $q_0 \in \text{int}( \mathcal{K}_c^b )\setminus \left \{ (0,0) \right \}:$ in this case, the velocity $\dot{q}(t_0)$ does not vanish and hence the orbit $q$ is an  immersion near $q_0$;  \vspace{1mm}

\item $q_0 \in \partial \mathcal{K}_c^b  $ : recall that on the boundary of the Hill's region, the velocity vanishes. Cieliebak-Frauenfelder-van Koert showed that in this situation  the orbit $q$ has a cusp at $t=t_0$, see \cite[Lemma 1]{invariant}.  Moreover, if $q$ is a member of a 1-parameter family $q^s$ of nearby  orbits of the same energy with $q^0=q$, then through the cusp birth or death of a  loop  can occur, see Figures \ref{ifiniity}, \ref{elddddksd};  \vspace{1mm}

\item $q_0=(0,0) :$ similar to the previous case, the orbit has a cusp at $t=t_0$. If the orbit is a member of    the 1-parameter family described as in the previous case, then through the cusp birth or death of a loop around the origin can occur, see \cite[Lemma 2]{invariant} and Figure \ref{izero}.
\end{enumerate}

Because of  the  additivity under connected sum, the event (ii) for $\partial \mathcal{K}_c^b$ above, which gives rise to birth or death of exterior loops,  does not influence   the $J^+$-invariant. However, through  the event (ii) for $\partial \mathcal{K}_c^u$ or (iii), the $J^+$-invariant can change. This implies that a generic 1-parameter family of periodic orbits in a Stark-Zeeman system is not just a generic homotopy and the $J^+$-invariant is not a suitable invariant for this family. In view of the discussion so far, Cieliebak-Frauenfelder-van Koert introduced the following notion which is a   generalization of a generic homotopy.

\begin{Definition}\label{defhomotopy}  \rm (\cite[Definition 1]{invariant}) 
A 1-parameter family $(K^s)_{s\in [0,1]}$ of closed curves in $\C$ is called a {\bf Stark-Zeeman homotopy} if each $K^s$ is a generic immersion in $\C^*=\C\setminus\{0\}$ except for finitely many $s\in (0,1)$ and for such  $s\in (0,1)$   the following events happen:\vspace{2mm}\\
{\bf ($I_0$)} birth or death of  loops around the origin through cusps at the origin;\vspace{2mm}\\
{\bf ($I_\infty$)} birth or death of exterior loops through cusps at the boundary of the Hill's region;\vspace{2mm}\\
{\bf ($II^+$)} crossings through inverse self-tangencies;\vspace{2mm}\\
{\bf ($III$)} crossings through triple points, see Figure \ref{stargdkzlbsdfs}.
\end{Definition}

\begin{figure}[h]
\begin{subfigure}{0.65\textwidth}
  \centering
  \vspace{3mm}
  \includegraphics[width=1\linewidth]{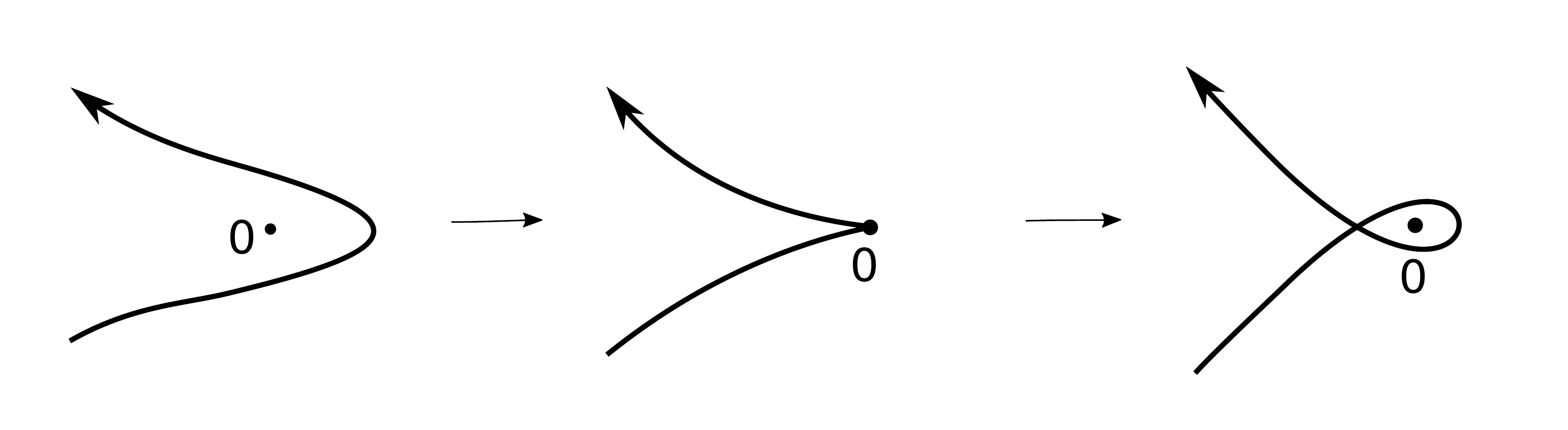}
  \caption{ Event {\bf ($I_0$)}: birth of a loop around the origin}
  \label{izero}
\end{subfigure}
\begin{subfigure}{0.65\textwidth}
  \centering
  \vspace{3mm}
  \includegraphics[width=1\linewidth]{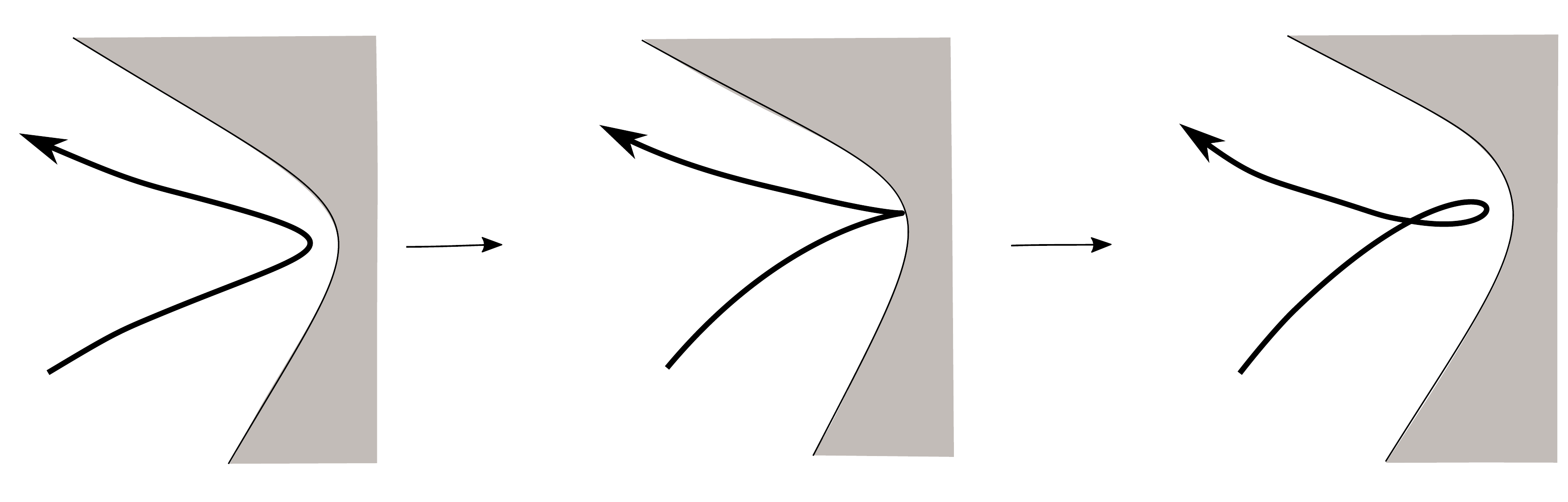}
  \caption{ Event {\bf ($I_\infty$)}: birth of an exterior loop}
 \label{ifiniity}
\end{subfigure}
\begin{subfigure}{0.65\textwidth}
  \centering
  \vspace{3mm}
  \includegraphics[width=1\linewidth]{is.pdf}
  \caption{ Event {\bf ($II^+$)}: crossing through an inverse self-tangency}
 \label{sdfsdfsfsdfsfd33331121df}
\end{subfigure}
\begin{subfigure}{0.65\textwidth}
  \centering
  \vspace{3mm}
  \includegraphics[width=1\linewidth]{tri.pdf}
  \caption{ Event {\bf ($III$)}: crossing through a triple point}
 \label{wesfdbvsisoi3}
\end{subfigure}
\begin{subfigure}{0.65\textwidth}
  \centering
  \vspace{3mm}
  \includegraphics[width=1\linewidth]{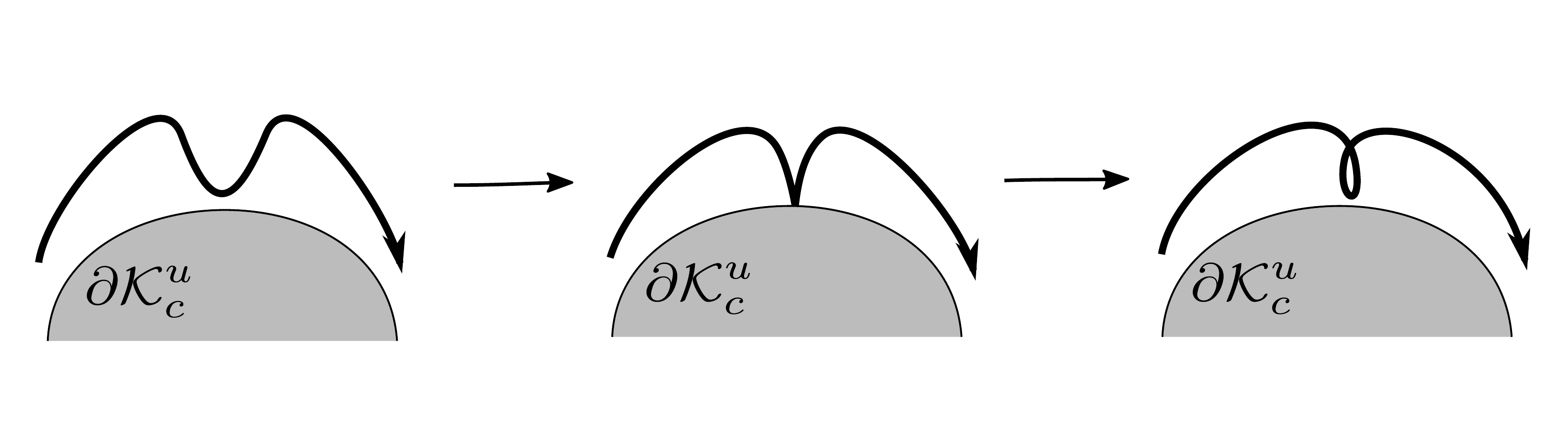}
  \caption{ Event {\bf ($I_{-\infty}$)}: birth of an interior loop}
 \label{elddddksd}
\end{subfigure}
\caption{Possible events along a family of periodic orbits in a Stark-Zeeman system}
\label{stargdkzlbsdfs}
\end{figure}

\begin{Remark} \rm   By \cite[Lemma 1]{invariant}, along a family of periodic orbits above $\mathcal{K}_c^u$, birth or death of interior loops which do not wind around the origin can happen. We denote by $(I_{-\infty})$ this event. Note that the event $(I_{-\infty})$ does not happen during a Stark-Zeeman homotopy and the invariants $\mathcal{J}_1$ and $\mathcal{J}_2$ might change under this event, see Sections \ref{sectheorem1} and \ref{sectionslkdhf}.
\end{Remark}

For a generic immersion $K\subset \C^*$, we write $w_0(K)\in \Z$ for the  {winding number around the origin}. The {\bf $\mathcal{J}_1$ invariant} is defined by 
	$$\J_1(K):=J^+(K)+\frac{(w_0(K))^2}{2}.$$
Recall that  the  {Levi-Civita regularization map} is given by the complex squaring map
	$$L:\C^*\to \C^*,\quad v\mapsto v^2.$$
Note that for a generic immersion $K\subset \C^*$ the preimage $L^{-1}(K)\subset \C^*$ is also a generic immersion. Since the restriction $L|_{L^{-1}(K)}:L^{-1}(K)\to K$ is a 2-fold covering, the preimage $L^{-1}(K)$ has at most two components. The {\bf $\J_2$ invariant}   is then defined as follows.\vspace{2mm}\\
If $w_0(K)$ is \textit{odd}, then $L^{-1}(K)$ is connected and   it is a single generic immersion. In this case we set
$$\J_2(K):=J^{+}(L^{-1}(K)).$$
If $w_0(K)$ is \textit{even}, then $L^{-1}(K)$ consists of two components, i.e., two copies of a generic immersion. Choose one generic immersion $\widetilde{K}$ of $L^{-1}(K)$ and we set
$$\J_2(K):=J^+(\widetilde{K}).$$
	 Since $L$ is the complex squaring map, the two generic immersions of $L^{-1}(K)$ are related by a $\pi$-rotation in $\C^*$. This implies that $\J_2(K)$ does not depend on the choice of $\widetilde{K}$.

\begin{Proposition}{\rm (\cite[Propositions 4 and 5]{invariant})} 
The two invariants $\J_1$ and $\J_2$ do not change during Stark-Zeeman homotopies.
\end{Proposition}

The relation between the invariants $\mathcal{J}_1$ and $\J_2$ are partially given as follows.

\begin{Proposition}\label{prorelation}{\rm (\cite[Proposition 4.6]{invariant})} 
If $w_0(K)$ is odd, then $\J_2(K)=2\J_1(K)-1$.	
\end{Proposition}
It is known that if $w_0(K)$ is even, the invariants $\J_1$ and $\J_2$ are completely independent. See \cite[Proposition 4.7]{invariant}.

 \section{The rotating Kepler problem}\label{sectionRKP}

In this section, we study some properties of the rotating Kepler problem which we need to prove the main theorem. In particular, we see bifurcations of periodic orbits with negative Kepler energies.  In what follows, $\gamma$ denotes a periodic orbit in the inertial Kepler problem and $\alpha$ a periodic orbit in the rotating Kepler problem. 

We first fix some conventions. The symplectic form on $T^* \R^2$ with coordinates $(q_1, q_2, p_1, p_2)$ is given by $\omega = dp_1 \wedge dq_1 + dp_2 \wedge dq_2$ and the Hamiltonian vector field  $X_H$ of a Hamiltonian $H$ is defined by the relation $\omega( X_H , \cdot ) = -dH$. The Hamiltonian flow $\phi_H^t$ associated with the Hamiltonian $H$ is defined to be the unique solution of the initial value problem: $(d/dt)\phi_H^t(x) = X_H ( \phi_H^t(x) )$ and $\phi_H^0(x) = x$ for all $x$ and  $t$.

\subsection{Hamiltonian and Integrals}
Recall that the Hamiltonian of the restricted three-body problem is given by
\begin{equation*}
H_{\text{3BP}}(q,p) = \frac{1}{2}|p|^2 - \frac{1-\mu}{|q-E|} - \frac{\mu}{|q-M|} + q_1 p_2 - q_2 p_1,
\end{equation*}
where $E=(\mu, 0)$ is the Earth, $M=(-(1-\mu), 0)$ is the Moon and  $\mu$ is the mass ratio of the two primaries. The rotating Kepler problem arises as the limit of this problem if we switch off the Moon, in other words, the mass ratio $\mu$ goes to zero. It follows that the Hamiltonian of the rotating Kepler problem is given by $H =E+L : T^* (\R^2 \setminus \left \{ (0,0) \right \} ) \rightarrow \R$, where
\begin{equation*}
E(q,p) = \frac{1}{2}|p|^2 - \frac{1}{|q|}
\end{equation*}
is the Hamiltonian of the inertial Kepler problem and 
\begin{equation*}
L(q,p)=q_1 p_2 - q_2 p_1
\end{equation*}
is the angular momentum. Note that the two quantities $E$ and $L$ are integrals of the system, namely $\left\{ H, E \right \} = \left\{ H, L \right\}=0$, where $\left \{ \cdot, \cdot \right\}$ means the Possion bracket of two smooth functions. In particular, by the Noether theorem they are constant along orbits: $E( \phi^t_H(x)) =E(x)$ and $L( \phi^t_H(x)) =L(x)$ for any $x$ and $t$.

By completing squares, we see that

\begin{equation*}\label{Hamiltonianewq}
H(q,p) = \frac{1}{2} ( (p_1 - q_2)^2 + (p_2 + q_1)^2 ) + U_{\text{eff} }(q),
\end{equation*}
where
\begin{equation*}
U_{\text{eff} }(q) = -\frac{1}{|q|} - \frac{1}{2}|q|^2
\end{equation*}
is the effective potential. The twisted terms in the momenta are due to the Coriolis force and the term $-(1/2)|q|^2$ in the effective potential is due to the centrifugal force. The Hamiltonian $H$ admits an $S^1$-family of critical points of the form $(q_1, q_2, q_2, -q_1)$, where $|q| = 1$.  The associated energy level $c_J=-3/2$ will be called the critical Jacobi energy.

For the later use,  we introduce  polar coordinates $(r, \theta)$ so that $q_1 = r\cos \theta$ and $q_2 = r \sin \theta$. The momenta $(p_r, p_{\theta})$ are defined by the canonical relation $p_1 dq_1 +  p_2 dq_2 = p_r dr + p_{\theta} d\theta$. Note that the angular momentum is given by $L = p_{\theta}$. Under the coordinates change, the Hamiltonian becomes
\begin{equation}\label{hamltiodnsRKP}
H(r, \theta, p_r, p_{\theta}) = \frac{1}{2} \bigg( p_r^2 + \frac{ p_{\theta}^2}{r^2} \bigg) - \frac{1}{r} + p_{\theta}
\end{equation}
and the Hamiltonian vector field is given by
\begin{equation}\label{transvec}
X_H = p_r \partial_r + \bigg( \frac{p_{\theta}}{r^2 } +1 \bigg)\partial_{\theta} + \frac{ p_{\theta}^2 - r}{r^3} \partial_{p_r} .
\end{equation}

\subsection{Hill's regions}
Given $c \in \R$, the Hill's region associated with the energy level $c$ is given by
\begin{equation*} 
\mathcal{K}_c := \pi ( H^{-1}(c) ) = \left\{ q= (q_1, q_2) \in \R^2 \setminus \left\{ (0,0) \right\} : U_{\text{eff}}(q) \leq c \right\},
\end{equation*}
where $\pi : T^* ( \R^2 \setminus \left\{ (0,0 )\right\} ) \rightarrow \R^2 \setminus \left\{ (0,0) \right\}$ is the projection along the fiber.  Consider the function 
\begin{equation}\label{functionf}
f(r) = - \frac{1}{r} - \frac{1}{2} r^2 , \;\;\; r>0.
\end{equation}
For $ c < c_J$, the equation $f(r) = c$ has two roots $r_1 < 1 < r_2$ having the property that 
\begin{equation*}
\mathcal{K}_c = (D_{r_1} \setminus \left \{( 0,0)\right \}) \cup ( \R^2 \setminus \text{int} (D_{r_2}) ),
\end{equation*}
where $D_r$ is the closed disk centered at the origin with radius $r$. Note that $\mathcal{K}_c$ consists of two connected components: the bounded component $D_{r_1} \setminus \left \{( 0,0)\right \}$ and the unbounded  component $  \R^2 \setminus \text{int} (D_{r_2}) $. As $c$ goes to $c_J$, the two radii $r_1$ and $r_2$ tend to $1$. At $ c=c_J$, we have $r_1 = r_2 =1$. For    $c > c_J$, the equation $f(r) = c$ has no solution and hence $\mathcal{K}_c =\R^2 \setminus \left \{ (0,0) \right\}$.

 \subsection{Periodic orbits}\label{sec:peor} In the following we assume that the Kepler energy is negative, $E<0$.   Recall that if $E<0$, then a periodic orbit in the inertial Kepler problem is either an elliptic orbit (including a collision orbit) or a circular orbit.  Note  that (noncollision) elliptic orbits have eccentricity $e \in (0,1)$, collision orbits are of $e=1$ and circular orbits are of $e = 0$.

 Since $\left \{ E, L \right \} =0$,  the   flows of  the  Hamiltonian vector fields $X_E$ and $X_L$ commute. It follows that the flow of their sums equals the composition of the flows of each vector field, namely
\begin{equation*}
\phi_H^t = \phi_{E+L}^t = \phi_{E}^t \circ \phi_L^t = \phi_L^t \circ \phi_E^t.
\end{equation*}
It follows that an orbit $\alpha$ of the rotating Kepler problem has the form $\alpha(t) = \exp( it ) \gamma(t)$, where $\gamma$ is an inertial Kepler orbit. Note that $\alpha$ is not necessarily periodic.

Suppose that $\gamma$ is a   Kepler ellipse of period $T>0$. Then $\alpha$ is   periodic  if and only if there exist some relatively prime $k, l \in \N$ satisfying  $ kT  = 2 \pi l$.  This implies that $\alpha$ is a $k$-fold covered Kepler ellipse in an $l$-fold covered coordinate system. In particular, it is $2 \pi l$-periodic. This observation gives rise to the following definition. 
\begin{Definition} \rm (\cite[Section 4]{RKP})   A $\tau$-periodic orbit $\alpha(t) = \exp( it ) \gamma(t)$   is called a \textit{$T_{k,l}$-type orbit} if $\tau = 2\pi l$ for some $l \in \N$ and $\gamma(t)$ is a  Kepler ellipse of period $T$ satisfying $kT = 2 \pi l$. A Liouville torus on which $T_{k,l}$-type orbits lie is referred to as a \textit{$T_{k,l}$-torus.}  The family of all $T_{k,l}$-tori is called the \textit{$T_{k,l}$-torus family}. Some examples of torus families are given in Figure \ref{circular}.
\end{Definition}

\begin{Remark} \rm Whenever $T_{k,l}$-torus families are under consideration, we always assume that $k$ and $l$ are relatively prime. This  implies that every $T_{k,l}$-type orbit is simple covered. 
\end{Remark}

     Recall that a  Kepler ellipse  $\gamma$ of energy $E$  is described by the equation
\begin{equation}\label{eqkepler7}
r = \frac{ a(1-e^2)}{1+e \cos ( \theta - \alpha)},
\end{equation}
where $r=|q|$, $a=-1/2E$ is the semi-major axis,  $\theta$ is the angle between the major axis and the position vector and $\alpha$ is the argument of the perihelion, see Figure \ref{conic}. 
\begin{figure}[t]
 \centering
 \includegraphics[width=0.65\textwidth, clip]{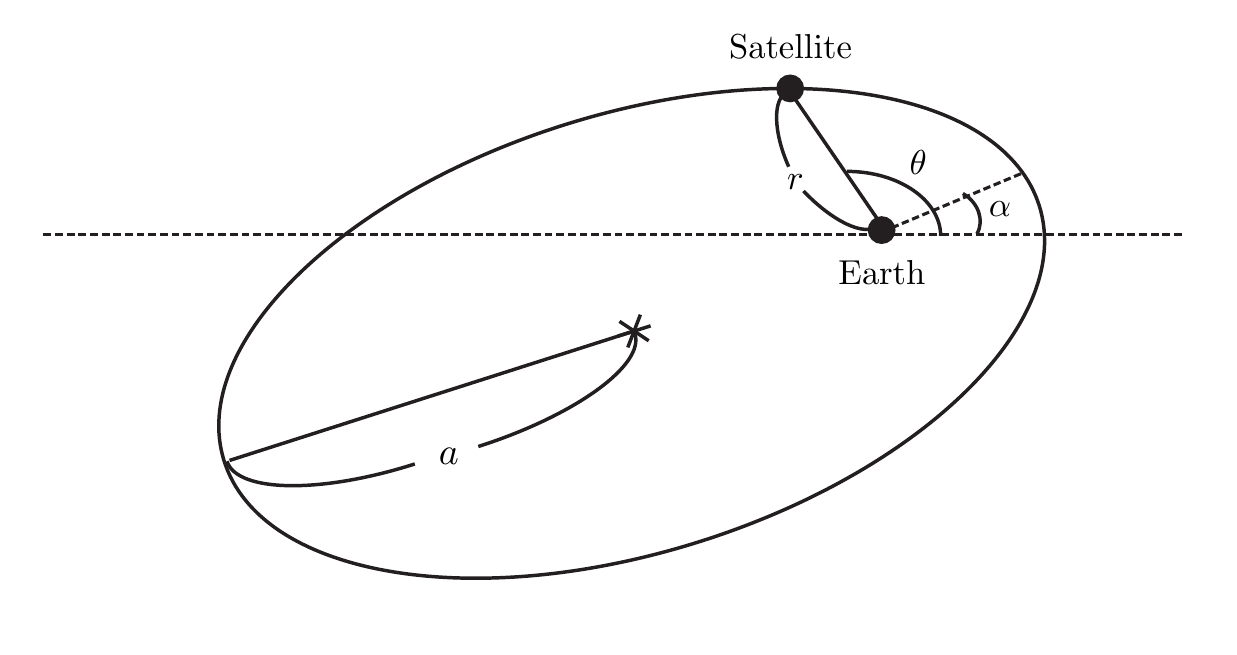}
 \caption{A Kepler ellipse}
 \label{conic}
\end{figure}
It follows that the \textit{perihelion} (the nearest point to the origin) of $\gamma$ has radius 
\begin{equation*}
r_{\min}= \frac{1-e^2}{-2E  (1+e)} = \frac{1-e}{-2E }
\end{equation*}
and   the \textit{aphelion} (the farthest point to the origin) has radius 
\begin{equation*}
r_{\max}= \frac{1-e^2}{-2E(1-e)} = \frac{1+e}{-2E}.
\end{equation*}
Since a $T_{k,l}$-type orbit $\alpha$ is obtained from the $k$-fold covering of   $\gamma$, there exist precisely $k$ perihelions and $k$ aphelions on $\alpha$ whose radii are given by
\begin{equation}\label{eqrmnsridncx}
r_{\min}=   \frac{1-e}{-2E_{k,l} } \;\;\; \text{ and } \;\;\; r_{\max}=   \frac{1+e}{-2E_{k,l}},
\end{equation}
respectively.

\begin{Lemma}{\rm (\cite[Section 6]{RKP})} The Kepler energy $E$ is constant in each torus family.  
\end{Lemma}
\begin{proof} Let $\alpha$ and $\gamma $ be as above.  In view of  Kepler's third law $ T^2 = (2\pi)^2 / (-2E )^3$ and the fact that $ k T = 2 \pi l$, we obtain that $4\pi^2 l^2 / k^2 = \pi^2/ (-2E^3)$ which implies that $E = -(1/2)(k/l)^{2/3}$. This completes the proof of the lemma.
\end{proof}
We denote by $E_{k,l}$ the Kepler energy of the $T_{k,l}$-torus family, namely
\begin{equation}\label{eqtorusenergy}
E_{k,l} = - \frac{1}{2} \bigg( \frac{k}{l} \bigg)^{2/3}.
\end{equation}
Observe that
\begin{equation}\label{eqkandlcompare}
\begin{cases} k>l &  \;\; \text{ if }E_{k,l}<-1/2 \\ k=l     &   \;\; \text{ if }E_{k,l}=-1/2\\ k<l &    \;\; \text{ if }-1/2<E_{k,l}<0. \end{cases}
\end{equation}

\;\;

Suppose that $\gamma$ is a circular orbit. Since $\exp( i t)$ takes circular orbits to circular orbits, $\alpha(t)$ is always a circular orbit. In particular, it is  periodic.  Recall that in the inertial Kepler problem, on each negative energy hypersurface there exist precisely two circular orbits, where one rotates counterclockwise and the other rotates  clockwise. Note that in our convention  the coordinate system rotates clockwise, see for example \cite[Proposition 10.1.2]{foundation}.   A circular orbit in the rotating system is then called a \textit{direct circular orbit} (or a \textit{retrograde circular orbit}) if the associated circular orbit in the inertial system rotates clockwise (or counterclockwise). In other words, from the viewpoint of the inertial system, the direct circular orbit rotates in the same direction as the coordinate system and the retrograde circular orbit rotates in the opposite direction. This implies that  a  retrograde circular orbit has bigger angular momentum than a direct one.

Recall that the eccentricity $e$ of a Kepler ellipse satisfies the relation
\begin{equation}\label{eccentrieq}
e^2 = 2EL^2 +1 = 2E(H-E)^2+1 .
\end{equation} 
 Plugging $e=0$ and $H=c$ into (\ref{eccentrieq}) gives rise to the equation 
\begin{equation}\label{equationecc}
2E(c-E)^2+1 =0
\end{equation}
whose graph consists of two connected components. For $E<0$, the inequality $2E(c-E)^2+1\geq 0$ is solved for $c$ lying in  an interval containing the value $c=E$, which corresponds to $e=1$, in its interior. Each such an interval represents a $T_{k,l}$-torus family, see Figure \ref{circular}.
\begin{figure}[t]
 \centering
 \includegraphics[width=0.7\textwidth, clip]{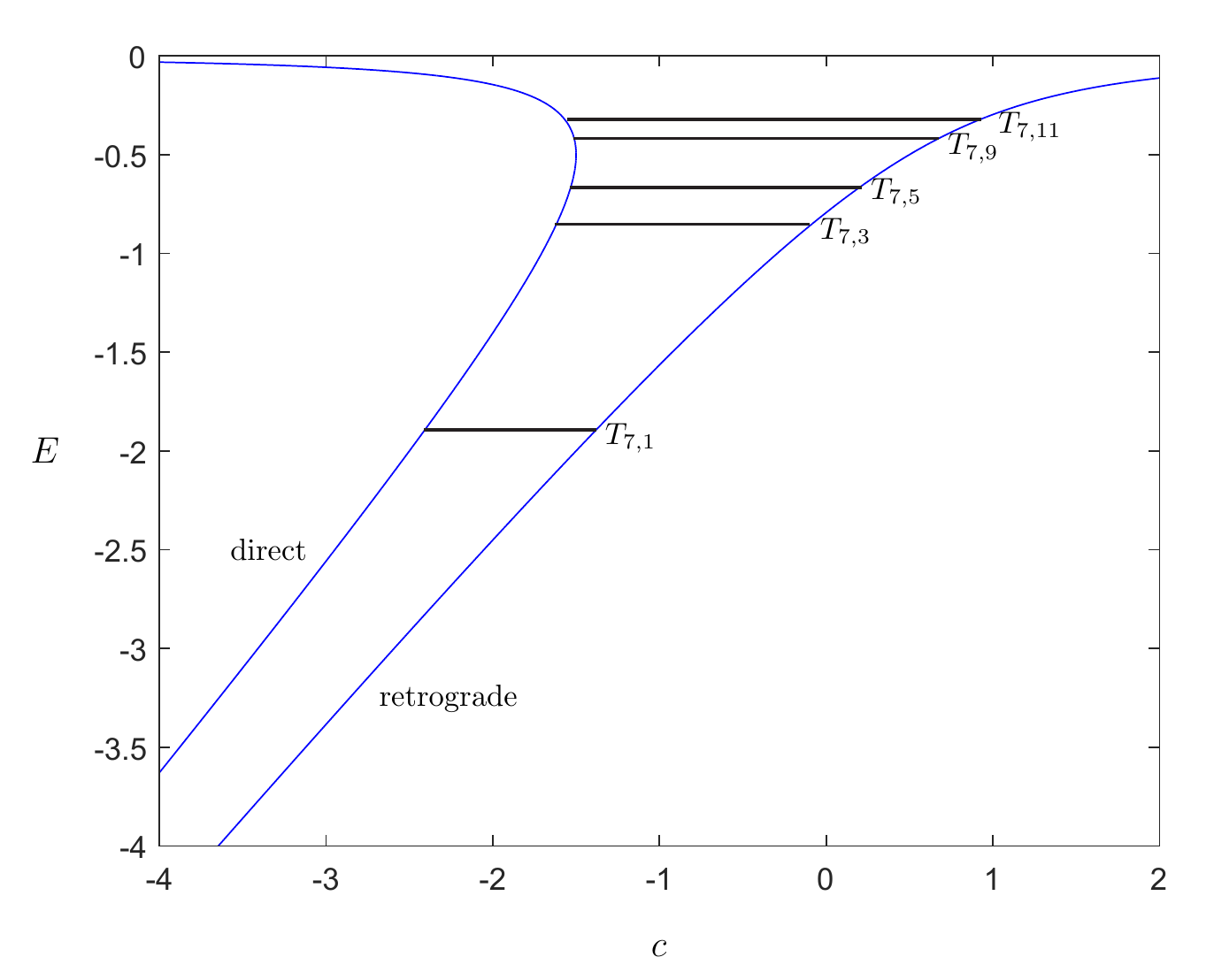}
 \caption{Some torus familes for $k=7$. The blue curves are the graph of the function $2E(c-E)^2 +1 =0$. The left curve is associated with direct circular orbits and the right one is associated with retrograde circular orbits. Observe that along each torus family, the energy level $c$ is increasing.}
 \label{circular}
\end{figure}
We abbreviate by $\Gamma_l$ the left component of the graph of \eqref{equationecc} and by $\Gamma_r$ the right component in the figure.  Each point on the graph represents a circular orbit in the rotating Kepler problem. Note that for every circular orbit the energy $H=c$ and the Kepler energy $E$ cannot be equal.  Indeed, if this is the case, then the left hand side of (\ref{eccentrieq}) becomes zero, while the right hand side is one.  This implies that in view of the relation $c=E+L$, the circular orbits associated with  a point  on $\Gamma_l$ (or $\Gamma_r$) is of negative angular momentum (or positive angular momentum). Consequently, $\Gamma_l$  corresponds to   direct  circular orbits and  $\Gamma_r$   to   retrograde circular orbits.
 
We now fix $c$. If $ c<c_J$, then (\ref{equationecc}) has three solutions $E_1 < E_2 < -1/2 < E_3 < 0$. By the previous argument, we see that on the bounded connected component of the energy hypersurface $H^{-1}(c)$, there exist precisely one retrograde and one direct circular orbits, denoted by $C_{\text{retro}}$ and $C_{\text{direct}}$,  whose Kepler energies are given by $E_{\text{retro}}:=E_1$ and $E_{\text{direct}}:=E_2$, respectively. On the unbounded component, there exists precisely one direct circular orbit of Kepler energy $E_{\text{direct}}^{\text{u}}:=E_3$, which is denoted by $C_{\text{direct}}^\text{u}$. In view of the relation $0 = 2EL^2+1$, the angular momenta of $C_{\text{retro}}$, $C_{\text{direct}}$ and  $C_{\text{direct}}^\text{u}$ are given by $L_{\text{retro}} = 1/\sqrt{-2E_{\text{retro}}}$, $L_{\text{direct}}= -1/\sqrt{-2E_{\text{direct}}}$ and  $L_{\text{direct}}^\text{u} = -1/\sqrt{-2E_{\text{direct}}^\text{u}}$, respectively. On the other hand, since in the inertial system the semi-major axis of a Kepler ellipse equals $1/-2E$, on the bounded component of $H^{-1}(c)$, the radius of $C_{\text{direct}}$ is bigger than that of $C_{\text{retro}}$.

\subsection{Bifurcations of torus families}\label{sec:bif}

To see bifurcation of the torus families, we assume that $E=E_{k,l}$.

\begin{Proposition}\label{probifurcation} {\rm (\cite[Section 6 and Appendix B]{RKP})} The $T_{k,l}$-torus family bifurcates from a $|k-l|$-fold covered direct circular orbit and ends at a $(k+l)$-fold covered retrograde circular orbit.
\end{Proposition}
\begin{proof} We first claim that the period of a direct or retrograde circular orbit of energy $E=E_{k,l}$ is given by
\begin{equation}\label{perioddirect}
\tau_{\text{direct}} =  \frac{ 2 \pi l}{|k-l|}  \quad \text{ or } \quad \tau_{\text{retro}} = \frac{ 2 \pi l}{k+l}, 
\end{equation}
respectively.  Indeed, the Kepler third law implies that in the inertial system direct and retrograde circular orbits of energy $E=E_{k,l}$ have the angular velocities $-k/l$ and $k/l$, respectively, and hence in the rotating frame the direct and retrograde circular orbits have the angular velocities 
\begin{equation}\label{eq:anguvelo}
\dot{\theta}_{\text{direct}}= 1-\frac{k}{l} \quad \text{ and } \quad \dot{\theta}_{\text{retro}}= 1+\frac{k}{l},
\end{equation} 
respectively. The claim is now straightforward.

Recall that the angular momenta of the direct and retrograde circular orbits are given by $-1/\sqrt{-2E}$ and $1/\sqrt{-2E}$, respectively. This shows that the $T_{k,l}$-torus family bifurcates from the (possibly multiply covered) direct circular orbit at 
\begin{equation}\label{energydirect}
c = c_{k,l}^{\text{direct}} = E_{k,l} -1/\sqrt{-2E_{k,l}}
\end{equation}
and ends at the (possibly multiply covered) retrograde circular orbit at 
\begin{equation*}\label{energydirect2}
c = c_{k,l}^{\text{retro}} = E_{k,l}+1/\sqrt{-2E_{k,l}}.
\end{equation*}
Assume that the direct circular orbit is $N$-fold covered. Then we obtain the relation $N \tau_{\text{direct}} = 2\pi l$. By   (\ref{perioddirect}) we find
\begin{equation*}
N = \frac{ 2\pi l}{\tau_{\text{direct}}} = |k-l|.
\end{equation*}
The assertion for the retrograde circular orbit can be proved in a similar way. This completes the proof of the proposition.
\end{proof}

\begin{Remark} \rm Equation \eqref{eq:anguvelo} implies that   the  direct circular orbit rotates clockwise for  $k>l$ and counterclockwise for $k<l$ and the retrograde circular orbit always rotates counterclockwise. 
\end{Remark}

One can describe the bifurcation of the  $T_{k,l}$-torus family by varying the eccentricity as follows. At $e=0$ which corresponds to the $|k-l|$-fold covered direct circular orbit, $T_{k,l}$-type orbits bifurcate. Such orbits will be referred to as   \textit{direct $T_{k,l}$-type orbits}. Note that the winding number of  direct $T_{k,l}$-type  orbits equals $l-k$. As $e$  increases, they  become more eccentric. The eccentricity increases until $e=1$ at which collisions happen. After that event, the eccentricity decreases and we again obtain $T_{k,l}$-type orbits, which will be referred to as  \textit{retrograde $T_{k,l}$-type orbits}. Their winding number equals $k+l$. As $e$ decreases,   retrograde $T_{k,l}$-type orbits become less eccentric and at $e=0$ the $T_{k,l}$-torus family ends at the $(k+l)$-fold covered retrograde circular orbit. We call this $e$-parameter family the \textit{$e$-homotopy}. We denote by $I_{\text{direct}}$ the interval in $e$ for direct orbits along which $e$ increases and  by $I_{\text{retro}}$ the interval (also in $e$) for retrograde orbits along which $e$ decreases. In what follows, we think of those intervals as open, closed or half open intervals depending on the situation. Note that this does not affect our argument.

In view of    (\ref{hamltiodnsRKP})     we have
\begin{equation*}
p_r = \pm \sqrt{ 2 \bigg( c + \frac{1}{|q|} - L \bigg) - \frac{L^2}{|q|^2} }.
\end{equation*}
Equation (\ref{transvec}) then gives rise to
\begin{equation}\label{dhasidlfues}
X_H (z) = \pm \sqrt{ 2 \bigg( c + \frac{1}{|q|} - L \bigg) - \frac{L^2}{|q|^2} } \partial _r + \bigg( \frac{L }{|q|^2}+1 \bigg) \partial_{\theta} + \frac{L^2 - |q|}{|q|^3} \partial_{p_r}.
\end{equation}
Since   by   \eqref{eccentrieq}   a    direct $T_{k,l}$-type orbit has angular  momentum $-\sqrt{ {1-e^2}/{-2E}}$ and  a  retrograde $T_{k,l}$-type orbit has angular momentum $\sqrt{ {1-e^2}/{-2E}}$, this implies that along a retrograde $T_{k,l}$-type orbit angular velocity $\dot{\theta}$ does not vanish. On the other hand, along a direct $T_{k,l}$-type orbit, we have $\dot{\theta}<0$ for $r<\sqrt{-L}$ and $\dot{\theta}>0$ for $r>\sqrt{-L}$.

\subsection{Symmetries}\label{sdfsdsssasasd}
 
\begin{Lemma}\label{rotsyme} {\rm (\cite[Section 8.2.1]{book})} The trajectory of each $T_{k,l}$-type orbit is invariant under the rotation by the angle $2\pi j l /k$, $j=1,2,\cdots, k-1$.
\end{Lemma}
\begin{proof}  It follows immediately from the fact that  $kT=2 \pi l$  and the observation
\begin{equation} \label{goodcompute}
\alpha(t+T) = \exp( it + iT)\gamma(t+T)  = \exp( 2 \pi i  l/k) \exp(it) \gamma(t) = \exp(  2 \pi i  l/k)  \alpha(t) .
\end{equation}
\end{proof}

On top of the rotational symmetry, trajectories of torus type orbits admit another symmetry.
 
\begin{Lemma}\label{lemmareflds} Assume that the perihelion of $\gamma$ is the starting point $\gamma(0)$ and   its argument is $\theta = \theta_0$.     Then the trajectory of a $T_{k,l}$-type orbit is invariant under the reflection with respect to the line $y=\tan ( j \pi    /k + \theta_0) x$, $j=0,1,\cdots, k-1$. By convention, we regard   the line $y=\tan (\pi/2) x $ as  the $y$-axis.
\end{Lemma}
\begin{proof} Without loss of generality, we may assume that $\theta_0=0$. We first claim that the trajectory is symmetric under the reflection with respect to the $x$-axis. Indeed, we observe that
\begin{equation*}
\overline{\alpha(t)}  = \overline{\exp( it ) \gamma(t) }  = \exp( -it ) \overline{\gamma(t) } = \exp( - it) \gamma( -t) = \alpha(-t).
\end{equation*}
In other words, the reflected orbit $\overline{\alpha}$ is the inverse of the original orbit $\alpha$. This proves the claim. 

Recall that the reflection matrix $\text{Ref}(\theta)$ with respect to the line $y=\tan(\theta) x$ is given by
\begin{equation*}
\text{Ref}(\theta) = \begin{pmatrix} \cos 2\theta & \sin 2 \theta \\ \sin 2 \theta & - \cos 2 \theta \end{pmatrix}.
\end{equation*}
One can easily check that the reflection and the rotation matrices satisfy the relation
\begin{equation*}
\text{Rot}(\theta_1 ) \text{Ref}(\theta_2) = \text{Ref} \bigg( \frac{1}{2} \theta_1 + \theta_2 \bigg).
\end{equation*}
Then the assertion  follows from  Lemma \ref{rotsyme} and the claim by setting $\theta_1 = - 2  j \pi    /k$ and $\theta_2 = j \pi   /k$.  This completes the proof of the lemma.
\end{proof}

\subsection{Disasters}\label{disatser} In this section, we examine  disasters which happen during the $e$-homotopy.

\subsubsection{Collisions} Fix any $T_{k,l}$-torus family. It is obvious that collisions happen precisely at $e=1$. Then by \cite[Lemma 2]{invariant}   the event $(I_0)$ happens only  at $e=1$.

\subsubsection{Triple points}

Suppose that $q$ is an intersection point on a $T_{k,l}$-type orbit of energy $c$. Since the angular momentum $L$ is constant along the orbit,  equation \eqref{dhasidlfues}  shows that the intersection point $q$ is double. We conclude that the event $(III)$ does not happen in the rotating Kepler problem.

\subsubsection{Hitting the boundary of the Hill's region}\label{sdfsdgsesd33} When a $T_{k,l}$-type orbit touches the boundary of the Hill's region, we must have $\dot{\theta}=0$ since the orbit has a cusp at the touching point. In view of \eqref{dhasidlfues}, we then have $L=-r_{\max}^2$ for the case $k>l$ and $L=-r_{\min}^2$ for the case $k<l$. In any case, angular momentum $L$ is negative. Since retrograde $T_{k,l}$-type orbits always have positive angular momentum, we see that    retrograde $T_{k,l}$-type orbits do not touch the boundary of the Hill's region. 

Assume the case $k>l$. As mentioned above,   a direct $T_{k,l}$-type orbit touches the boundary if and only if 
$$
L^2 = r_{\max}^4 = \bigg( \frac{1+e}{2E_{k,l}} \bigg)^4.
$$
Equation \eqref{eccentrieq}, however, implies that
$$
L^2 = \frac{e^2-1}{2E_{k,l}}.
$$
Putting together these two equalities shows that in the case $k>l$, a direct $T_{k,l}$-torus type orbit touches the boundary if and only if its eccentricity $e$ solves the equation
$$
   8(1-e)E_{k,l}^3 + (1+e)^3 =   0.
$$
In view of  $E_{k,l}<0$ the uniqueness of $e$ solving this equation is obvious, which implies that during the   $e$-homotopy for $k>l$, the event $(I_{\infty})$ happens precisely once at $ e = e_{k,l}^{\infty} \in I_{\text{direct}}$, see Figure \ref{homotopyofrkp}.

\;\;

For the case $k<l$, in a similary way we have that 
$$
\frac{e^2-1}{2E_{k,l}}= L^2 = r_{\min}^4 = \bigg( \frac{1-e}{2E_{k,l}} \bigg)^4
$$
which implies that 
$$ 
 8(1+e)E_{k,l}^3 + (1-e)^3 =   0  . 
$$
Consequently,  during the   $e$-homotopy for $k<l$, the event $(I_{-\infty})$ happens precisely once at $ e = e_{k,l}^{-\infty} \in I_{\text{direct}}$, see Figure \ref{homotopyofrkp2}.

\subsubsection{Self-tangencies} \label{sec:tanself}

Suppose that a  self-tangency happens at $ q(t_0) = q(t_1)$ for $ t_0 \neq t_1$. Since the tangent vectors are parallel at the self-tangency,  in view of (\ref{dhasidlfues}) we see that a self-tangency can happen only on the circle $r=r_{\text{inv}}:=\sqrt{-L}$. From this we obtain the following two necessary conditions for the existence of self-tangencies: 

\begin{enumerate}[label=(\roman*)]
\item the angular momentum is negative and hence the orbit is direct; \vspace{1mm}

\item $r_{\text{inv}} \leq r_{\max}$ for the case $k>l$ and $r_{\text{inv}} \geq r_{\min}$ for the case $k<l$, where $r_{\max}$ and $r_{\min}$ are defined as in (\ref{eqrmnsridncx}).

\end{enumerate}
It follows from (i) that the event $(II^+)$  does not happen on $I_{\text{retro}}$. 

Assume that the self-tangency at $q(t_0) = q(t_1)$ is direct. Since this point is also an intersection point, we have $r(t_0) = r(t_1)$ and $\theta(t_0) = \theta(t_1)$. Moreover, since $r(t_0) = r(t_1)$, it follows from  (\ref{dhasidlfues})  that $\dot{r}(t_0) = \dot{r}(t_1)$ and $\dot{\theta}(t_0)  = \dot{\theta}(t_1)=0$. Therefore, in view of the uniqueness theorem for second order ordinary differential equations, it follows that $t_0 = t_1$ which contradicts   $t_0 \neq t_1$. Consequently, direct self-tangencies do not happen in the rotating Kepler problem.

\;\;
 
Finally, we now consider an inverse self-tangency on $I_{\text{direct}}$. We assume that at $q(t_0)=q(t_1)$ an inverse self-tangency happens. Abbreviate $r_{\text{inv}} = |q(t_0)|$. Since the orbit is direct, we have $L= -\sqrt{(1-e^2)/-2E_{k,l}}$ and  we then obtain $r_{\text{inv}} =  \sqrt[4]{ (1-e^2)/-2E_{k,l} }$.

We first assume that $k>l$ and  compute that
\begin{eqnarray*}
 \sqrt[4]{ \frac{1-e^2}{ -2E_{k,l}}} < r_{\max} \;\; &\Leftrightarrow& \;\; \sqrt[4]{1-e}< \sqrt[4]{\frac{1+e}{-2E_{k,l}}}^3 \\
&\Leftrightarrow& \;\; 1-e< \bigg(\frac{1+e}{-2E_{k,l}}\bigg)^3 \\
&\Leftrightarrow& \;\; 8 (1-e)E_{k,l}^3 + (1+e)^3 > 0.
\end{eqnarray*}
Recall from   Section \ref{sdfsdgsesd33} that $e=e_{k,l}^{\infty}$ is a unique solution of $8 (1-e)E_{k,l}^3 + (1+e)^3 =0$. Since $8 (1-e)E_{k,l}^3 + (1+e)^3>0$ for $e>e_{k,l}^{\infty}$, this implies that if the event $(II^+)$ is achieved at $e=e_0$, then we have $ e_{k,l}^{\infty} <e_0 $.   One can easily obtain the same result for the event $(I_{-\infty})$  the case $k<l$.

We now assume that  a direct $T_{k,l}$-type orbit $\alpha$ has an inverse self-tangency. By the previous discussion it follows that $\alpha$ has $k$ exterior loops, where the number $k$ of the exterior loops follows from the rotational symmetry. Note that there is an exterior loop for each aphelion. We claim that all the inverse self-tangencies   must lie on the $k$ exterior loops.   To this end, abbreviate by $\alpha_{\text{int}}$ the part of $\alpha$ with the $k$ exterior loops removed.  In view of equation \eqref{dhasidlfues}   we have $ r<r_{\text{inv}}$  along $\alpha_{\text{int}}$. By means of the argument given in the last paragraph of Section \ref{sec:bif}, this implies that   along  $\alpha_{\text{int}}$   angular velocity $\dot{\theta}$ does not vanish. Since inverse self-tangencies happen only at points with $\dot{\theta}=0$, this proves the claim. Consequently, inverse self-tangencies along a direct $T_{k,l}$-type orbit must lie on the intersection of exterior loops and the circle $r=r_{\text{inv}}$.

\vspace{2mm}

Recall that the $e$-homotopy of a $T_{k,l}$-torus family is defined on $I_{\text{direct}} \cup I_{\text{retro}}$. We divide the interval $I_{\text{direct}}$ into two subintervals (in $e$): $I_{\text{direct}}^1$ from $0$ to $e_{k,l}^{\infty}$ (or $e_{k,l}^{-\infty}$) and $I_{\text{direct}}^2$ from $e_{k,l}^{\infty}$ (or $e_{k,l}^{-\infty}$) to $1$. All the results obtained so far show the following.

\begin{Proposition} \label{corstar} \

\begin{enumerate}[label=(\roman*)]

\item If $k>l$, then the $J^+$-invariant does not change on each interval $I_{\text{direct}}$ and $I_{\text{retro}}$ and the $\mathcal{J}_1$ and $\mathcal{J}_2$ invariants do not change along the $e$-homotopy, see Figure \ref{homotopyofrkp}. \vspace{1mm}

\item If $k<l$, then the quantity $J^+$ are invariant   on   $I_{\text{direct}}^1$, $I_{\text{direct}}^2$ and $I_{\text{retro}}$  and  $\mathcal{J}_1$ and $\mathcal{J}_2$ are invariant   on   $I_{\text{direct}}^1$ and $I_{\text{direct}}^2 \cup I_{\text{retro}}$, see Figure \ref{homotopyofrkp2}. 
\end{enumerate}

\end{Proposition}
\begin{proof} The only  issue  is that we do not know that there exist at most finitely many disasters $(II^+)$ on $I_{\text{direct}}^2$. This can be solved as follows. We would define the $J^+$-invariant for immersions possibly having inverse self-tangencies (but no direct self-tangencies). Given such an immersion $K$ we slightly perturb it to a generic immersion $\widetilde{K}$. We then define $ J^+(K):= J^+(\widetilde{K})$. This definition is independent of the choice of a perturbation. Indeed, for a (small) perturbation we have only two cases illustrated in Figure \ref{invtan}. 
\begin{figure}[h]
  \centering
  \includegraphics[width=0.5\linewidth]{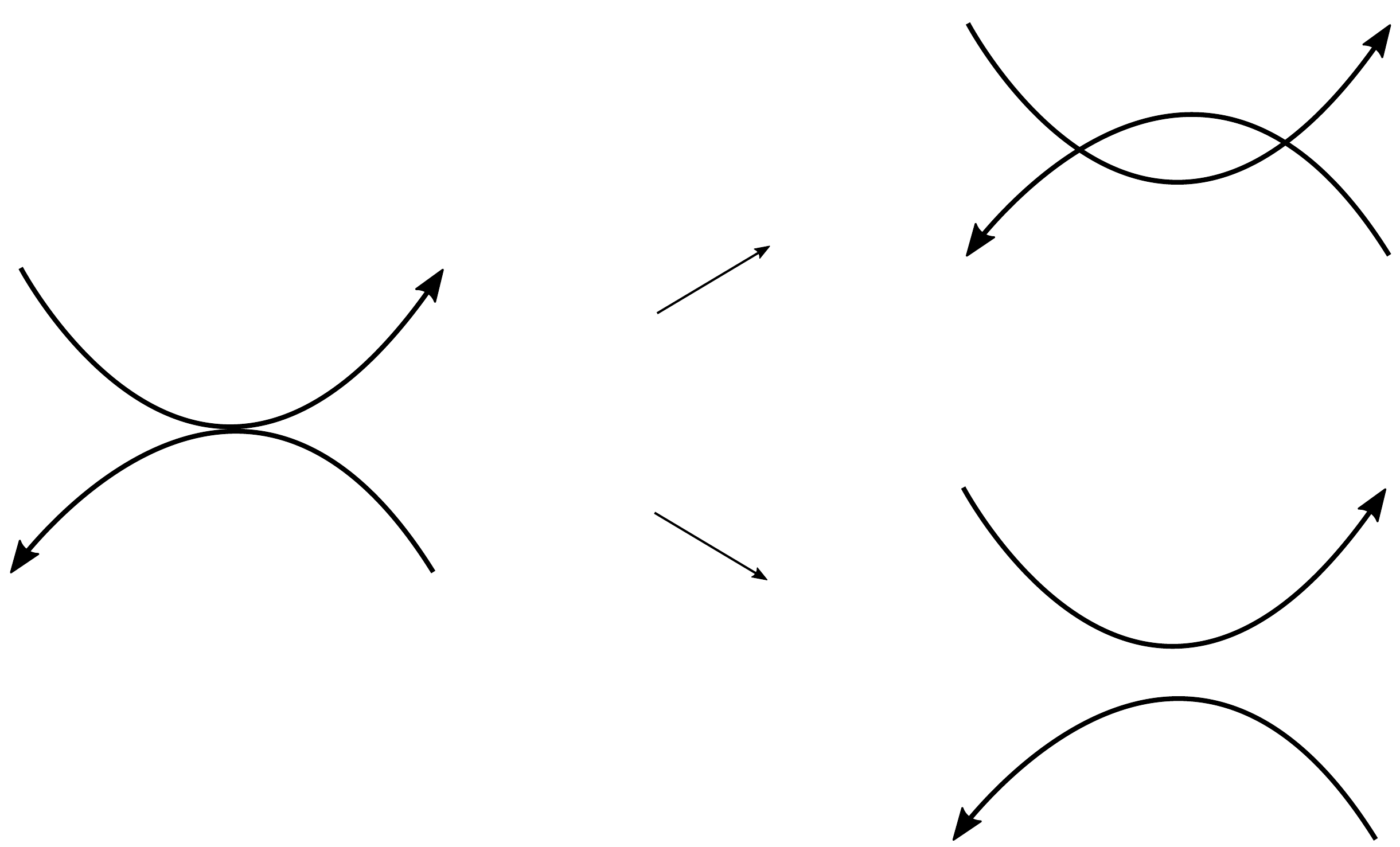}
    \caption{Possible perturbations of an inverse self-tangency}
 \label{invtan}
\end{figure}
Observe that the two perturbations in the figure can be obtained from each other via the crossing through an inverse self-tangency. Since the $J^+$-invariant does not change under  such an event,  this shows the well-definedness of $J^+(K)$. 

If an inverse self-tangency happens along a direct $T_{k,l}$-type orbit $\gamma$ with $e \in I_{\text{direct}}^2$, then there exists $\epsilon>0$ such that $e\pm \epsilon \in I_{\text{direct}}^2$ and the corresponding $T_{k,l}$-type orbits are generic immersions. Moreover, they are perturbations of $\gamma$, where each of them can be obtained from the other via crossings through inverse self-tangencies. In particular, their $J^+$-invariants coincide with each other and they determine $J^+(\gamma)$. Consequently, even though we do not know that the $T_{k,l}$-torus family is a generic homotopy or  Stark-Zeeman homotopy, we can compute its invariants by perturbing it slightly to obtain a generic homotopy or Stark-Zeeman homotopy. This completes the proof of the proposition.
\end{proof}

\begin{Remark} \rm As numerical experiments support, it is conceivable that during each $T_{k,l}$-torus family the disaster $(II^+)$ happens at most once.
\end{Remark}

\begin{figure}[t]
\begin{subfigure}{0.24\textwidth}
  \centering
  \includegraphics[width=0.95\linewidth]{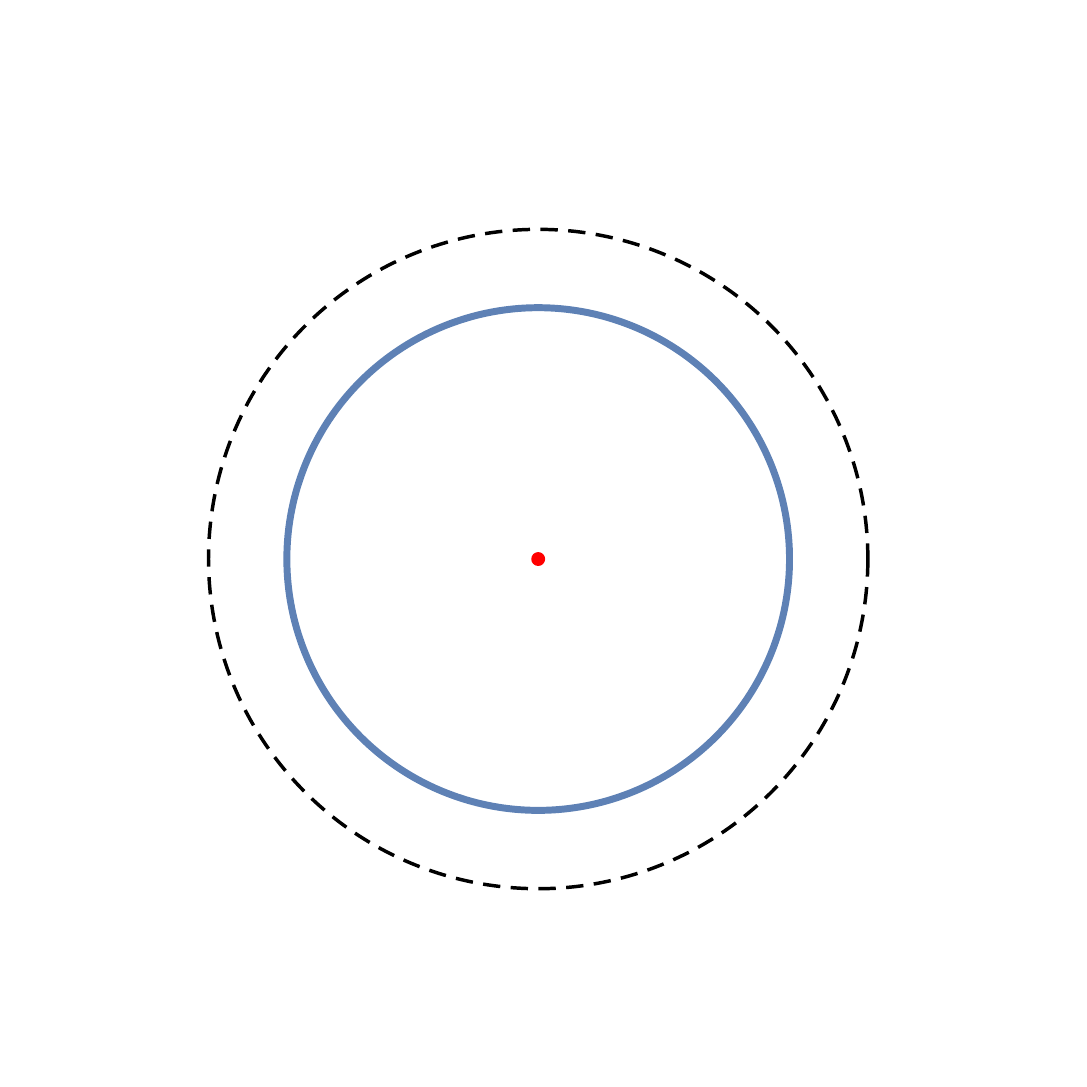}
  \caption{ A $5-2=3$-fold direct circular orbit}
  \label{  }
\end{subfigure}
\begin{subfigure}{0.24\textwidth}
  \centering
  \includegraphics[width=0.95\linewidth]{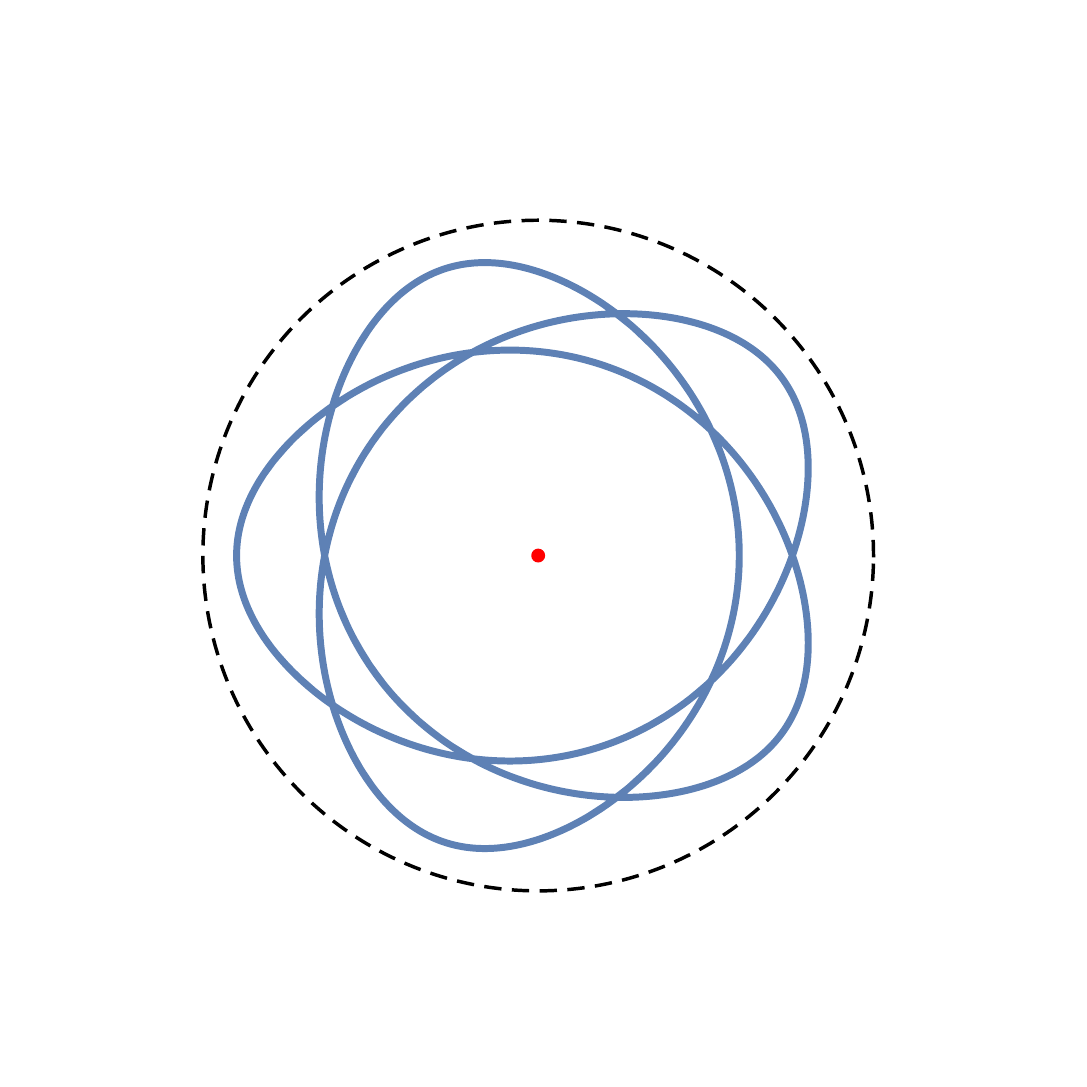}
  \caption{direct, $e=0.2$}
 \label{   }
\end{subfigure}
\begin{subfigure}{0.24\textwidth}
  \centering
  \includegraphics[width=0.95\linewidth]{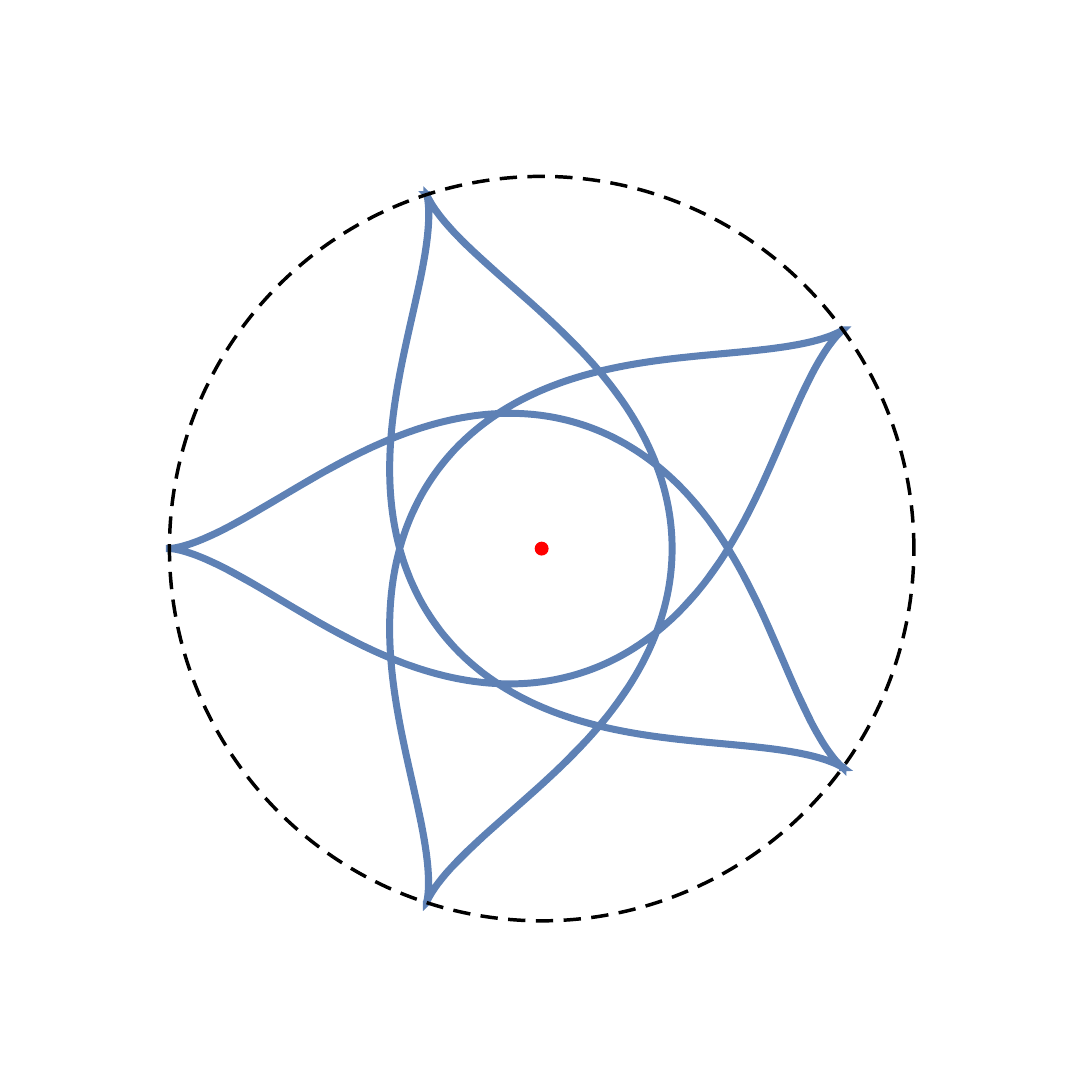}
  \caption{direct, $e=e_{5,2}^{\infty} \approx 0.481$}
 \label{   }
\end{subfigure}
\begin{subfigure}{0.24\textwidth}
  \centering
  \includegraphics[width=0.95\linewidth]{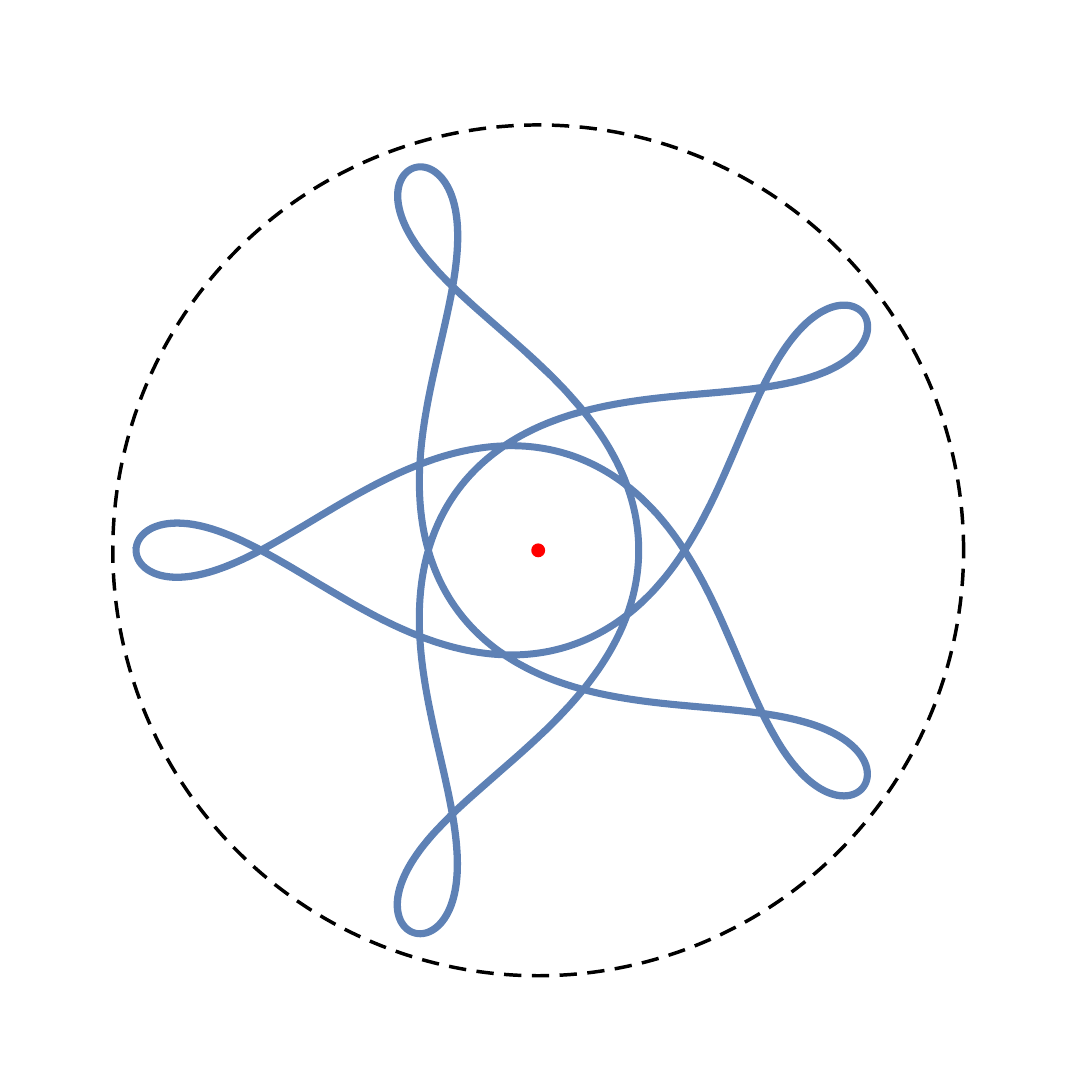}
  \caption{direct, $e=0.6$}
 \label{   }
\end{subfigure}
\begin{subfigure}{0.24\textwidth}
  \centering
  \includegraphics[width=0.95\linewidth]{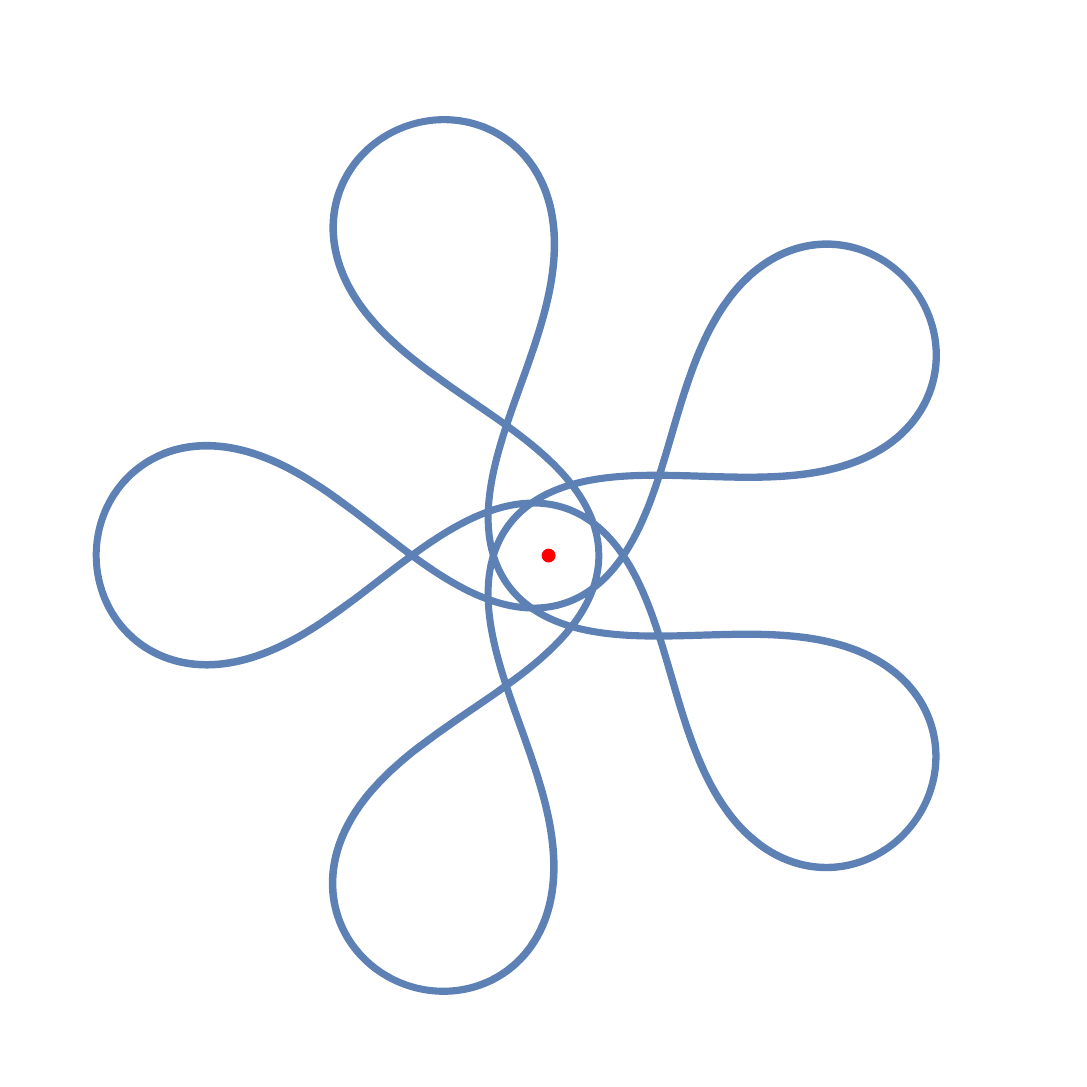}
  \caption{direct, $e=0.9$}
 \label{   }
\end{subfigure}
\begin{subfigure}{0.24\textwidth}
  \centering
  \includegraphics[width=0.95\linewidth]{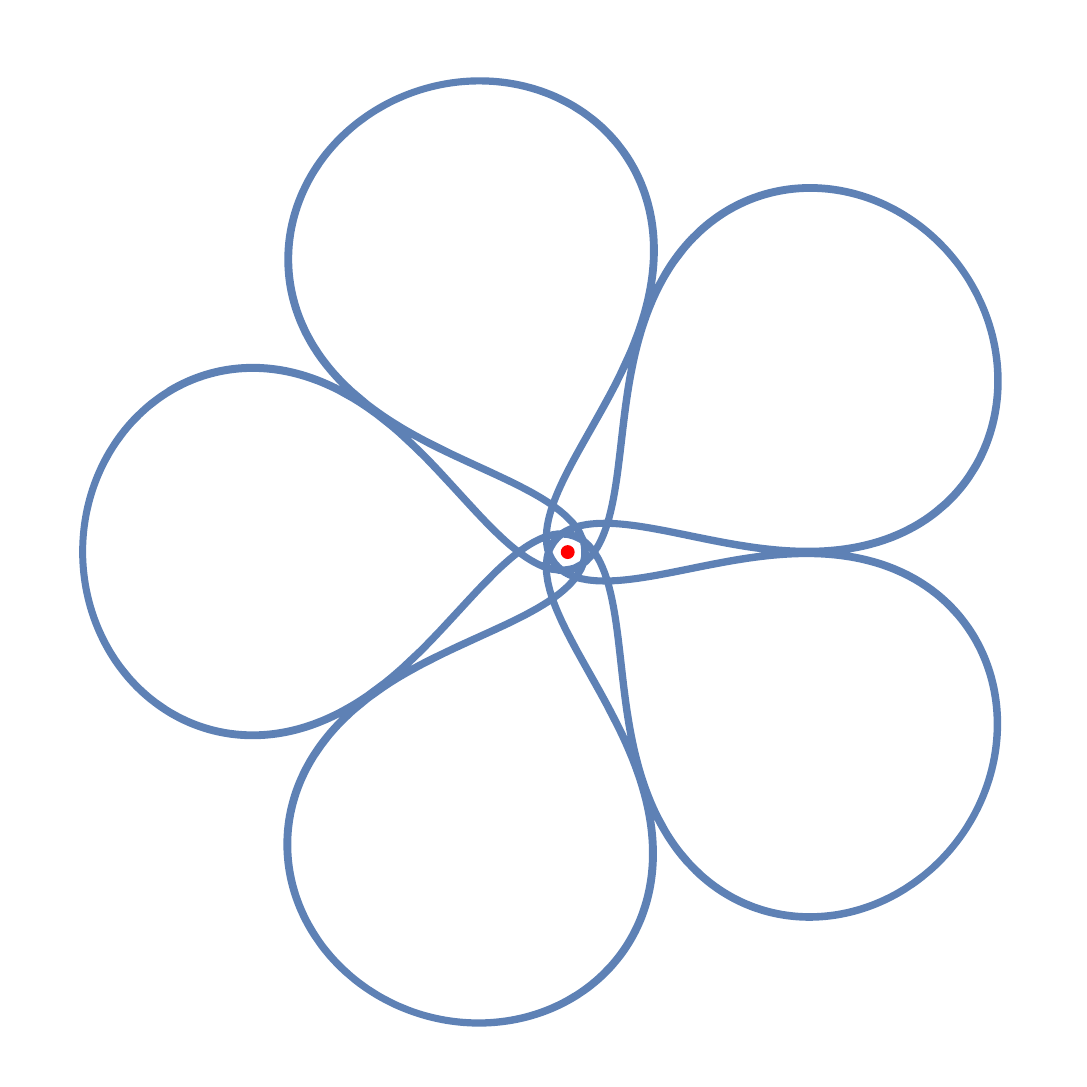}
  \caption{direct, $e  \approx 0.93$}
 \label{   }
\end{subfigure}
\begin{subfigure}{0.24\textwidth}
  \centering
  \includegraphics[width=0.95\linewidth]{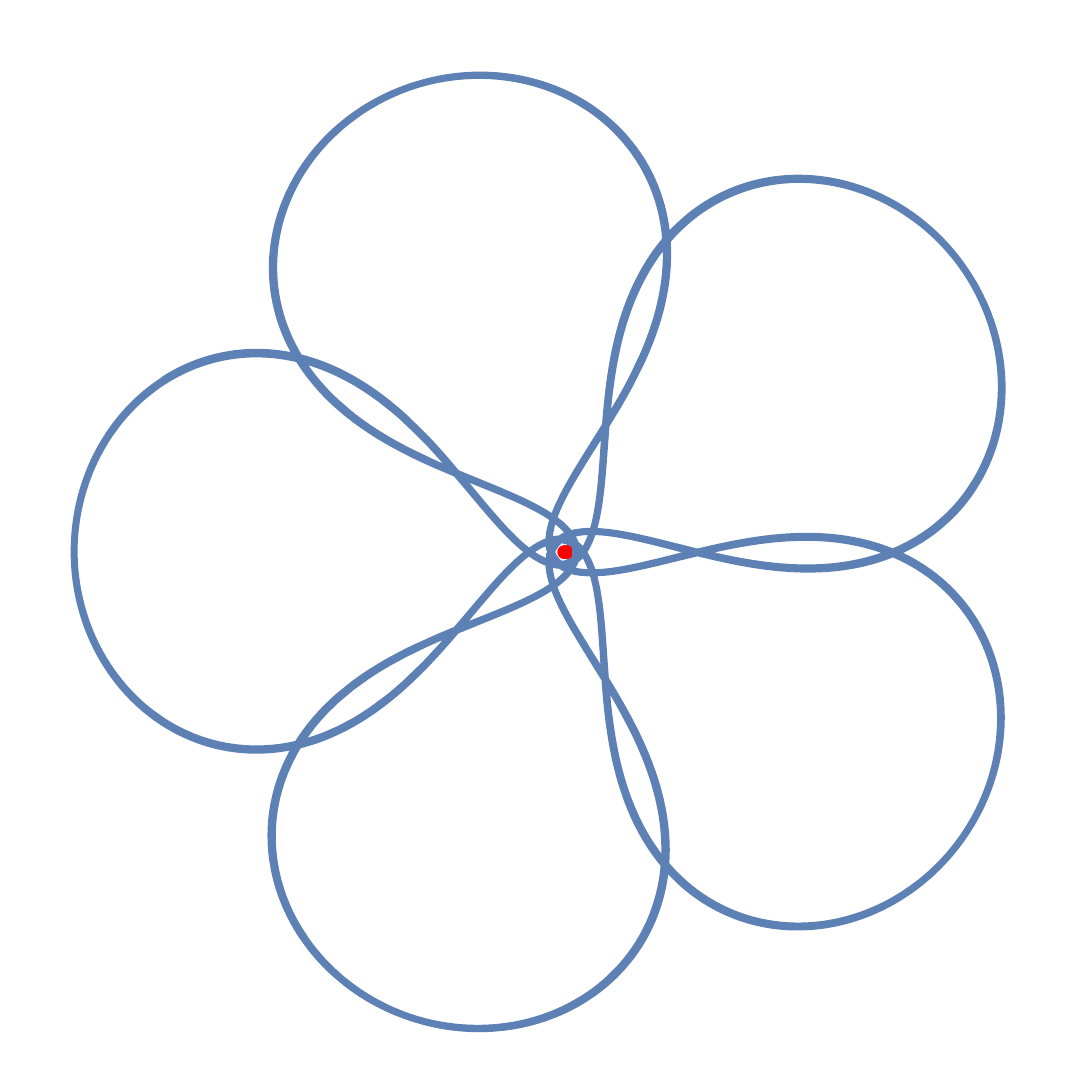}
  \caption{direct, $e= 0.95$}
 \label{   }
\end{subfigure}
\begin{subfigure}{0.24\textwidth}
  \centering
  \includegraphics[width=0.95\linewidth]{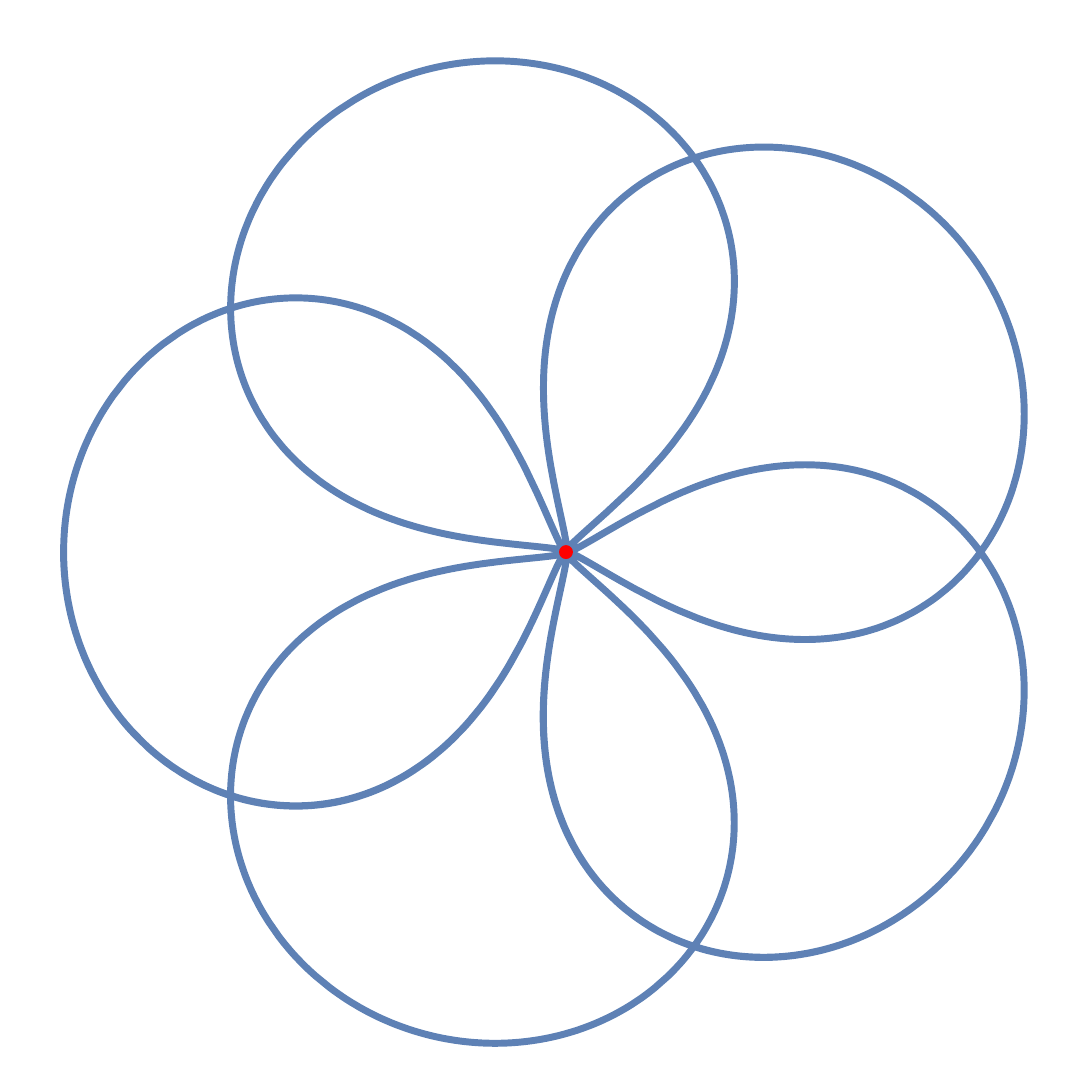}
  \caption{A collision orbit}
 \label{   }
\end{subfigure}
\begin{subfigure}{0.24\textwidth}
  \centering
  \includegraphics[width=0.95\linewidth]{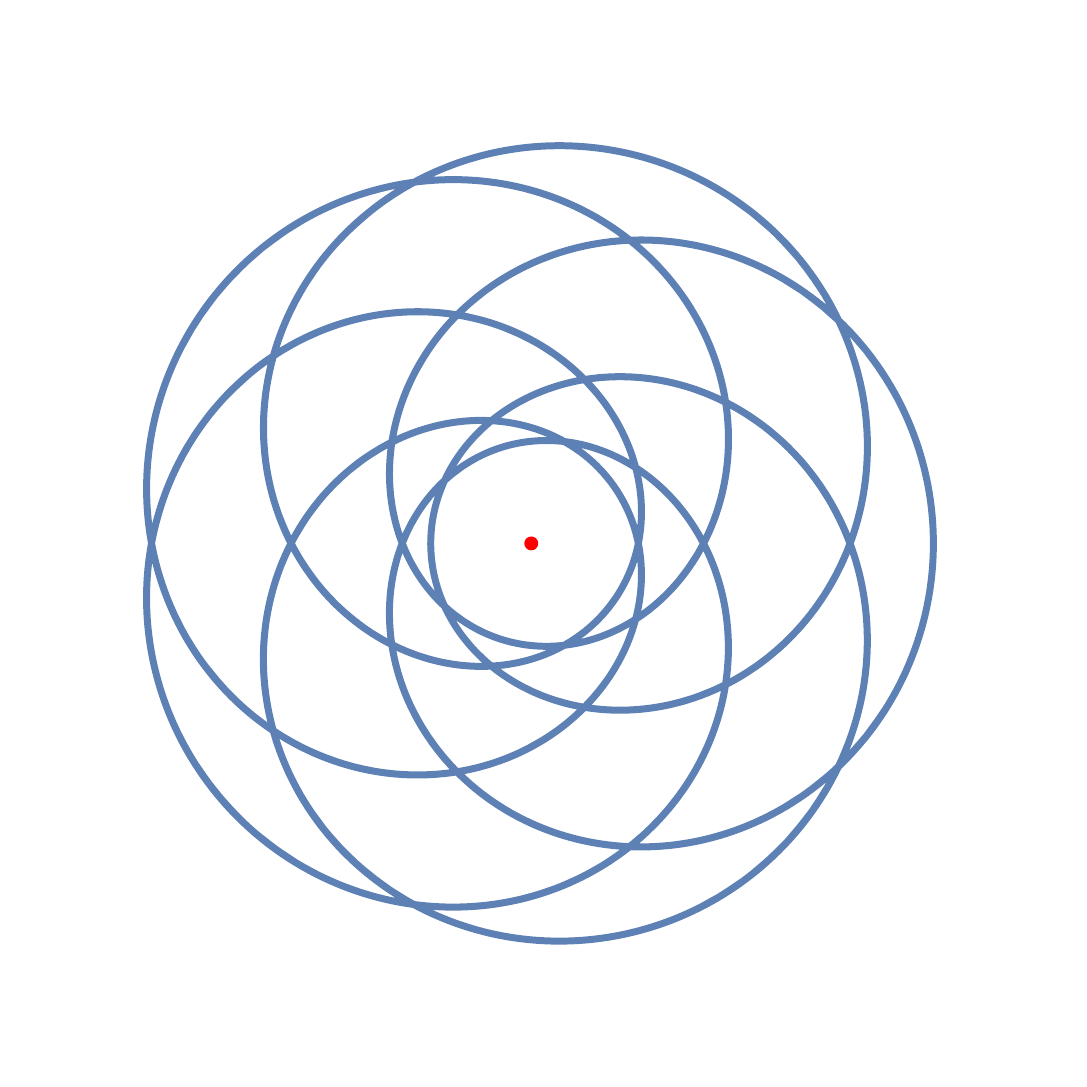}
  \caption{retrograde, $e=  0.6$}
 \label{   }
\end{subfigure}
\begin{subfigure}{0.24\textwidth}
  \centering
  \includegraphics[width=0.95\linewidth]{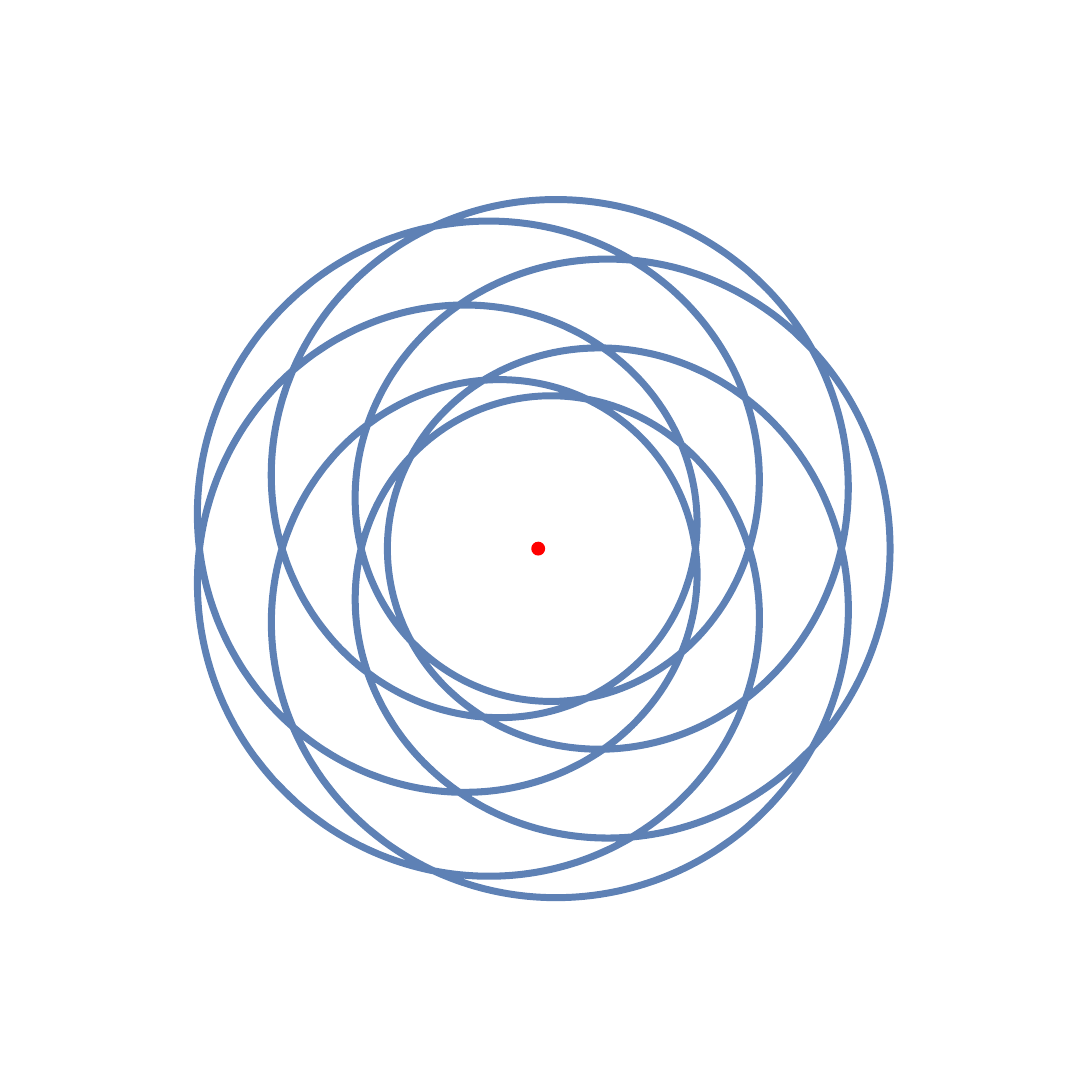}
  \caption{retrograde, $e=0.4$}
 \label{   }
\end{subfigure}
\begin{subfigure}{0.24\textwidth}
  \centering
  \includegraphics[width=0.95\linewidth]{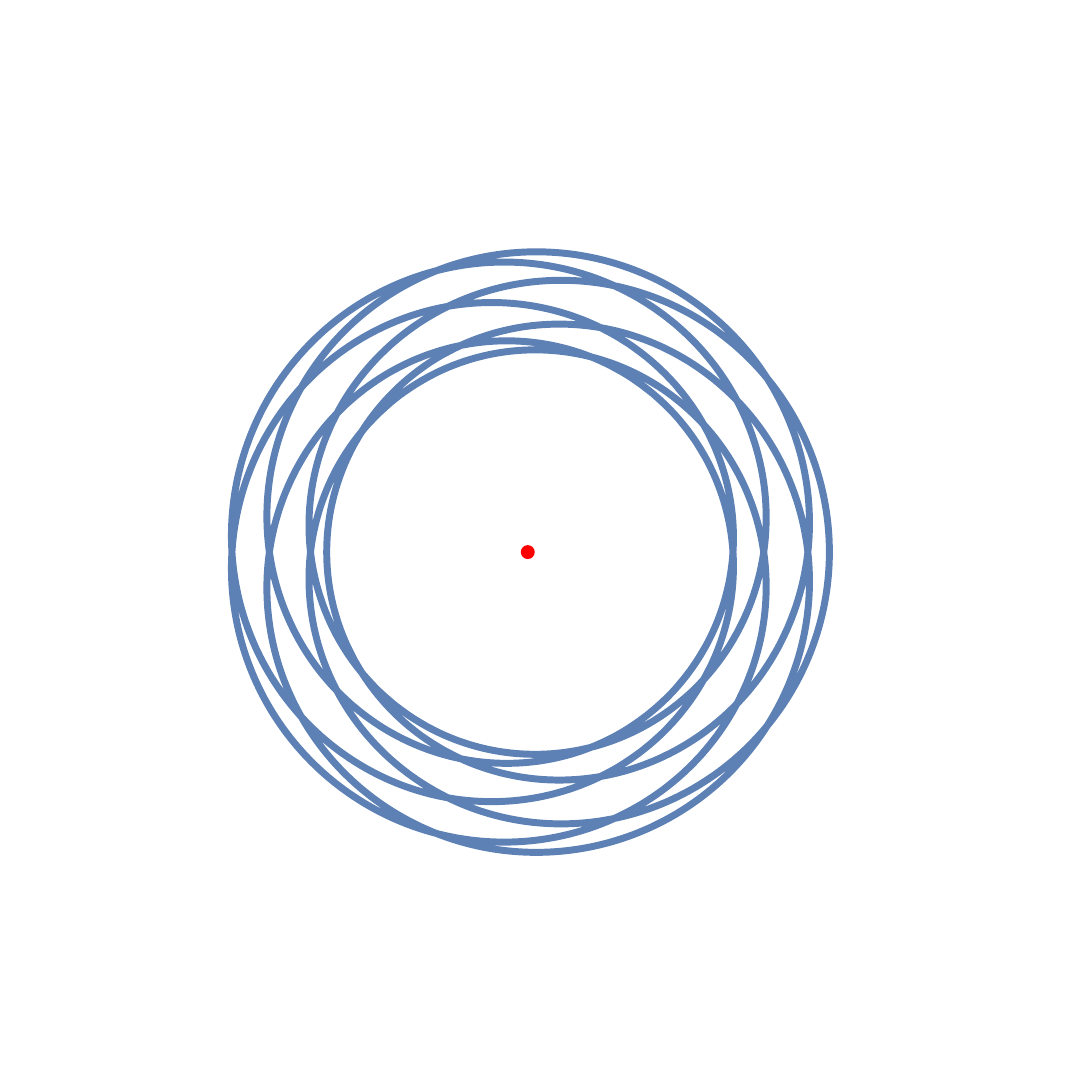}
  \caption{retrograde, $e=  0.2$}
 \label{   }
\end{subfigure}
\begin{subfigure}{0.24\textwidth}
  \centering
  \includegraphics[width=0.95\linewidth]{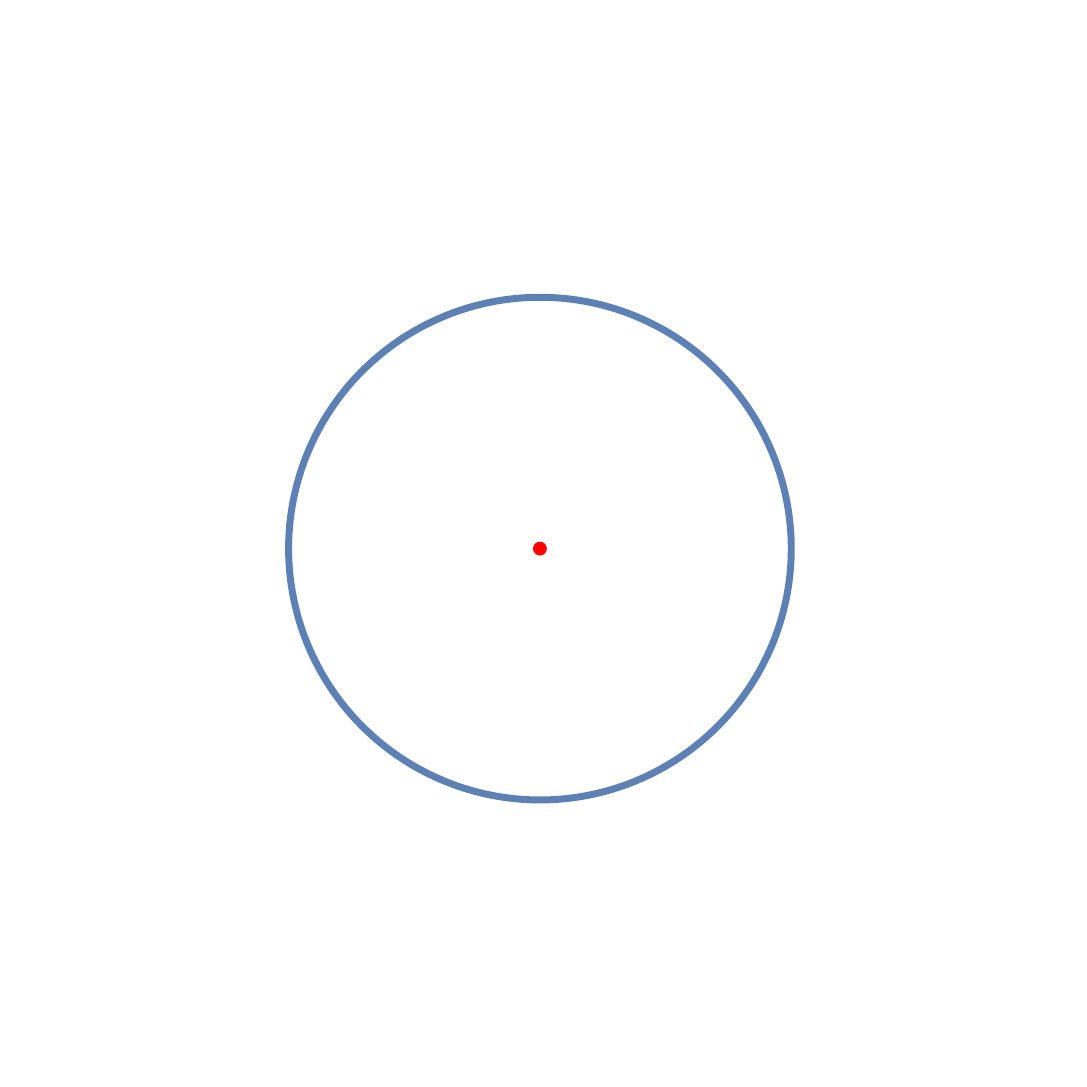}
  \caption{A $5+2=7$-fold retrograde circular orbit}
 \label{   }
\end{subfigure}
\caption{ The $e$-homotopy of the $T_{5,2}$-torus family.  The dashed circles are the boundaries of the bounded components of the Hill's regions. In (e)-(l) no boundaries are indicated since the associated energies are bigger than the critical Jacobi energy $c_J  = -3/2$.    In (c) the direct $T_{5,2}$-type orbit has cusps at the boundary of the Hill's region and then exterior loops appear in (d). Inverse self-tangencies are illustrated in (f). In (h) cusps at the origin happen and loops around the origin appear in (i). }
\label{homotopyofrkp}
\end{figure}

\begin{figure}[t]
\begin{subfigure}{0.24\textwidth}
  \centering
  \includegraphics[width=0.95\linewidth]{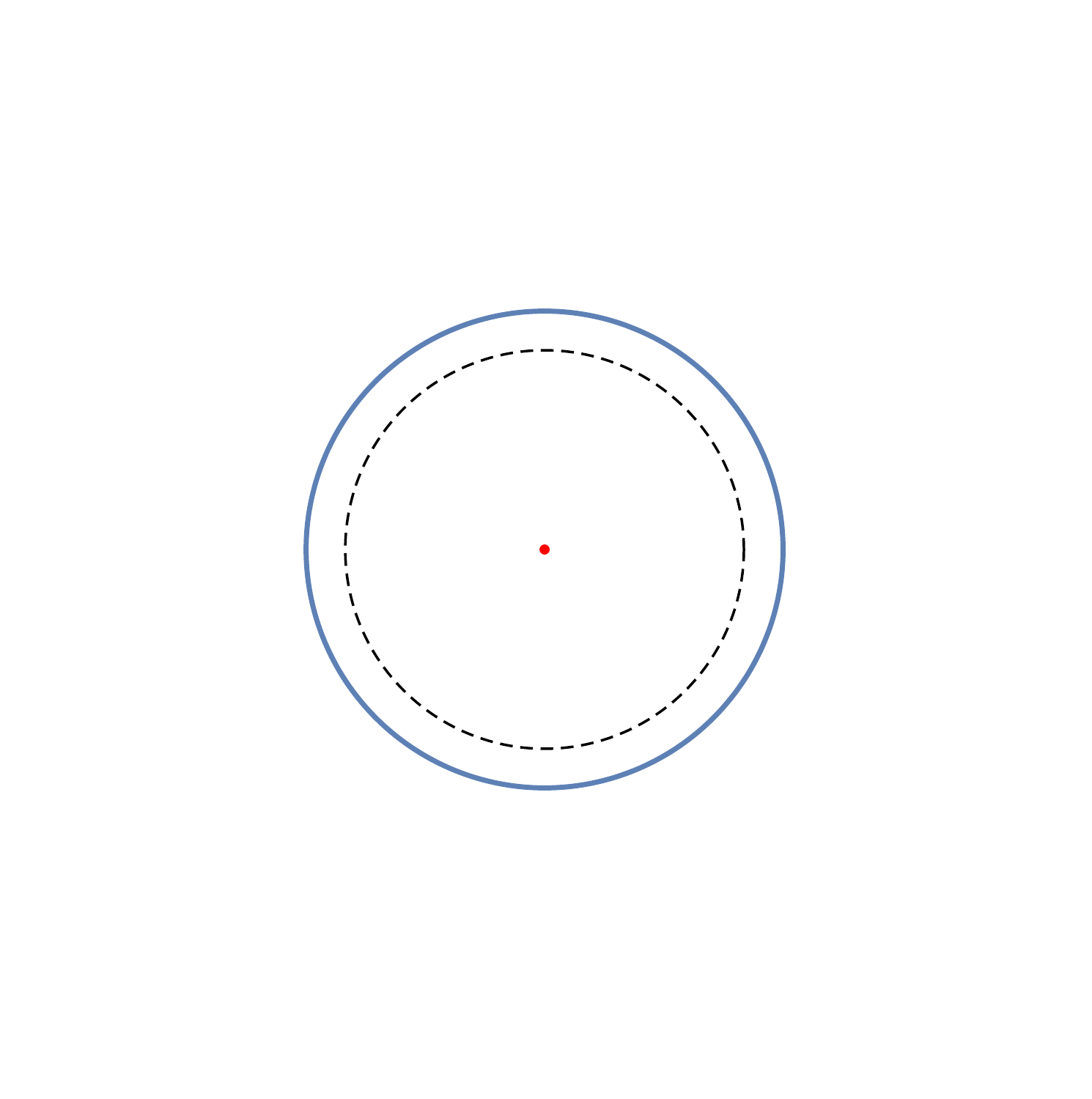}
  \caption{ A $5-3=2$-fold direct circular orbit}
  \label{  }
\end{subfigure}
\begin{subfigure}{0.24\textwidth}
  \centering
  \includegraphics[width=0.95\linewidth]{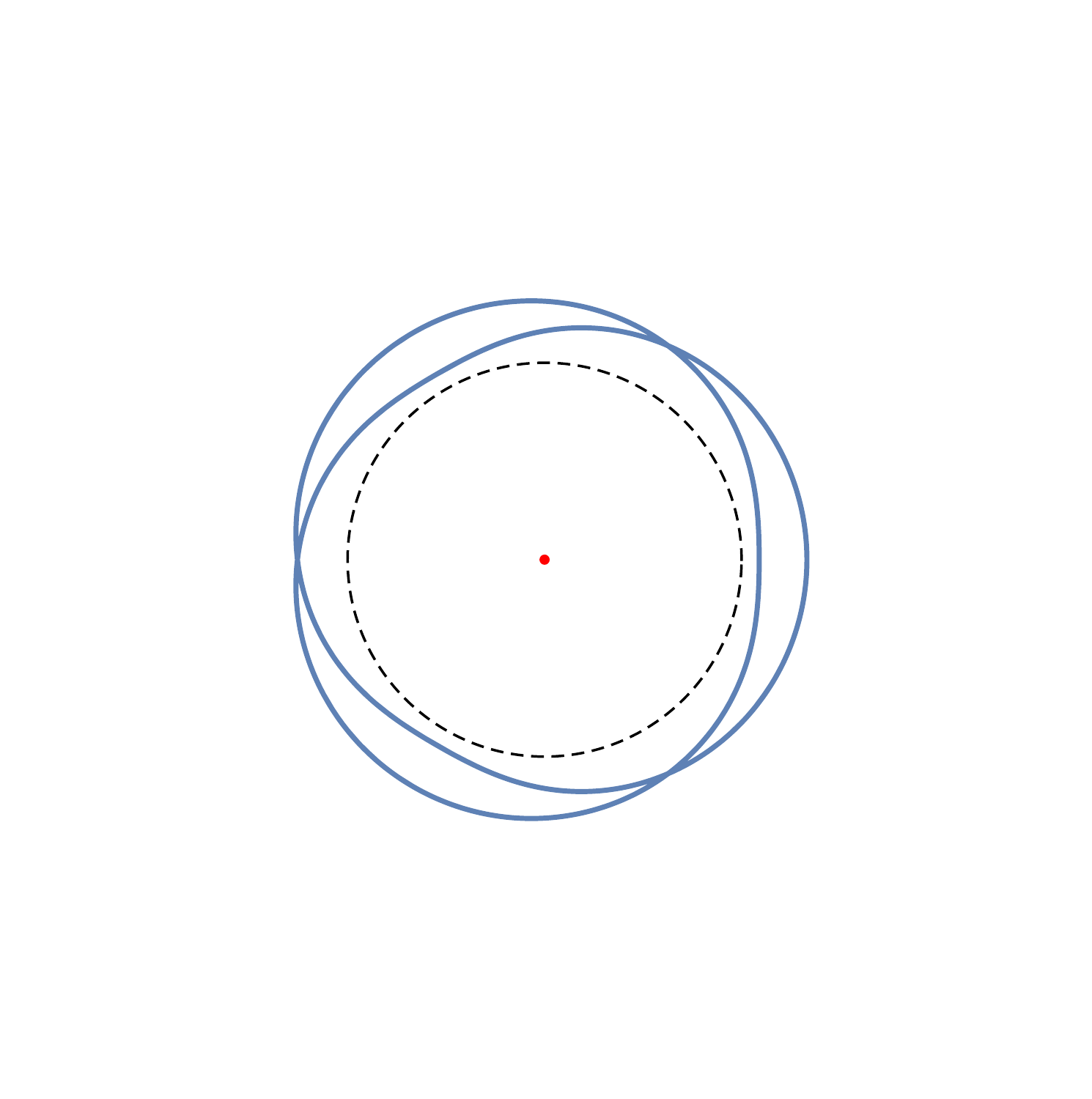}
  \caption{direct, $e=0.1$}
 \label{   }
\end{subfigure}
\begin{subfigure}{0.243\textwidth}
  \centering
  \includegraphics[width=0.95\linewidth]{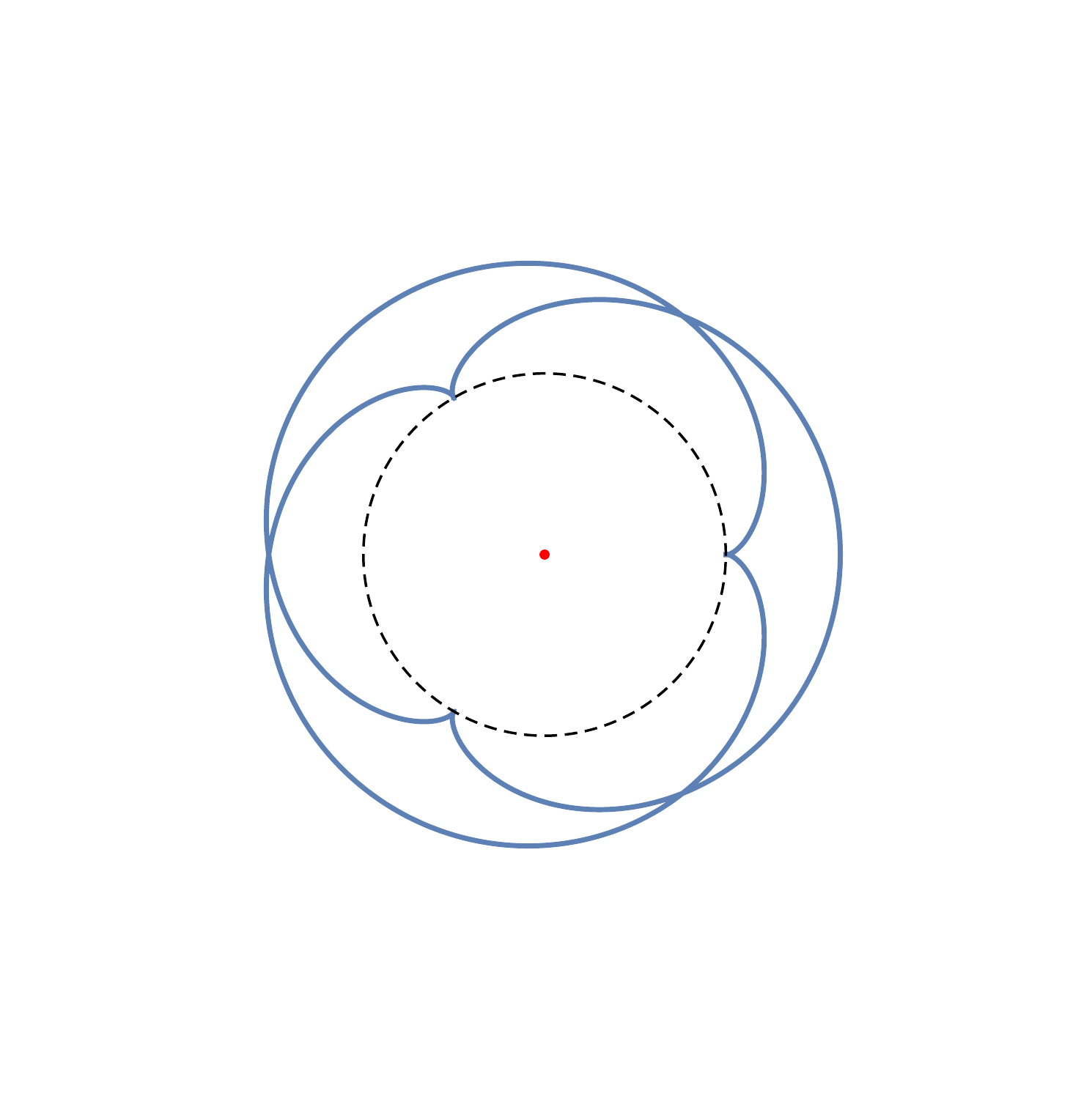}
  \caption{direct, $e=e_{3,5}^{-\infty} \approx 0.24$}
 \label{   }
\end{subfigure}
\begin{subfigure}{0.237\textwidth}
  \centering
  \includegraphics[width=0.95\linewidth]{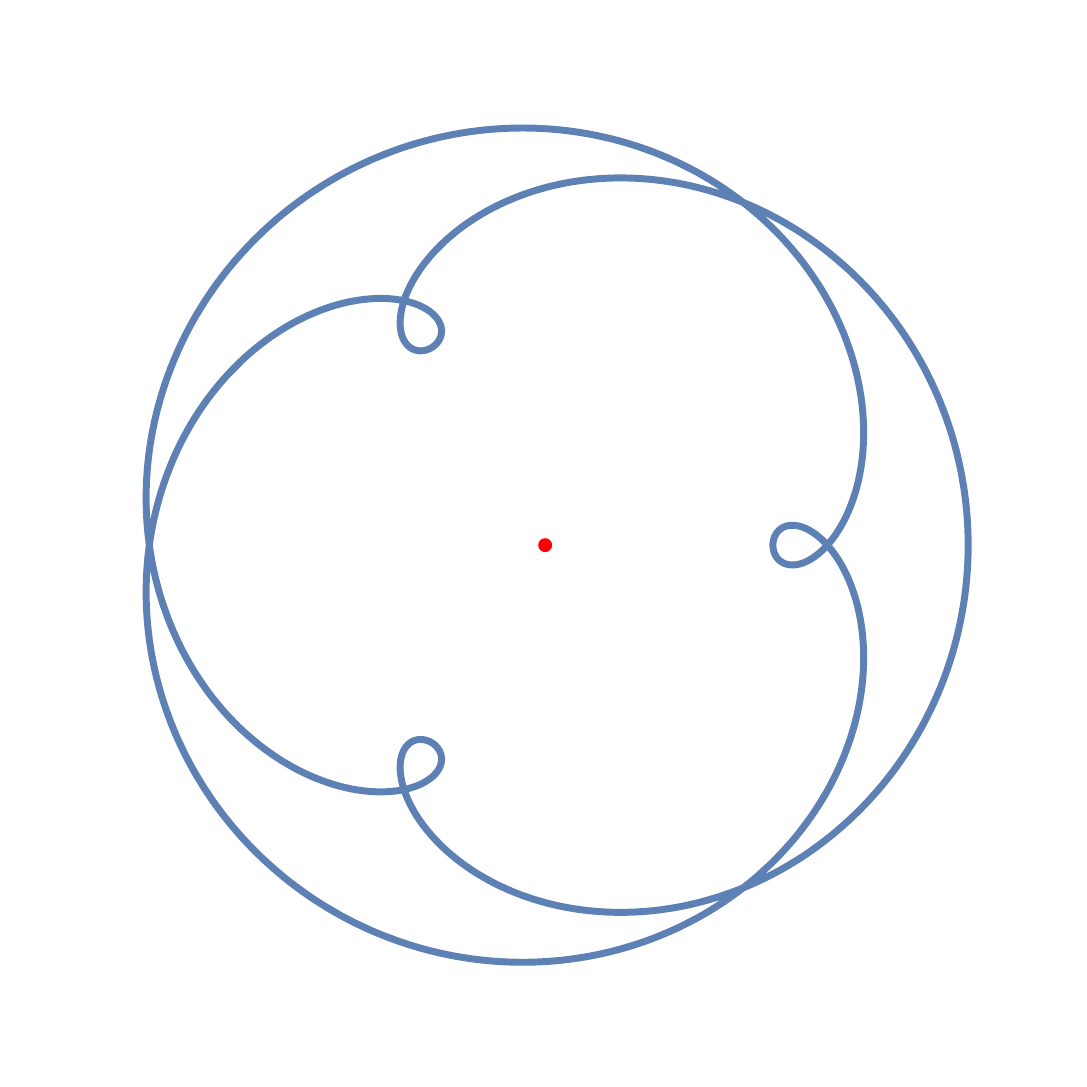}
  \caption{direct, $e=0.3$}
 \label{   }
\end{subfigure}
\begin{subfigure}{0.24\textwidth}
  \centering
  \includegraphics[width=0.95\linewidth]{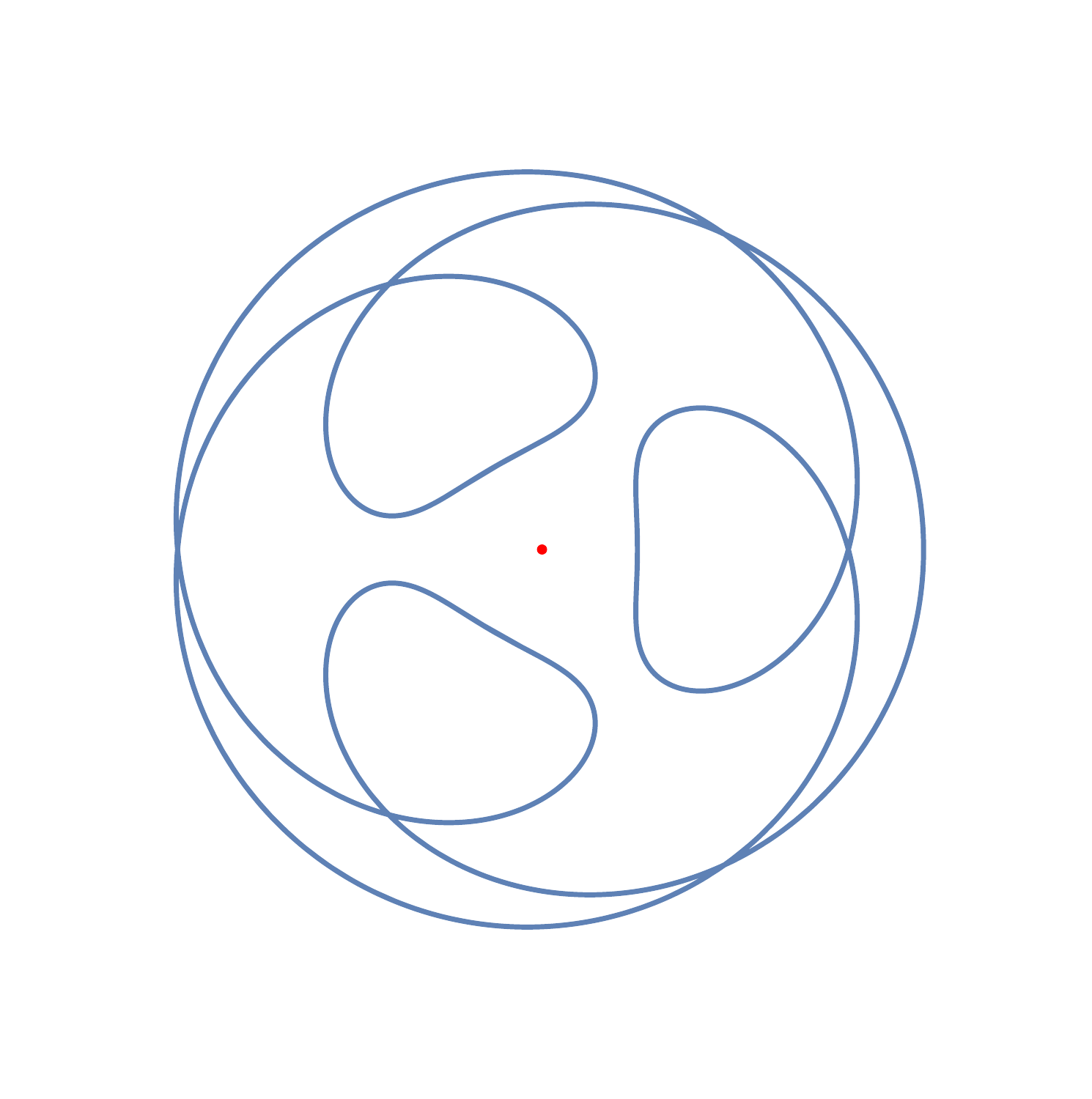}
  \caption{direct, $e=0.6$}
 \label{   }
\end{subfigure}
\begin{subfigure}{0.24\textwidth}
  \centering
  \includegraphics[width=0.95\linewidth]{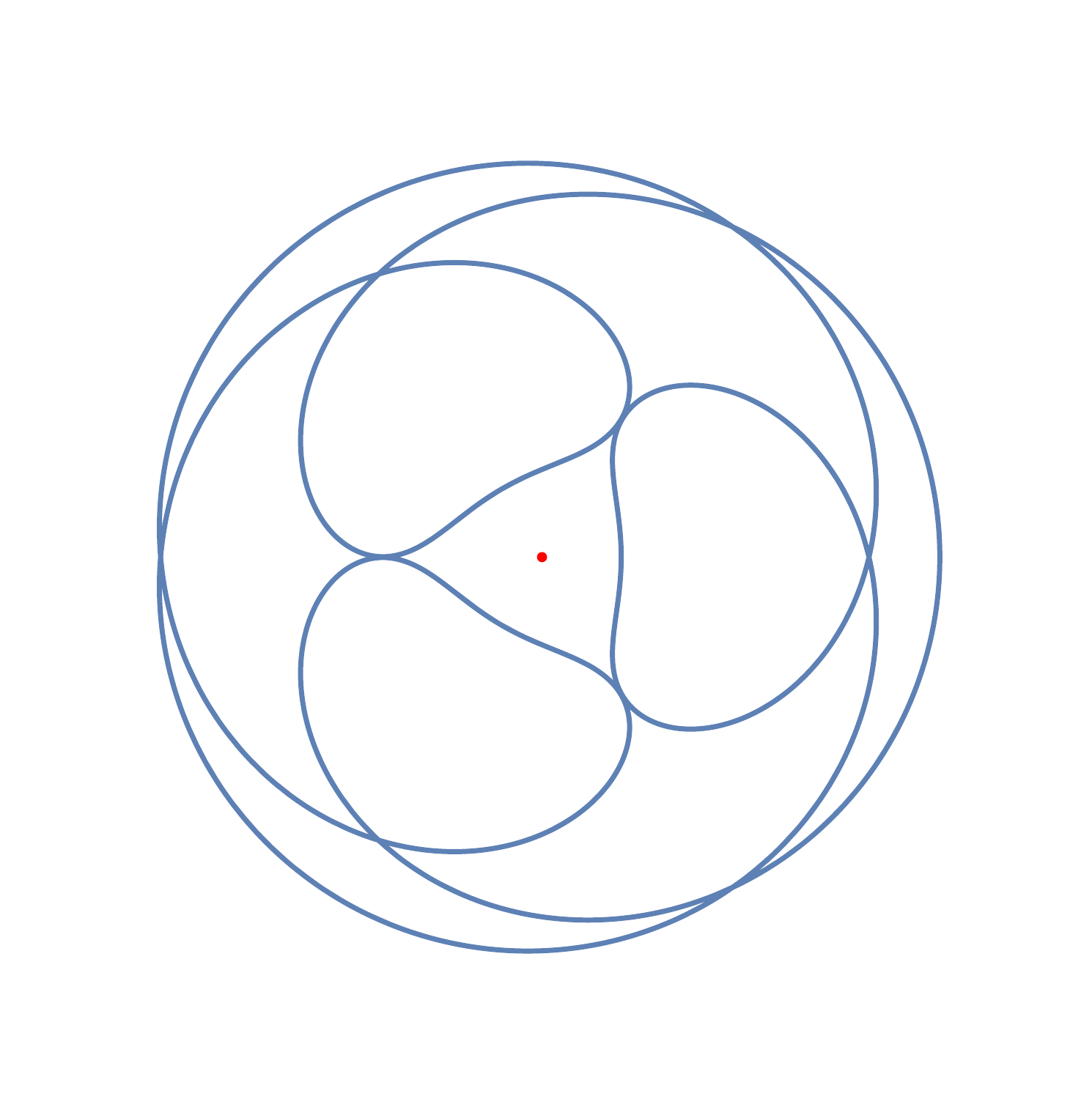}
  \caption{direct, $e    \approx 0.668$}
 \label{   }
\end{subfigure}
\begin{subfigure}{0.24\textwidth}
  \centering
  \includegraphics[width=0.95\linewidth]{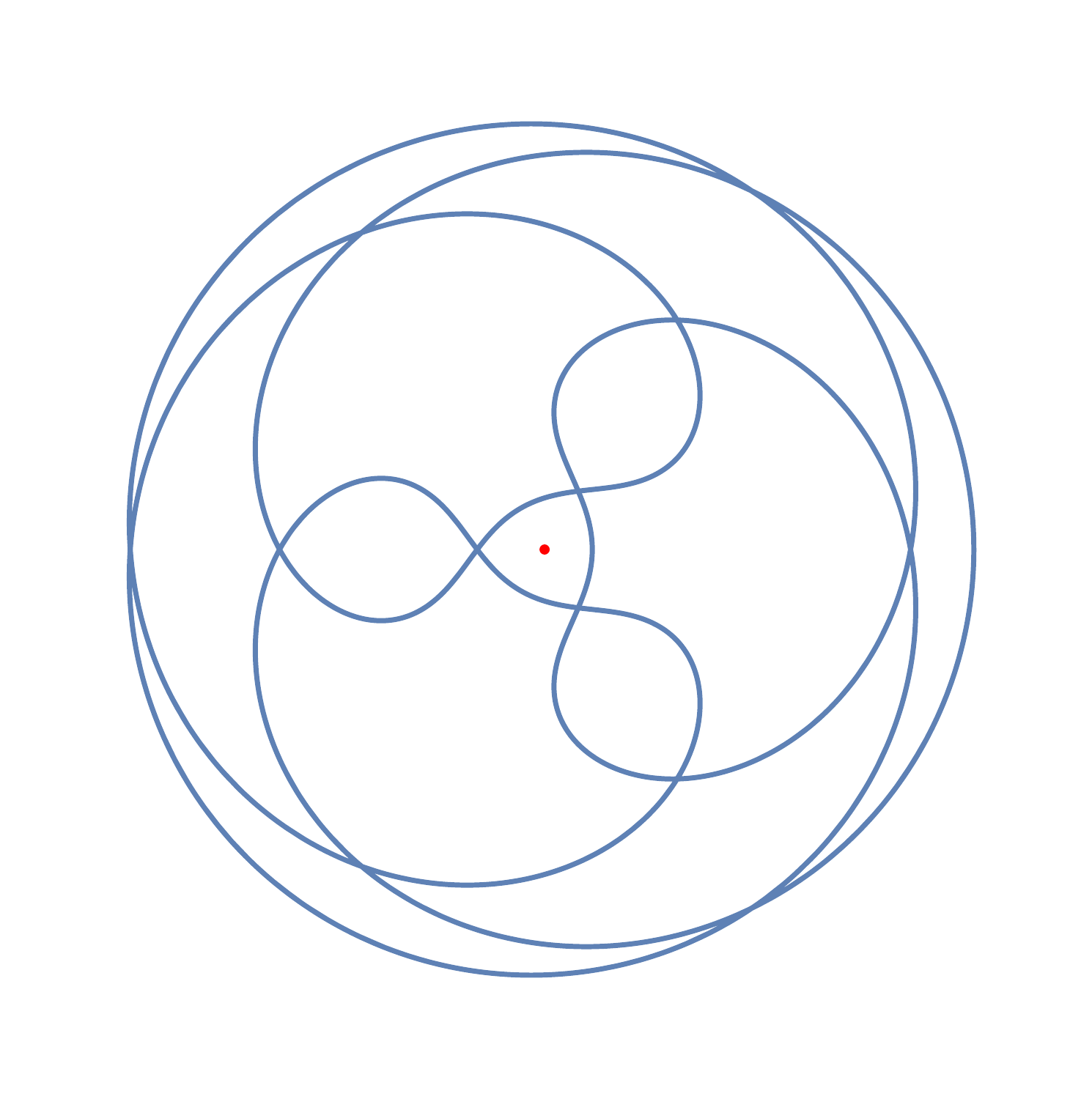}
  \caption{direct, $e=0.8$}
 \label{   }
\end{subfigure}
\begin{subfigure}{0.24\textwidth}
  \centering
  \includegraphics[width=0.95\linewidth]{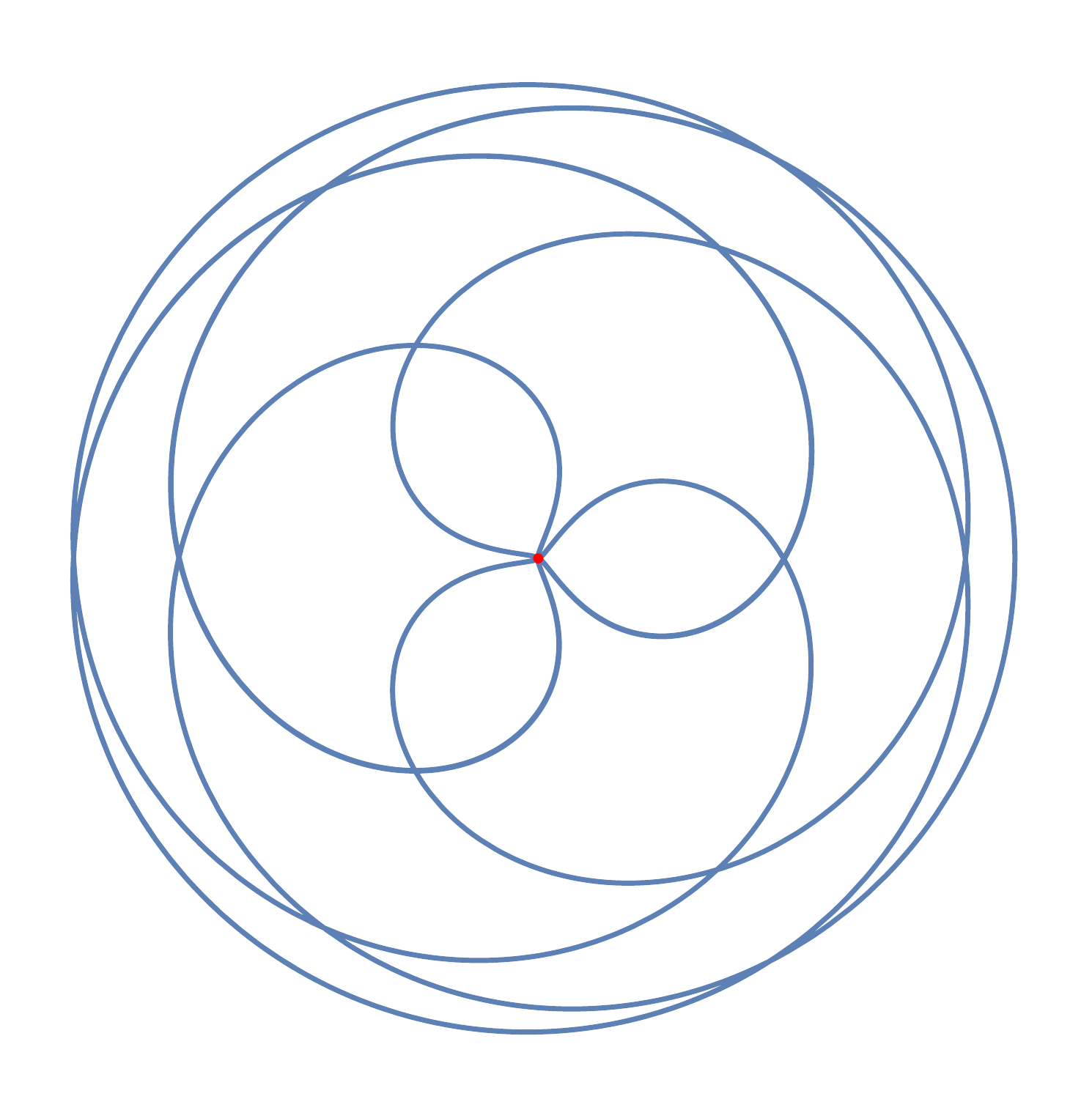}
  \caption{A collision orbit}
 \label{   }
\end{subfigure}
\begin{subfigure}{0.24\textwidth}
  \centering
  \includegraphics[width=0.95\linewidth]{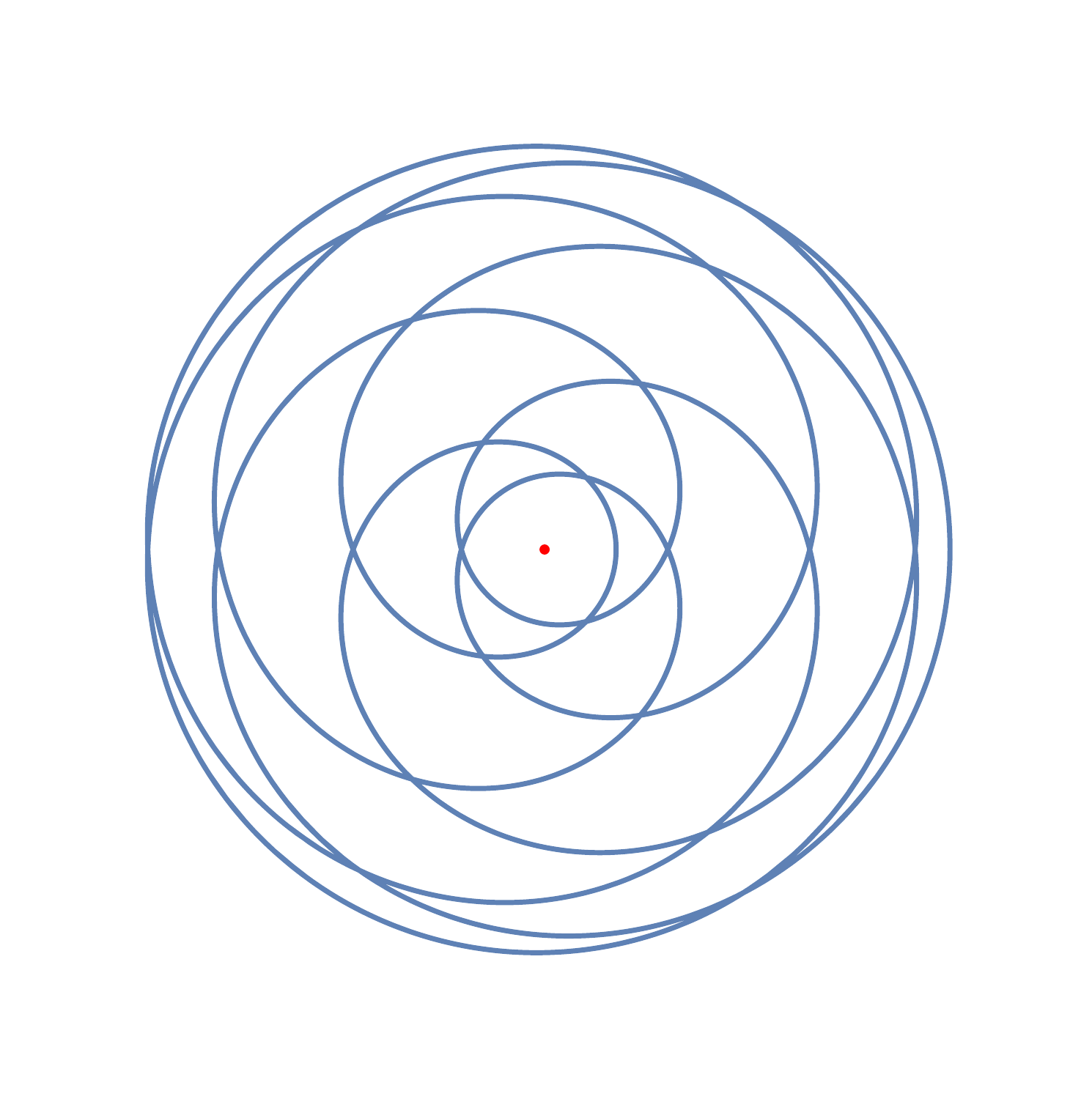}
  \caption{retrograde, $e=  0.7$}
 \label{   }
\end{subfigure}
\begin{subfigure}{0.24\textwidth}
  \centering
  \includegraphics[width=0.95\linewidth]{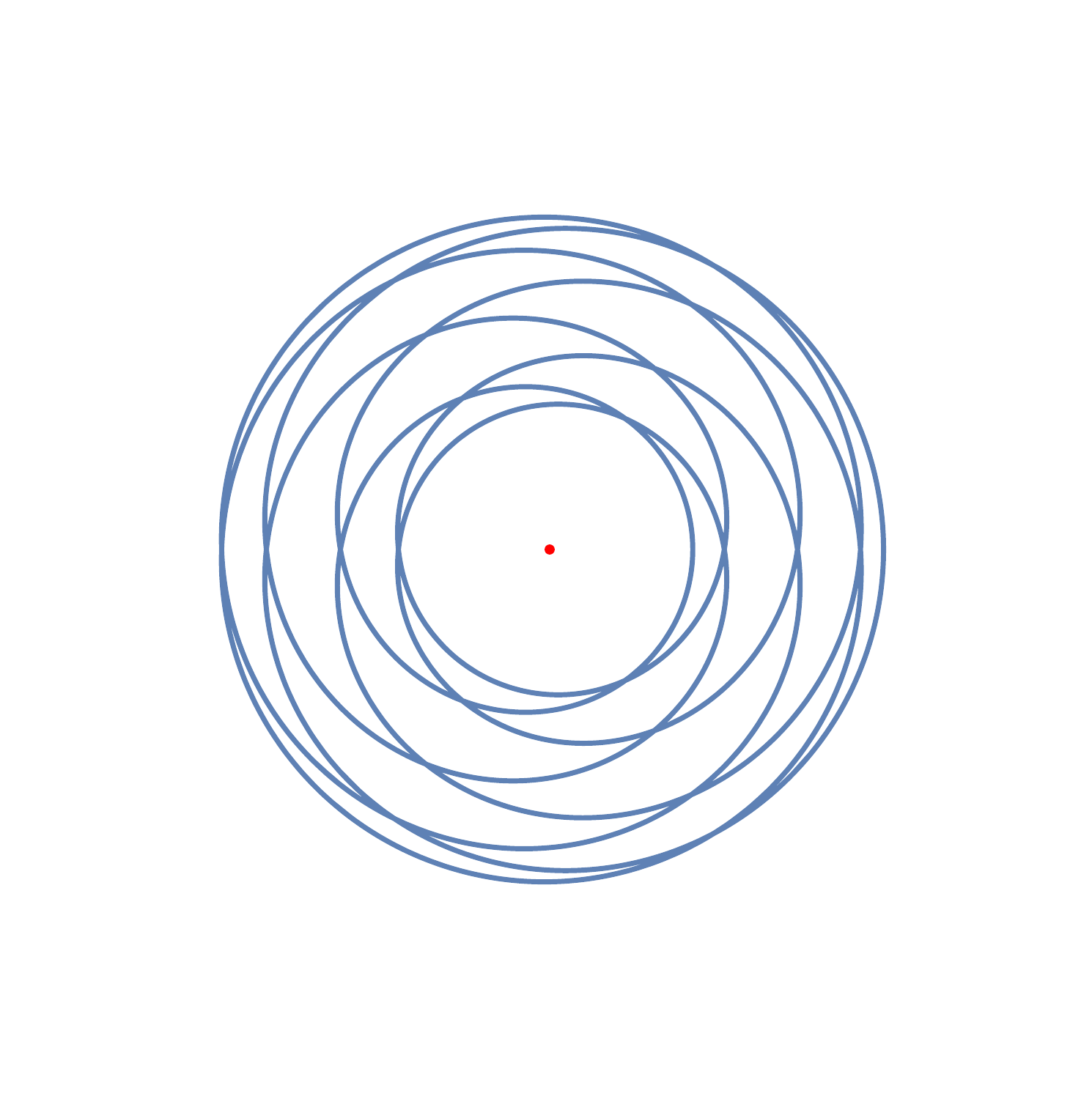}
  \caption{retrograde, $e=0.4$}
 \label{   }
\end{subfigure}
\begin{subfigure}{0.24\textwidth}
  \centering
  \includegraphics[width=0.95\linewidth]{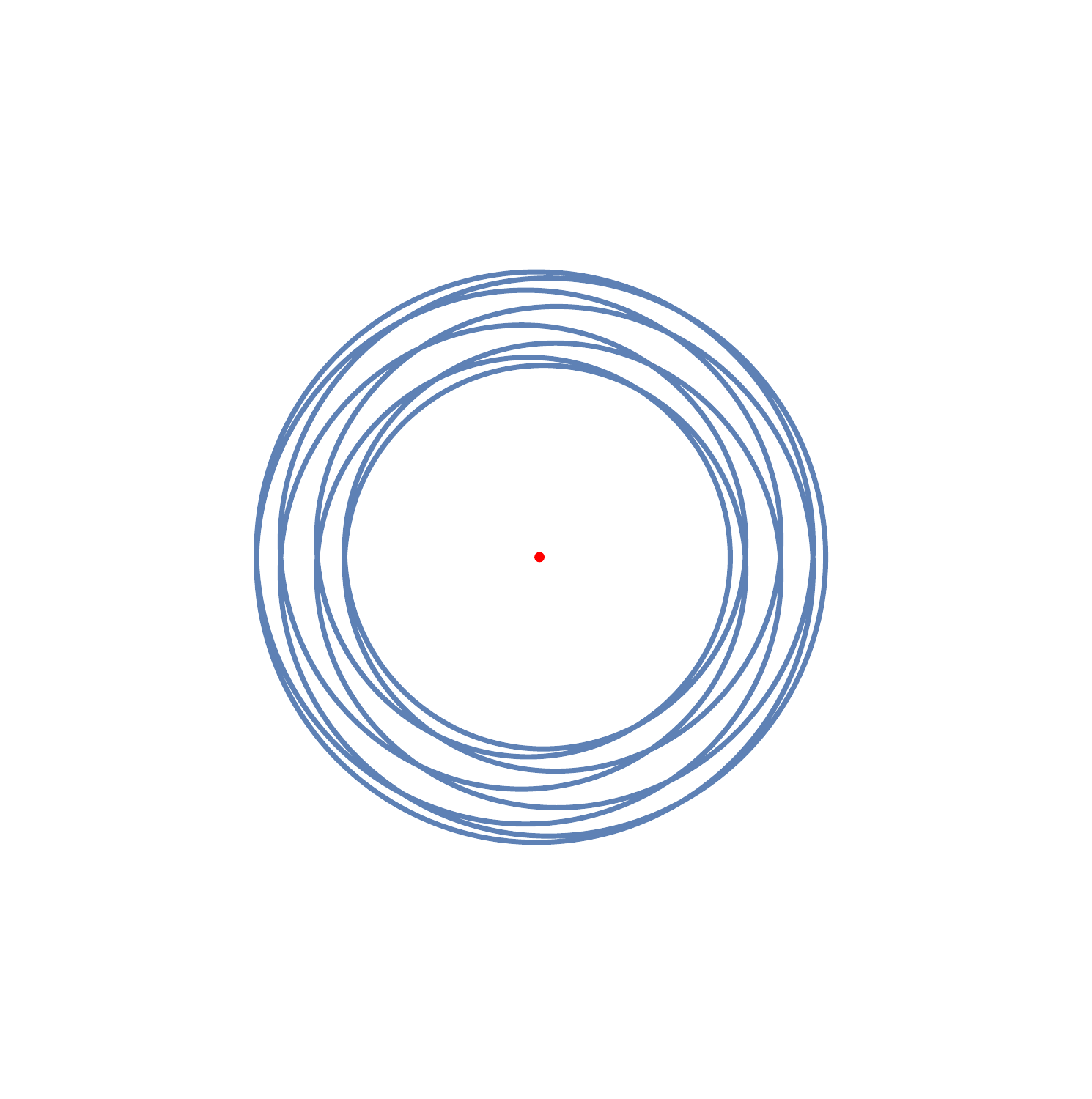}
  \caption{retrograde, $e=  0.2$}
 \label{   }
\end{subfigure}
\begin{subfigure}{0.24\textwidth}
  \centering
  \includegraphics[width=0.95\linewidth]{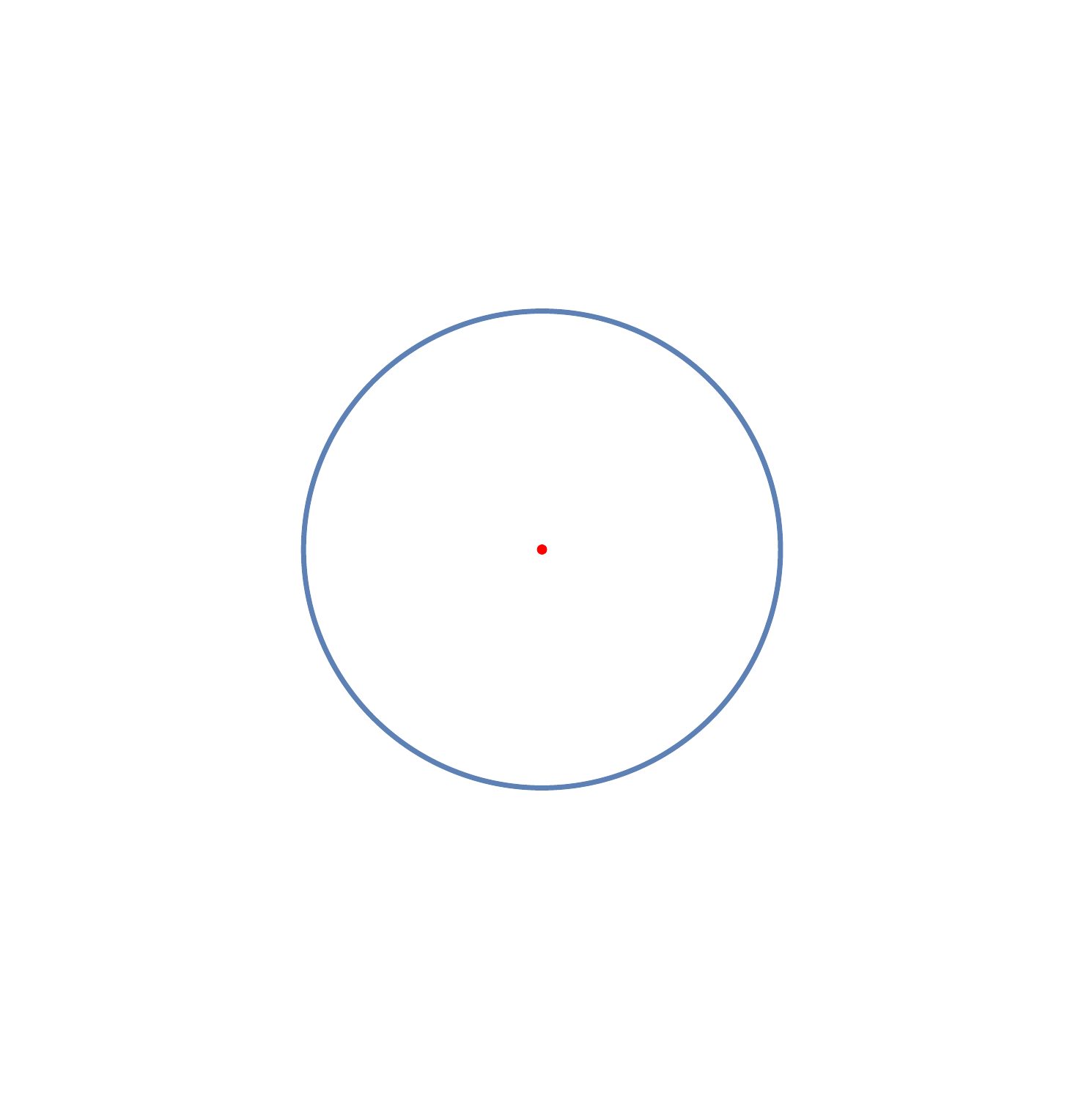}
  \caption{A $5+3=8$-fold retrograde circular orbit}
 \label{   }
\end{subfigure}
\caption{ The $e$-homotopy of the $T_{3,5}$-torus family.  The dashed circles are the boundaries of the unbounded components of the Hill's regions. In (d)-(l) no boundaries are indicated since the associated energies are bigger than the critical Jacobi energy $c_J  = -3/2$.    In (c) the direct $T_{3,5}$-type orbit has cusps at the boundary of the Hill's region and then interior loops appear in (d). Inverse self-tangencies are illustrated in (f). In (h) cusps at the origin happen and loops around the origin appear in (i). }
\label{homotopyofrkp2}
\end{figure}

 \subsection{Trajectories of $T_{k,l}$-type orbits}\label{sdflsdhfeg} Without loss of generality we may assume that the perihelion of $\gamma$ lies at the positive $q_1$-axis. We parametrize it so that the perihelion is the starting point. Then by Lemma \ref{rotsyme}, since $k$ and $l$ are relatively prime, the set of  arguments of the perihelions of $\alpha$ is given by
\begin{equation} \label{setwpik}
\left \{   0, \frac{2\pi  }{k}, \frac{4 \pi }{k}, \cdots, \frac{2(k-1)  \pi   }{k}        \right \}.
\end{equation}
Moreover, by (\ref{goodcompute}) the set of times at which the massless body is located at the perihelions of $\alpha$ is given by 
\begin{equation} \label{setwpik2}
\left \{   0, \frac{2\pi l  }{k}, \frac{4 \pi l }{k}, \cdots, \frac{2(k-1)  \pi  l  }{k}        \right \}.
\end{equation}

In order to argue with the aphelions, we observe that
\begin{equation*} 
\alpha(t+ T/2) = \exp( it + i T/2)\gamma(t+ T/2)  = \exp(   \pi i  l/k) \exp(it) \gamma(t+T/2) . 
\end{equation*}
Since $\text{arg}(\gamma(t) ) = \text{arg}(\gamma(t+T/2)) + \pi$, we see that
\begin{eqnarray*}
\text{arg}( \alpha(t+T/2)) &=& \text{arg} ( \exp(\pi i l /k) \exp( i t) \gamma( t + T/2) )\\
 &=& \pi l /k + \text{arg}(\alpha(t) ) + \pi\\
&=& \pi( k+l)/k + \text{arg}( \alpha(t) ).
\end{eqnarray*}
Therefore, the set of arguments  of the aphelions equals the set (\ref{setwpik}) if $k+l$ is even and equals 
\begin{equation} \label{sdfseglijset}
  \left \{   \frac{\pi   }{k}, \frac{3\pi   }{k}, \frac{5 \pi  }{k}, \cdots, \frac{(2 k-1) \pi  }{k}        \right \} 
\end{equation}
if $k+l$ is odd. Moreover,  the set of times  of the aphelions equals 
\begin{equation} \label{sdfseglijset}
  \left \{   \frac{\pi  l  }{k}, \frac{3\pi l   }{k}, \frac{5 \pi l  }{k}, \cdots, \frac{(2 k-1) \pi  l}{k}        \right \} .
\end{equation}
We summarize the results in the following.

\begin{Lemma}\label{lemmaargutime} Assume that the perihelion of $\gamma$ has the argument $\theta = \theta_0$ and we parametrize $\gamma$ so that the perihelion is achieved at $t =  0$. Then the following statements hold true:
\begin{enumerate}[label=(\roman*)]
\item the sets of the perihelions and aphelions of $\alpha$ are given by
\begin{equation*}
\left \{  \alpha(t) : t =    \frac{2j \pi l}{k}, \; j=0, 1,2, \cdots, k-1             \right \}
\end{equation*}
and
\begin{equation*}
\left \{  \alpha(t) : t =     \frac{(2j+1) \pi l}{k}, \; j=0, 1,2, \cdots, k-1             \right \},
\end{equation*}
respectively;    \vspace{2mm}

\item  assume that $k \pm l$ is even. Then the sets of arguments of the perihelions and aphelions of $\alpha$ are equal and given by 
\begin{equation}\label{argumentmaxnin}
\left \{  \theta : \theta = \theta_0 +   \frac{2j \pi  }{k}, \; j=0, 1,2, \cdots, k-1             \right \}.
\end{equation}
If $k \pm l$ is odd, then the set of arguments of the perihelions is given  by (\ref{argumentmaxnin}) and the set of arguments of the aphelions by 
\begin{equation} \label{sdfseglijset}
\left \{  \theta : \theta = \theta_0 +   \frac{(2j+1) \pi  }{k}, \; j=0, 1,2, \cdots, k-1             \right \}.
\end{equation}
\end{enumerate}
\end{Lemma}

In what follows, we assume that the perihelion of $\gamma$ lies on the positive $q_1$-axis and is achieved at $ t=0$. In view of the rotational and reflection symmetries, to draw the trajectory of a $T_{k,l}$-type orbit $\alpha$ it   suffices to consider the part of the orbit which is contained in the sector
\begin{equation*}
\left \{q \in \C : \arg(q) \in [0, \pi/k] \right \}.
\end{equation*}

Note that the radius of points on $\gamma$ varies in $[r_{\min}, r_{\max}]$, where $r_{\min}$ and $r_{\max}$ are defined as in (\ref{eqrmnsridncx}). Since we have assumed that the perihelion $\gamma(0)$ lies on the positive $q_1$-axis,  for  each $r \neq r_{\min}, r_{\max}$, there exist precisely two points $\gamma(\pm t_r )$ such that $|\gamma(t_r) | = |\gamma(-t_r)| = r$. Therefore, since $\alpha$ is obtained from the $k$-fold covering of  $\gamma$, for each $r \in (r_{\min}, r_{\max})$, there exist precisely $2k$ (possibly with intersections) points 
\begin{equation}\label{twosets}
\left \{ \alpha(t_r + \frac{ 2 \pi j l}{k} ) :  j = 0, 1, 2, \cdots, k-1 \right \} \cup \left \{   \alpha( -t_r + \frac{ 2 \pi j l}{k} ) :j = 0, 1, 2, \cdots, k-1 \right \}
\end{equation}
of radius $r$ on $\alpha$.  If there exists an intersection between those points, i.e., $\alpha( t_r + 2 \pi m l /k) = \alpha ( -t_r + 2 \pi n l/k)$ for some $n,m$, then by the rotational symmetry the two sets in (\ref{twosets}) are identical, namely there exist $k$ double points of radius $r$.

We have fixed the radius $r$ and see how many points of $\alpha$ lies on the circle of radius $r$. In the following lemma we fix the angle  $\theta_0$ and examine points of $\alpha$ which lies on the ray $\theta = \theta_0$.

\begin{Lemma}\label{8h98g97gds}  Let $\alpha$ be a direct $T_{k,l}$-type orbit of eccentricity $e < e_{k,l}^{\infty}$. For each $ \theta_0 \in [0, 2 \pi)$, there exist precisely $|k-l|$ points (possibly with multiple points) of $\alpha$ on the ray $\theta = \theta_0$. If $\alpha$ is retrograde, then there exist $(k+l)$ points (possibly with multiple points) having the same property.
\end{Lemma}
\begin{proof}  Since $e<e_{k,l}^{\infty}$, in view of the argument in Section \ref{disatser}, the angular velocity does not change the sign. Then the assertions follow from the fact that the winding number of direct or retrograde $T_{k,l}$-type orbits equals $l-k$ and $k+l$, respectively.
\end{proof}

We now determine on which rays double points lie. 
 
\begin{Lemma}\label{corollarypositionofdouble} Every double point of $\alpha$ has  argument $ j \pi /k$ for some $j=0,1,2,\cdots, 2k-1$, provided that $ e <e_{k,l}^{\infty}$ for direct $T_{k,l}$-type orbits and that $e \in I_{\text{retro}}$ for retrograde $T_{k,l}$-type orbits.  Moreover, every point on each ray $\theta = j\pi/k$, $j = 0,1,2, \cdots, 2k-1$, is either a perihelion, a aphelion or  a double point. 
\end{Lemma}
\begin{proof} We prove the assertions only for direct $T_{k,l}$-type orbits. One can prove the same results for retrograde orbits in a similar way.

 In view of the rotational and reflection symmetries, the cardinality of the orbit of a point $q \in \C^*$ with $|q| \in [r_{\min}, r_{\max}]$ is given by $2k$ if $\text{arg}(q) \neq j \pi /k$ and by $k$ otherwise. It follows that any double point $q$ must satisfy $\text{arg}(q)= j \pi/k$. This proves the first assertion.

Let $q$ be a double point. By means of Lemma \ref{lemmaargutime}, we see that an aphelion or a perihelion cannot be a double point and hence we have $ r_{\min} < |q| < r_{\max}$. Hence,  the orbit is not perpendicular to the radial direction at the double point $q$. Moreover, since the orbit cannot be contained in the ray $\theta=j\pi/k$ either, by means of  the reflection symmetry with respect to the ray $\theta=j\pi/k$ we see that every point on this ray with the radius in $(r_{\min}, r_{\max})$ is in fact a double point from which the second assertion is proved.    This completes the proof of the lemma.
\end{proof}

\begin{Remark} \rm    Since no self-tangencies happen on the intervals $I_{\text{direct}}^1$ and $I_{\text{retro}}$, all double points in the previous lemma are transverse.
\end{Remark}

\begin{Remark} \rm  The previous lemma also holds true for $e>e_{k,l}^{\infty}$ for direct $T_{k,l}$-orbits. To see this,  we first recall that the angular velocity vanishes only along exterior loops. Then the proof of the previous lemma implies that every double point which does not lie on exterior loops has  argument $j \pi/k$. If there exist no intersections between attached loops, then we have nothing to prove. Suppose that such intersections exist. Then the assertion follows from the fact that $|\text{arg}(\alpha( jT+ t_{\text{inv}})) - \text{arg}(\alpha((j+1)T-t_{\text{inv}} ))|$ is constant for all $j=0,1,2, \cdots, k-1$, where $t=t_{\text{inv}}$  is time at which $\alpha$ has an inverse self-tangency. 
\end{Remark}

We now give an algorithm to draw a piecewise smooth curve, which after smoothing is homotopic without any disasters to $T_{k,l}$-type orbits with eccentricity smaller than $e_{k,l}^{\pm\infty}$.  We only consider direct orbits. For retrograde orbits, one can show with $k+l$ instead of $|k-l|$ in a similar way. In view of the rotational and reflection symmetries,  it suffices to consider the part of the curve in the region
\begin{equation*}
\left \{     q \in \R^2  : \arg(q)    \in  [ 0 ,   \frac{\pi}{ k} ]   , \; r_{\min} \leq|q| \leq r_{\max}     \right \}.
\end{equation*}

\;\;

\textit{Case 1.}  $|k-l|$ is odd.\\
Abbreviate $N= (|k-l|-1)/2$. By Lemma \ref{lemmaargutime} on the ray $\theta = 0$ there exist the perihelion $a_0 = (r_{\min}, 0)$ and $N$ double points $a_1, a_2, \cdots, a_N$. On the other hand, on   $\theta = \pi/k$ , there exist another $N$ points $b_1, b_2, \cdots, b_N$ and the aphelion  $b_{N+1} = r_{\max}\exp( i \pi/k)$. Taking into account  the radial velocity $\dot{r}$, we obtain that
\begin{equation*}
r_{\min} = |a_0| < |b_1| < |a_1| < |b_2| < |a_2| < \cdots < |a_{N-1}| < |b_N| < |a_N| < |b_{N+1}| = r_{\max}.
\end{equation*}
Moreover, the point $b_j$ is connected with $a_{j-1}$ and $a_j$ for $j = 1,2, \cdots, N$ and $b_{N+1}$ is connected only with $a_N$. For convenience, we connected those points by straight lines, see Figure \ref{aaa}. 

\;\;

\textit{Case 2.} $|k-l|$ is even.\\
Abbreviate $N= |k-l|/2$. Then on the ray $\theta =0$ there exist the perihelion $a_0 = (r_{\min}, 0)$, the aphelion $a_N=(r_{\max}, 0)$ and $(N-1)$ double points $a_1, a_2, \cdots, a_{N-1}$. On the ray $\theta = \pi/k$, there exist $N$ double points $b_1, b_2, \cdots, b_N$. Those points satisfy
\begin{equation*}
r_{\min} = |a_0| < |b_1 | < |a_1| < |b_2 | < |a_2| < \cdots < |a_{N-1}| < |b_N| < |a_N| = r_{\max}.
\end{equation*}
In this case, $b_j$ is connected with $a_{j-1}$ and $a_j$ for $j = 1,2, \cdots, N$. Also, we connect them by straight lines, see Figure \ref{bbb}. 
\begin{figure}[t]
\begin{subfigure}{0.48\textwidth}
  \centering
  \includegraphics[width=0.95\linewidth]{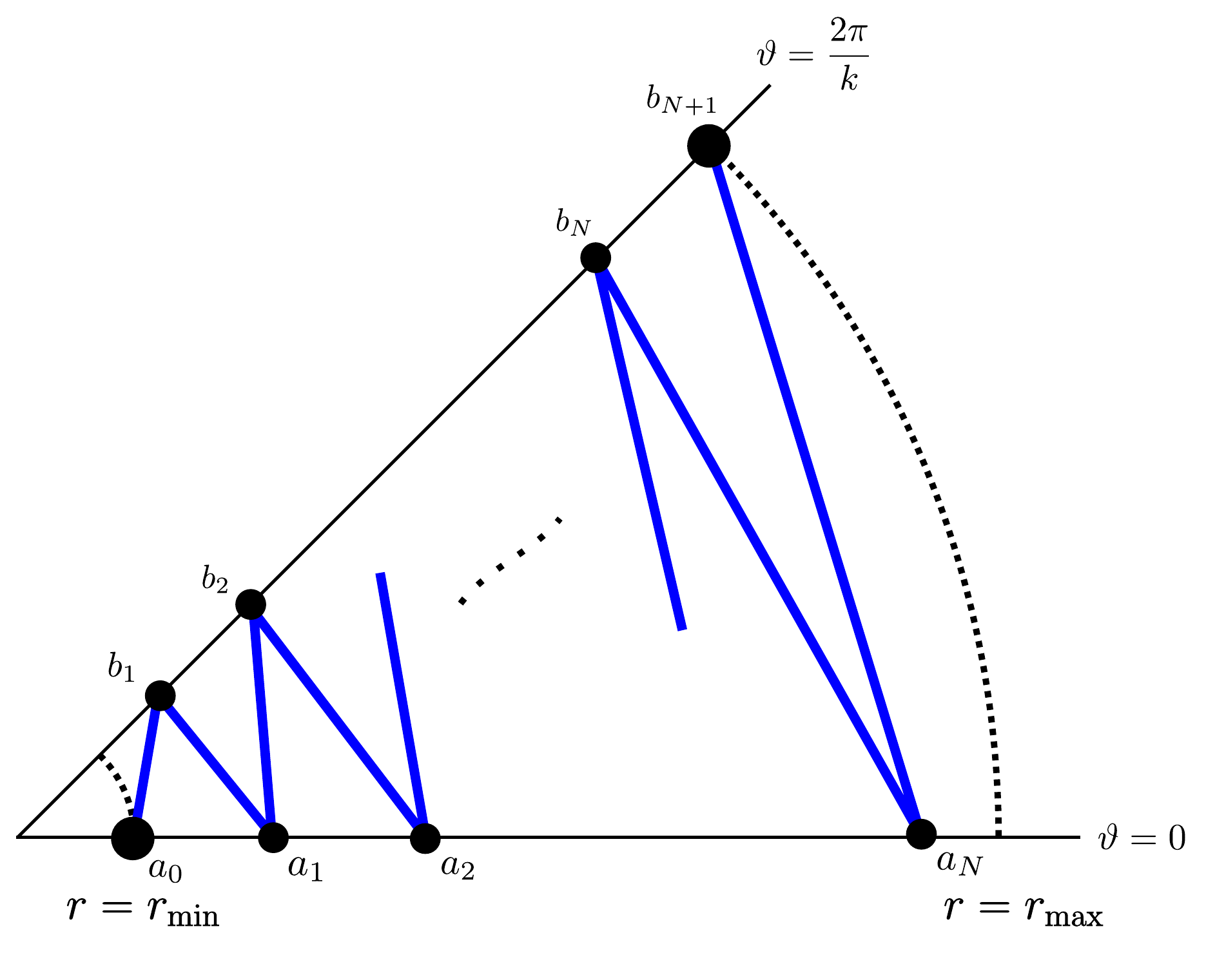}
  \caption{$|k-l|$ is odd  }
  \label{aaa}
\end{subfigure}
\begin{subfigure}{0.48\textwidth}
  \centering
  \includegraphics[width=0.95\linewidth]{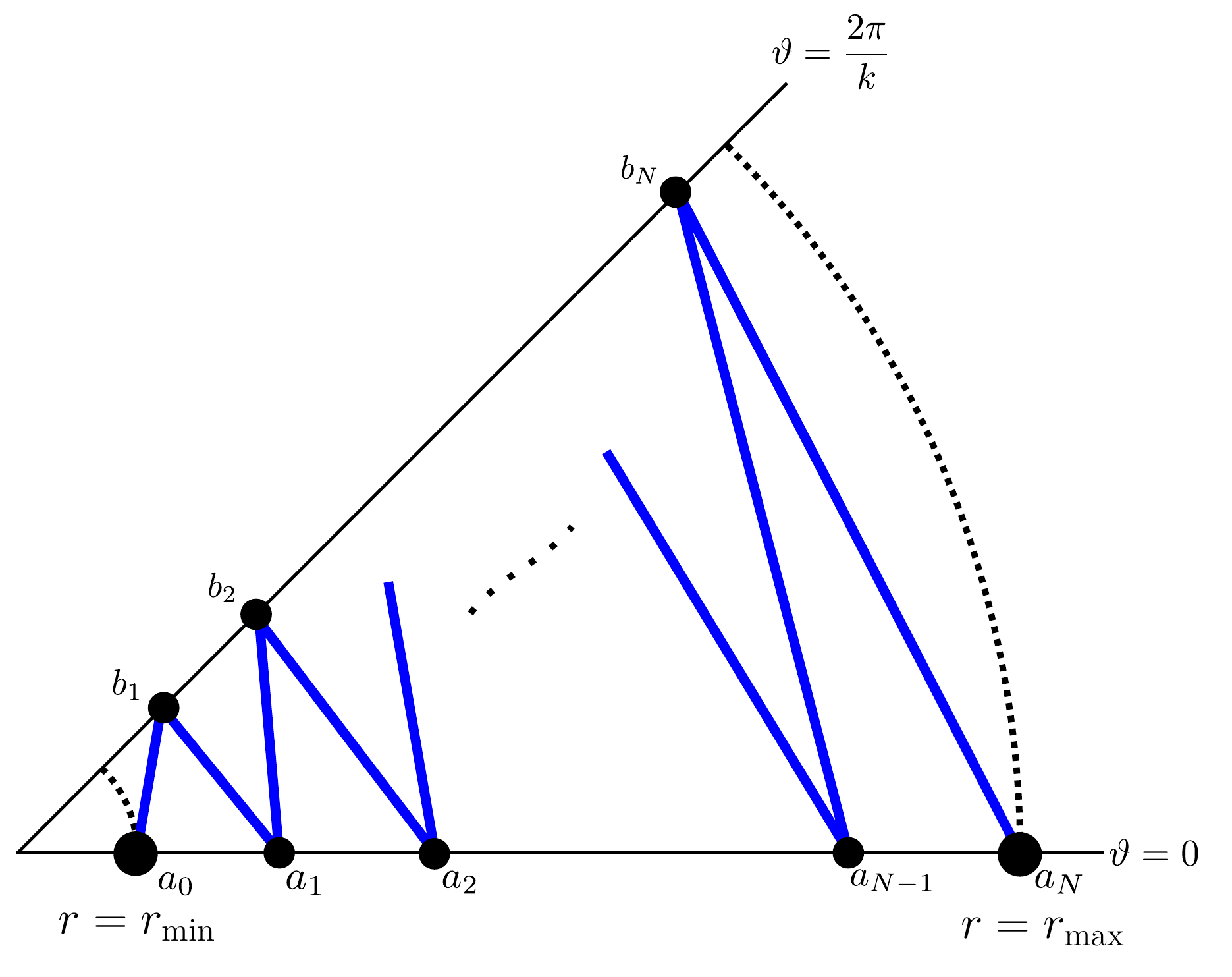}
 \caption{ $|k-l|$ is even  } 
 \label{bbb}
\end{subfigure}
\caption{Connect the marked points}
\end{figure}

\;\;\;

Using the rotational and reflection symmetries, we then obtain the   trajectory of  a direct $T_{k,l}$-type orbit up to homotopy without any disaster.  We denote   the   points on each ray $\theta = 2j\pi/k$ or $\theta = (2j+1)\pi/k$  by the same letters $a$'s or $b$'s, respectively. Recall that we parametrize the orbit so that the perihelion $a_0 = (r_{\min}, 0)$ is the starting point. Then in view of   equation (\ref{goodcompute}) we see that each time interval $[jT, (j+1)T]$ for $j = 0,1,2, \cdots, l-1$ is associated with a finite sequence of points
\begin{equation*}
a_0, \; b_1, \; a_1, \; b_2, \; a_3, \; \cdots, \; a_N, \; b_{N+1}, \; a_N, \; b_{N-1}, \; \cdots, \; b_2, \; a_1, \; b_1 , \; a_0
\end{equation*}
for the case $|k-l|$ is odd or
\begin{equation}\label{sequnsdkjbgs}
a_0, \; b_1, \; a_1, \; b_2, \; a_3, \; \cdots, \; b_N, \; a_{N}, \; b_N, \; a_{N-1}, \; \cdots, \; b_2, \; a_1, \; b_1 , \; a_0
\end{equation}
for the case $|k-l|$ is even, which corresponds to the period $T$ of the underlying Kepler ellipse $\gamma$, see Figure  \ref{figeven}.

 \begin{figure}[t]
\begin{subfigure}{0.32\textwidth}
  \centering
  \includegraphics[width=0.95\linewidth]{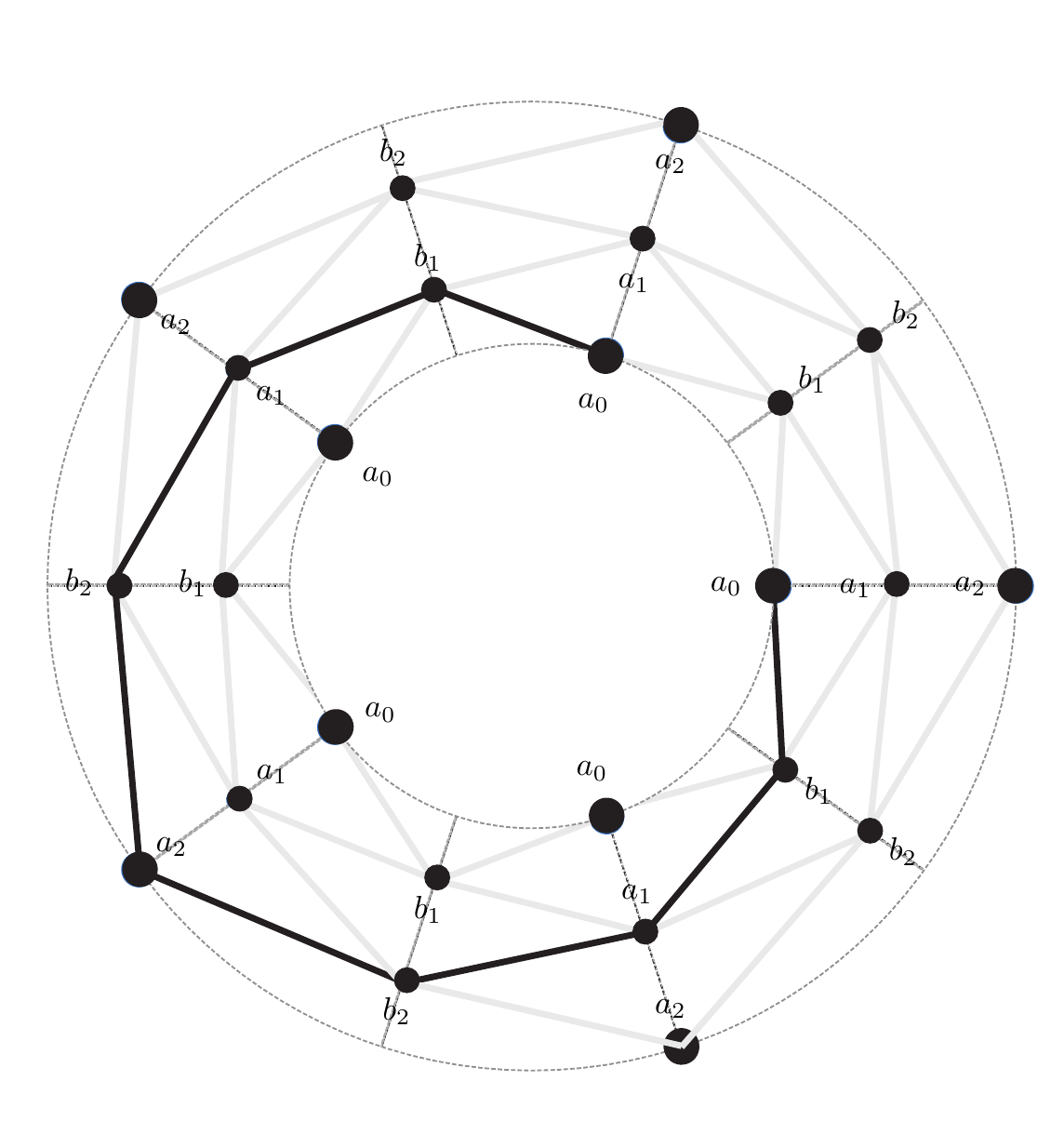}
\end{subfigure}
\begin{subfigure}{0.32\textwidth}
  \centering
  \includegraphics[width=0.95\linewidth]{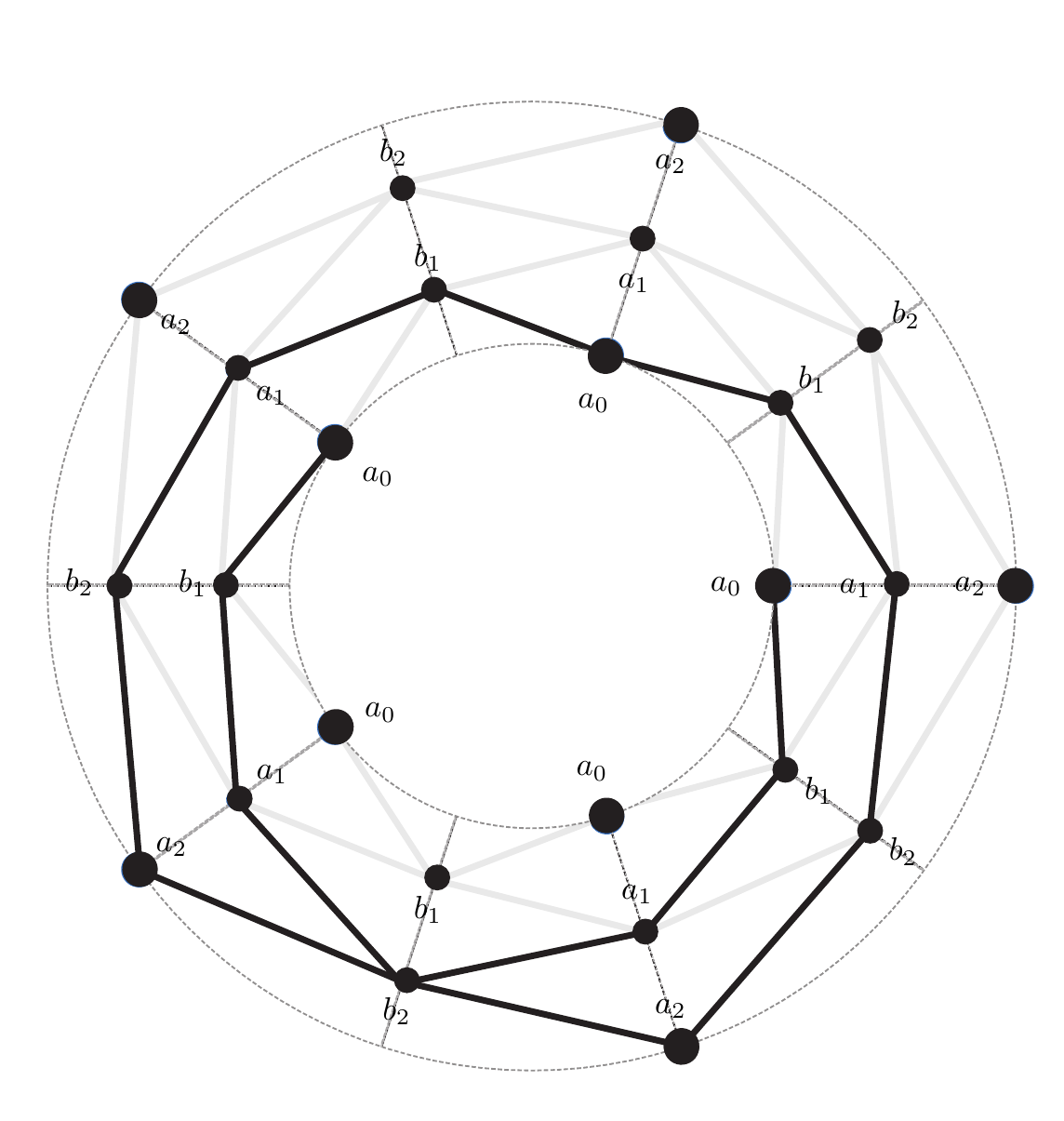}
\end{subfigure}
\begin{subfigure}{0.32\textwidth}
  \centering
  \includegraphics[width=0.95\linewidth]{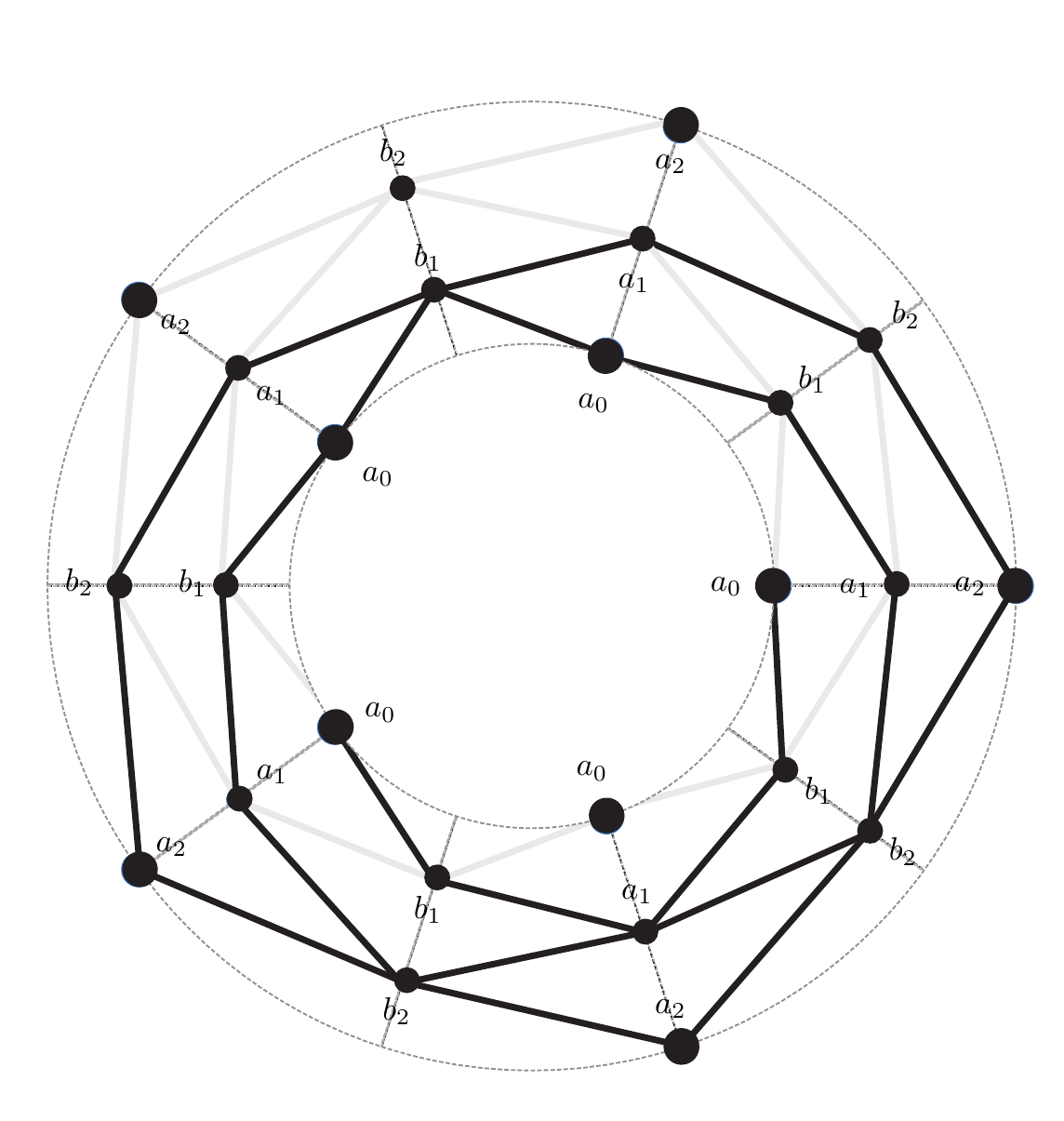}
\end{subfigure}
\begin{subfigure}{0.32\textwidth}
  \centering
  \includegraphics[width=0.95\linewidth]{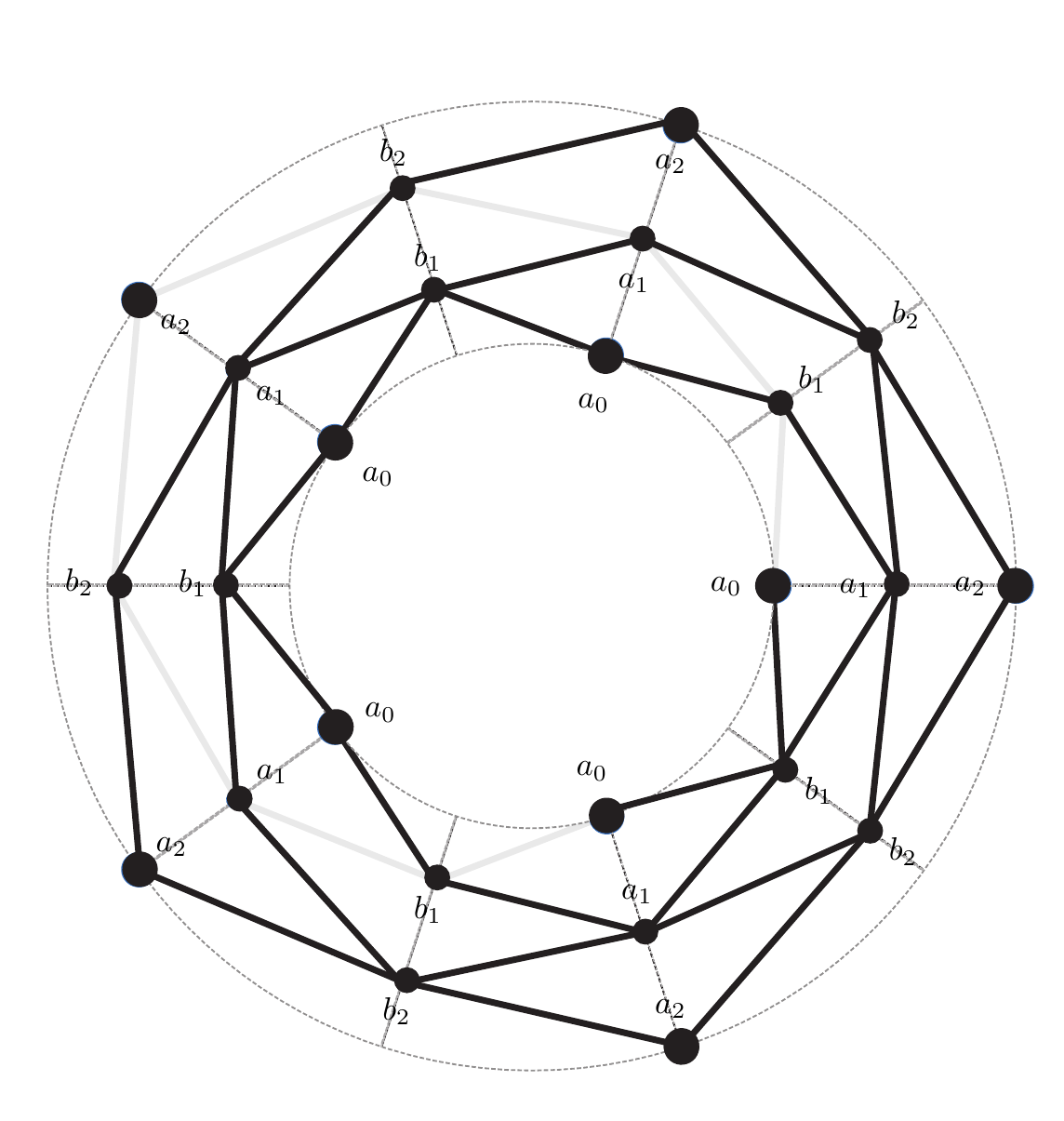}
\end{subfigure}
\begin{subfigure}{0.32\textwidth}
  \centering
  \includegraphics[width=0.95\linewidth]{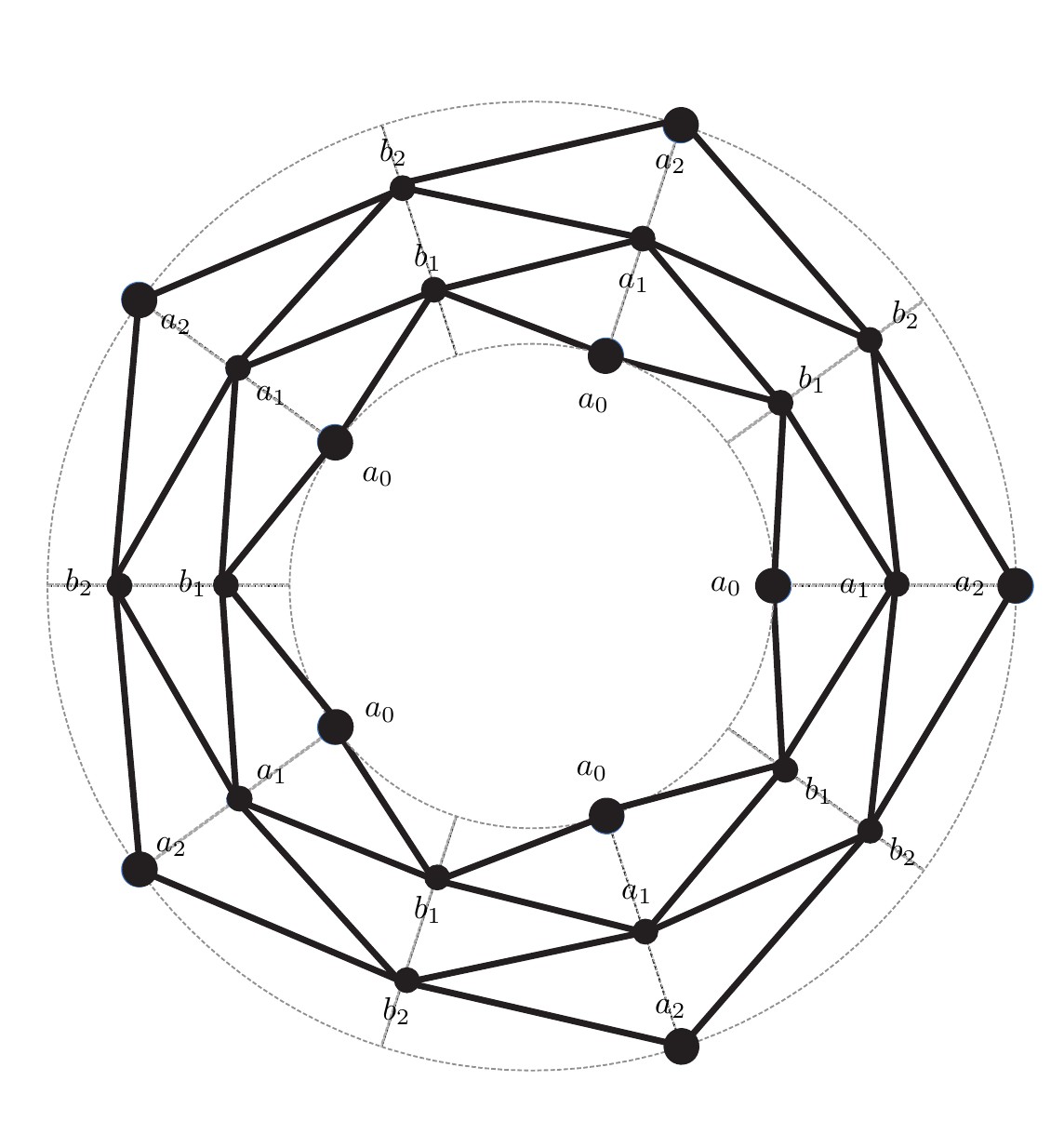}
\end{subfigure}
\caption{A direct $T_{5,1}$-type  orbit (before smoothing). The $j$th step corresponds to a $j$-fold covered Kepler ellipse, $j=1,2,3,4,5$. }
\label{figeven}
\end{figure}

Since we connect the marked points $a$'s and $b$'s on each ray $\theta = j \pi /k$ by straight lines, the obtained trajectory is not smooth. More precisely, the trajectory has a corner at each marked point. However, following the sequence (\ref{sequnsdkjbgs}) one can smoothen (with a small perturbation) the corners and then obtain a unique smooth trajectory which is a generic immersion up to homotopy without any disaster.

In Figure  \ref{hoyle52} we compare  an original orbit  in the rotating Kepler problem and an orbit  obtained by the algorithms (before smoothing).

\begin{figure}[t]
\begin{subfigure}{0.49\textwidth}
  \centering
  \includegraphics[width=0.85\linewidth]{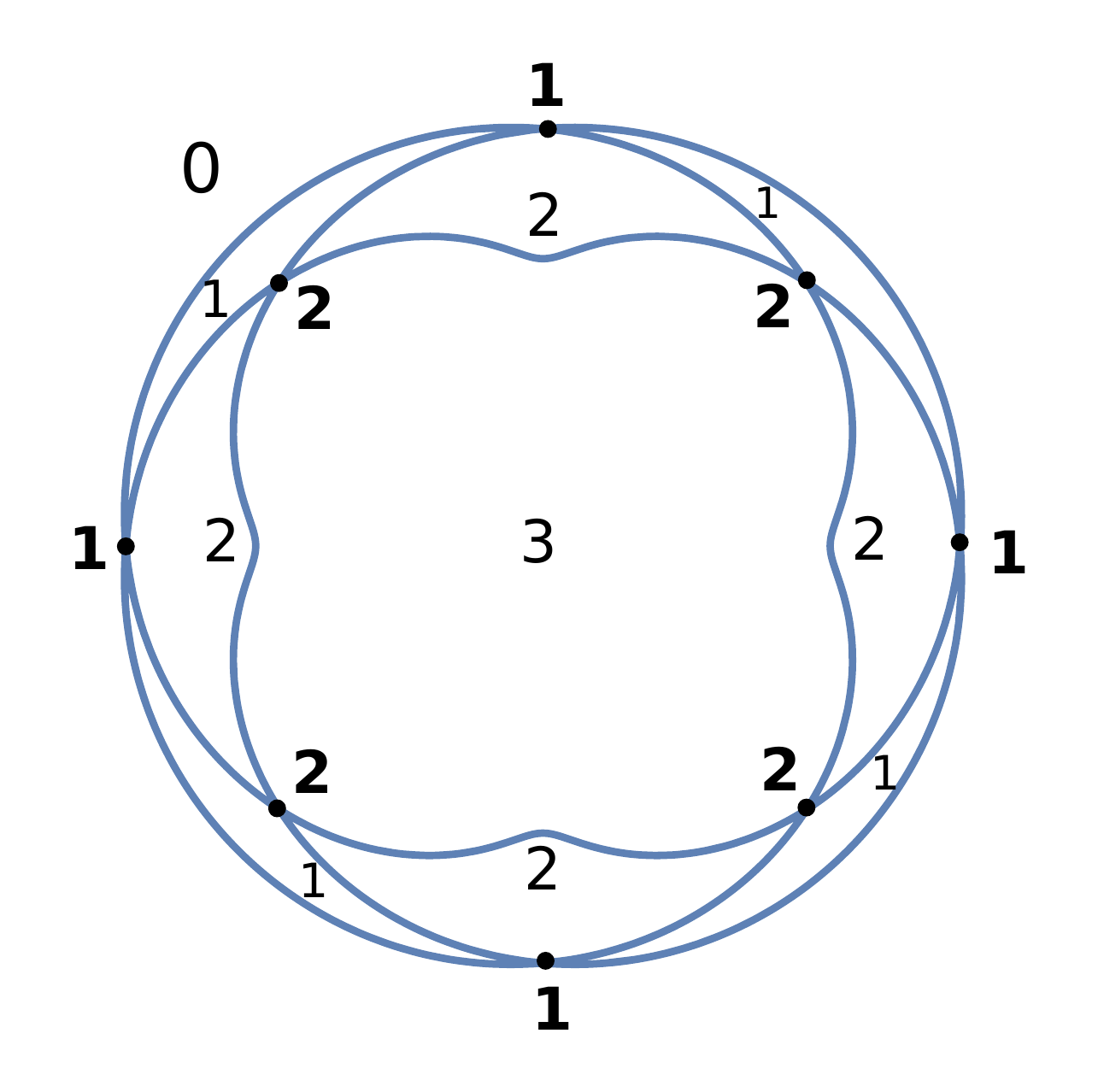}
  \caption{A direct $T_{4,7}$-type orbit of $e=0.2$. The numbers in the interiors of the components means their winding numbers and the bold ones are indices of the double points }
 \label{  }
\end{subfigure}
\begin{subfigure}{0.49\textwidth}
  \centering
  \includegraphics[width=0.85\linewidth]{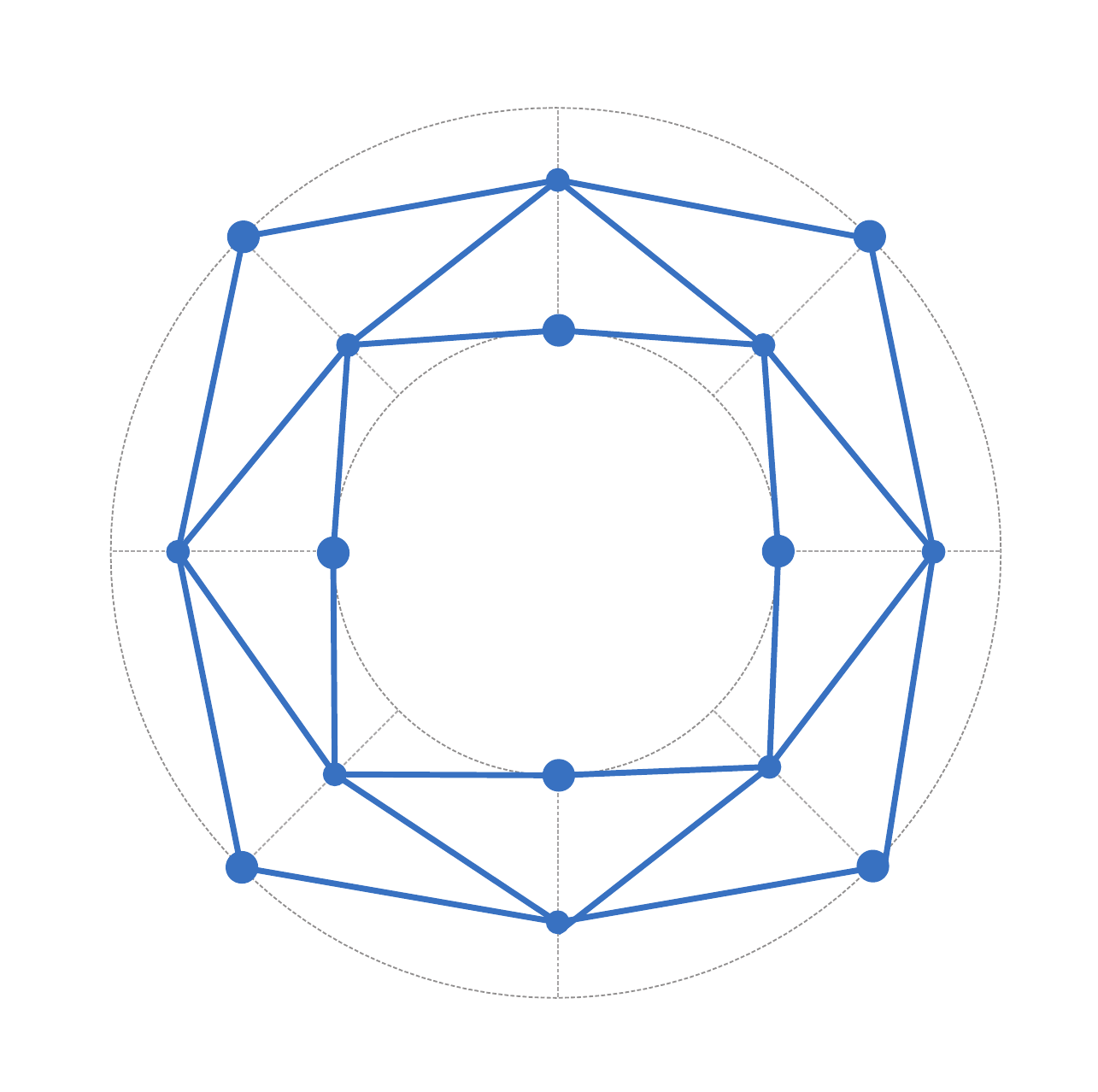}
  \caption{ The orbit obtained by the algorithm}
  \label{  }
\end{subfigure}
\caption{ A direct $T_{4,7}$-type orbit. There exist $4 \times (7-4-1) =8$ double points and $7-4-1=2$ layers. Each layer consists of $4$ connected components. The winding number equals $7-4=3$.}
\label{hoyle52}
\end{figure}

\begin{Remark} \rm  By means of the idea  of Ptolemy and Copernicus which they used to verify their geocentric and heliocentric theories, for example see \cite{Holye},   one can approximate torus-type orbits, provided that the eccentricity is sufficiently small, as follows.

Fix $k$ and $l$ and  consider the $e$-homotopy of the $T_{k,l}$-torus family. To obtain (homotopic) trajectories for small eccentricities, we  begin with a (inertial) Kepler orbit.  Since we are considering negative Kepler energies, all Kepler orbits are ellipse. Recall from \eqref{eqkepler7} that
\begin{equation}\label{sdfsdf1}
r = \frac{ a(1-e^2)}{1+e \cos \theta}.
\end{equation}
By the Kepler  laws we also have
\begin{equation}\label{sdfsdf2}
r^2 \frac{ d \theta }{dt } = \sqrt{ a(1-e^2)} 
\end{equation}
and
\begin{equation}\label{sdfsdf3}
 \frac{k}{l} = \frac{2\pi}{T}=\frac{1}{a^{3/2}},
\end{equation}
where $T$ is the period of the Kepler ellipse.   Using (\ref{sdfsdf1}) and (\ref{sdfsdf3}) we rewrite (\ref{sdfsdf2}) as
\begin{equation}\label{eqderivativee}
  \frac{ d \theta }{dt } = \frac{\sqrt{ a(1-e^2)}}{r^2} = \frac{ k ( 1 + e \cos \theta)^2}{l (1-e^2)^{3/2}} . 
\end{equation}
The Taylor expansion  for the right hand side of (\ref{eqderivativee}) at $e=0$ gives rise to
\begin{equation*}
\frac{ d \theta}{dt} = \frac{k}{l} +   \frac{2k}{l} e \cos \theta + O(e^2)
\end{equation*}
and hence
\begin{align}
 \nonumber \theta &= \frac{k}{l} t +   \frac{2k}{l} e \int_0^t \cos( \theta(\tau)) d\tau + O(e^2) \\
\nonumber &= \frac{k}{l} t +   \frac{2k}{l} e \int_0^t \cos\left( \frac{k}{l} \tau + O(e) \right) d\tau + O(e^2) \\
\label{sdfsdf4} &= \frac{k}{l} t  + 2   e \sin \frac{k}{l}t +O(e^2).
\end{align}
Using \eqref{sdfsdf4}, a similar business for (\ref{sdfsdf1}) yields
\begin{equation}\label{sdfsdf5}
\frac{r}{a} = 1 - e \cos \frac{k}{l}t + O(e^2).
\end{equation}
In view of  (\ref{sdfsdf4}) and (\ref{sdfsdf5}), we now obtain  a complex number expression for the Kepler orbit  
\begin{equation}\label{sdfsdf6}
\zeta(t) = r(t) \exp{ i \theta(t)} = a( 1 - e \cos \frac{k}{l}t + O(e^2) ) \exp i( \frac{k}{l} t  + 2   e \sin  \frac{k}{l}t + O(e^2)).
\end{equation}
The linear approximation to (\ref{sdfsdf6}) in $e$ then yields
\begin{eqnarray}
\nonumber \frac{\zeta(t)}{a} &=& \exp ( i\frac{k}{l}t) +( - \cos \frac{k}{l}t \exp (i\frac{k}{l}t) + 2 i \sin \frac{k}{l}t \exp (i\frac{k}{l}t) ) e   \\
\nonumber &=& \exp ( i\frac{k}{l}t) +(  -\cos^2 \frac{k}{l}t + i \cos \frac{k}{l}t \sin \frac{k}{l}t - 2 \sin ^2 \frac{k}{l}t ) e   \\
\nonumber &=& -2 e + \exp ( i\frac{k}{l}t) + \cos \frac{k}{l}t(    \cos  \frac{k}{l}t    +  i \sin  \frac{k}{l}t )    e   \\
\label{sdfsdf7}&=& -2 e + ( 1 + e  \cos \frac{k}{l}t) \exp(i \frac{k}{l} t).
\end{eqnarray}
If the eccentricity is sufficiently small, or equivalently the   energy is sufficiently close to one of the extremal energies, then (\ref{sdfsdf7}) gives rise to the approximated orbit $z(t) = R_{\pm  t}\zeta(t),$ where $R_{\theta}$ is the rotation matrix and the plus or the minus sign corresponds to the retrograde, respectively, the direct orbit. In Figure \ref{hoyle41sdsd}, we compare an approximated orbit and an original periodic orbit in the rotating Kepler problem of small eccentricity.  

One can use the approximated orbit $z=z(t)$ to prove the lemmas in Sections \ref{sdfsdsssasasd} and \ref{sdflsdhfeg} whose assertions are the main ingredients in the proofs of Lemma \ref{corollarypositionofdouble}. Therefore, the same algorithms as before carry over to the case of the approximated orbit. 
\begin{figure}[t]
\begin{subfigure}{0.49\textwidth}
  \centering
  \includegraphics[width=0.6\linewidth]{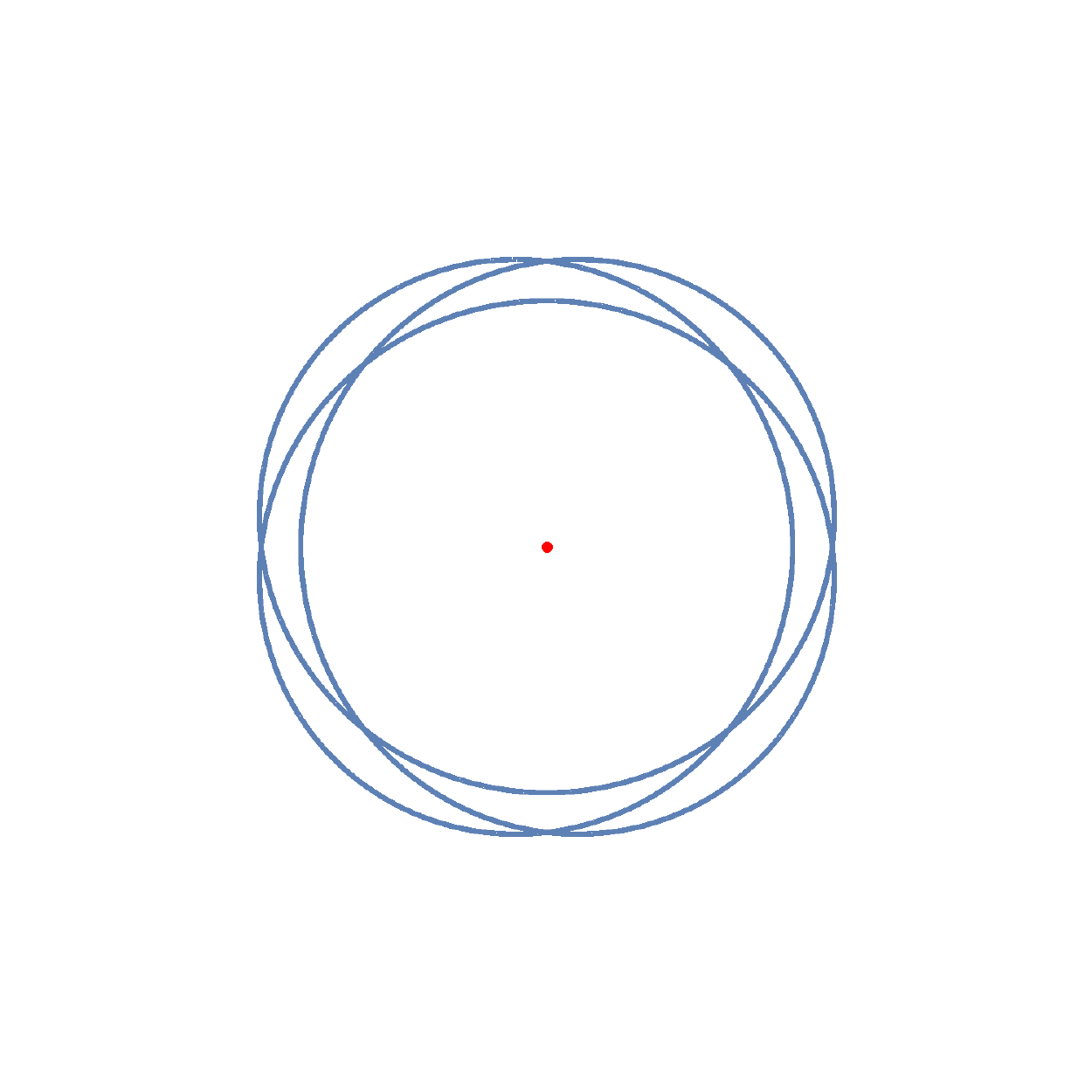}
  \caption{ The approximated orbit}
 \label{   }
\end{subfigure}
\begin{subfigure}{0.49\textwidth}
  \centering
  \includegraphics[width=0.6\linewidth]{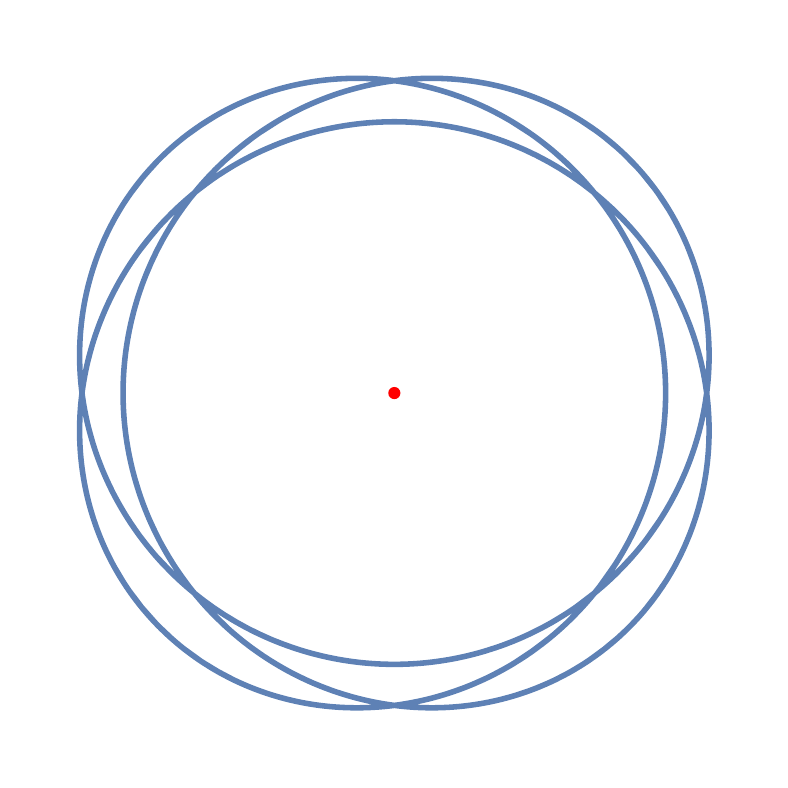}
  \caption{The original orbit in the rotating Kepler problem of which eccentricity equals 0.1}
 \label{  }
\end{subfigure}
\caption{$T_{4,1}$-torus type orbits}
\label{hoyle41sdsd}
\end{figure}
\end{Remark}

\section{Computation of the $J^+$-invariant}\label{sectheorem1}

Fix $k$ and $l$ which are relatively prime and  consider the $e$-homotopy of the $T_{k,l}$-torus family. By  Proposition  \ref{corstar} we just need to choose suitable representatives associated with the intervals for the cases $k>l$ and $k<l$.

\;\;\;

\textit{Case 1.} On $  I_{\text{direct}}$ for $k>l$ or on $I_{\text{direct}}^1$ for $k<l$ \\
We choose $ e \in I_{\text{direct}}^1$ for both cases $k>l$ and $k<l$. After obtaining the trajectory   by the arguments in Section \ref{sdflsdhfeg}, we compute the $J^+$-invariant using the   
 formula which is given in Theorem \ref{virotheorem}. We first observe that the complement of the trajectory in $\R^2$ consists of $k(|k-l|-1) +2$ connected components. The center component contains the origin and the most outside one is the unbounded component. The remaining $k(|k-l|-1)$ components form $(|k-l|-1)$ \textit{layers} of bounded components, where each layer consists of $k$ bounded components.  

Choose  $ \theta_0 \neq j\pi/k$ for some $j$ and rotate the curve by the angle $-\theta_0$ so that the ray $\theta  = \theta_0$ becomes the ray $\theta =0$. Then the ray $\theta =  0$   can be written as the union 
\begin{equation*}
(0, d_0 ] \cup [ d_0, d_1 ] \cup [d_1,d_2] \cup \cdots \cup [d_{|k-l|-2}, d_{|k-l|-1}] \cup [ d_{|k-l|-1}, \infty) 
\end{equation*}
according to intersections with the layers. We label each layer in such a way that \textit{$i$th-layer} is the one which corresponds to the interval $[d_{i-1}, d_i]$ for $ i = 1,2, \cdots, |k-l|-1$. The center and unbounded components are referred to as the \textit{zeroth} and \textit{$|k-l|$th-layers}, respectively.

The winding number of each component is given as follows. Recall that the winding number of a component of the complement is defined to be the winding number of the curve around any interior point in the component. We first note that all components in the same layer are of the same winding number. Choose  some $ \theta_0 \neq j\pi/k$ as in the above.  Since direct $T_{k,l}$-type orbits always rotate  in the same direction: clockwise for the case $k>l$ and counterclockwise for the case $k<l$, we see that the winding number of a component $C$ equals  the number of intersections between the trajectory and the ray $\theta = \theta_0$ starting at any point on $C \cap \left \{ \theta = \theta_0 \right \}$. We then conclude that the zeroth component is of winding number $l-k$ and the components in the $i$th layer are of winding number $l-k-i$. The $|k-l|$th-layer has the winding number zero. Note that  as we traverse from the zeroth layer to the $|k-l|$th-layer, the winding number decreases by one everytime we cross a layer.

We also label double points as follows. Recall that the double points can be divided into $(|k-l|-1)$ subsets according to  the associated radii. We order such radii in increasing order:  $r_1 < r_2 < \cdots < r_{|k-l|-1}$, and  the double points on the circle $r = r_j$ are then called the  \textit{$j$th double points}, $j = 1,2,\cdots, |k-l|-1$.   Recall that the index of a double point is defined by the arithmetic mean of the winding numbers of adjacent components. Note that each $j$th double point is surrounded by four components: one in the $(j-1)$th layer, two in the $j$th layer and one in the $(j+1)$th layer. Since every component in the $i$th layer is of winding number $l-k-i$, we conclude that all the $j$th double points are of index 
\begin{equation*}
\frac{1}{4} \bigg(  (l-k-i+1) + (l-k-i) + (l-k-i) + (l-k-i-1)          \bigg)=l-k-i , 
\end{equation*}
see Figure  \ref{hoyle52}.

By  Viro's formula  the $J^+$-invariant of the direct $T_{k,l}$-type orbits is given by
\begin{eqnarray*}
J^+ &=& 1 + k(|k-l|-1) - \bigg(  1 \cdot  |k-l| ^2 + k \cdot (|k-l|-1)^2 + k \cdot(|k-l|-2)^2 + \cdots + k \cdot 1^2   \bigg) \\
&& \; \;  \; \;\;\;\;\;\;\;\; \;\;\; \;\;\; \;\;\; \;\;\; \;\;\; \;\; + \bigg(  k \cdot (|k-l|-1)^2 + k \cdot (|k-l|-2)^2 + \cdots + k \cdot1^2  \bigg) \\
&=& 1 + k(|k-l|-1) -|k-l|^2\\
&=& \begin{cases}  1-k + kl - l^2 & \text{ for $k>l$ and $  I_{\text{direct}}$} \\ 1 - k - 2k^2 +3kl - l^2 & \text{ for $k<l$ and $I_{\text{direct}}^1$} \end{cases}
\end{eqnarray*}

\;\;

\textit{Case 2.} On $  I_{\text{direct}}^2$ for $k<l$\\
Abbreviate by $\alpha_j$ orbits for $e \in I_{\text{direct}}^j$, for $j =1,2$. Then $\alpha_2$   looks like   $\alpha_1$   with  $k$ interior loops, which do not wind around the origin,  attached at each aphelion.  In particular, the aphelions of $\alpha_1$ become double points.  It does not affect any difference on the winding numbers and indices of other components and double points,  respectively.    Since the winding number of the components enclosed by attached interior loops equals $l-k+1$, the index of the new double points is given by 
\begin{equation*}
\frac{1}{4} \bigg( (l-k+1) + (l-k) + (l-k) + (l-k-1)     \bigg) = l-k.
\end{equation*}
Then in view of the computation in the case 1 we compute that
\begin{eqnarray*}
J^+ &=& 1 - k - 2k^2 +3kl - l^2 + k - k (l-k+1)^2 + k (l-k)^2 \\
&=& 1 - k + kl - l^2.
\end{eqnarray*}

\begin{Remark} \rm One can compute the $J^+$-invariant for the case 2 by using \cite[Lemma 4]{invariant}. Indeed, since $k$ new born loops lie in the center component, their winding number is given by $l-k$. \cite[Lemma 4]{invariant}  says that the $J^+$-invariant for the case 2 differs from that for the case 1 with $l>k$ by $-2k(l-k)$. We then compute that
\begin{eqnarray*}
J^+(\text{Case 2}) &=& J^+(\text{Case 1}) - 2k(l-k) \\
&=& 1-k-2k^2 +3kl -l^2 -2kl +2k^2 \\
&=& 1-k+kl-l^2. 
\end{eqnarray*}
\end{Remark}

\;\;

\textit{Case 3.} On $  I_{\text{retro}}$   for both $k>l$ and $k<l$\\
Recall that for retrograde orbits, no disaster happens. Therefore, one can choose any eccentricity in $I_{\text{retro}}$.  By a  similar argument as in the  case 1 we conclude that the complement consists of $k(k+l-1)+2$ components and they consist of $(k+l+1)$-layers: the zeroth layer (the center component), the $(k+l)$th layer (the unbounded component) and $(k+l-1)$ layers, where each layer consists of $k$ bounded components. The components in the $i$th layer are of winding number $k+l-i$ and then $j$th double points are of index $k+l-i$. We then compute that
\begin{eqnarray*}
J^+ &=& 1 + k(k+l-1) - \bigg(  1 \cdot (k+l)^2 + k \cdot (k+l-1)^2 + k \cdot(k+l-2)^2 + \cdots + k \cdot 1^2   \bigg) \\
&& \;\;\;\;\;\;\;\; \;\;\; \;\;\; \;\;\; \; \;\;\; \;\;\; \;\; + \bigg(  k \cdot (k+l-1)^2 + k \cdot (k+l-2)^2 + \cdots + k \cdot1^2  \bigg) \\
&=& 1 + k(k+l-1) -(k+l)^2\\
&=& 1-k - kl - l^2.
\end{eqnarray*}

\;\;

We have proven
\begin{Proposition}\label{jplusformeusla} Let $\alpha$ be a $T_{k,l}$-type orbit. Its $J^+$-invariant is given by
 \begin{eqnarray*}
J^+(\alpha) &=& \begin{cases} 1 - k + kl -l^2    & \text{ if  ${k>l}$ and $\alpha$ is direct}  \\    1 - k - 2k^2 + 3kl - l^2     &  \text{ if  ${k<l}$ and $e \in I_{\text{direct}}^1$}  \\   1 - k + kl -l^2        &  \text{ if  ${k<l}$ and $e \in I_{\text{direct}}^2$}   \\     1 - k - kl - l^2 & \text{ if $\alpha$ is retrograde}        \end{cases} 
\end{eqnarray*}
\end{Proposition}

Consider  a periodic orbit of second kind (near the heavier primary) in the PCR3BP which is  obtained from a $T_{k,l}$-type orbit under a small perturbation of the mass ratio $\mu$ and which is a generic immersion. Note that if $\mu$ is sufficiently small, the $\mu$-perturbation (or $\mu$-homotopy) is a generic homotopy without any disaster. In particular, the $J^+$-invariant does not change during the $\mu$-perturbation. We then obtain the following

\begin{Corollary}\label{sglihiwle333} Let $\alpha$  be a $T_{k,l}$-type orbit  which is a generic immersion.   There exists a small $\mu_{k,l} >0$ such that for any $\mu < \mu_{k,l}$ all periodic orbits $\delta$ of the second kind which are obtained from $\alpha$ are also generic immersions and have the same $J^+$-invariant as $\alpha$. 
\end{Corollary}

\section{Computation of $\mathcal{J}_1$ and $\mathcal{J}_2$ invariants}\label{sectionslkdhf}
  Recall that the winding number around the origin of  retrograde  or   direct $T_{k,l}$-type orbits is given by $k+l$ or $l-k$, respectively. The formulas for the $\mathcal{J}_1$ invariant then follow from its definition and Proposition \ref{jplusformeusla}.

\;\;

In order to  obtain the $\mathcal{J}_2$ invariant, by definition, one should consider the Levi-Civita mapping $L : \C^* \rightarrow \C^*$, $z \mapsto z^2$. Let $K$ be the trajectory of a $T_{k,l}$-type orbit $\alpha$. We parametrize $\alpha$ so that the starting point is one of the perihelions and assume that it lies on the positive $q_1$-axis. We then    observe that since   $L$ is a squaring map, the results in Section \ref{sectionRKP} give rise to the following:

\begin{enumerate}[label=(\roman*)]

\item the curve $L^{-1}(K)$  is invariant under the rotation by the angle $  j\pi/k$ and    is symmetric with  respect to the line $y=(j \pi /2k ) x$, $j=0,1,2,\cdots, 2k-1$, see Lemmas \ref{rotsyme} and \ref{lemmareflds} \vspace{2mm};

\item   if  $k \pm l$ is odd, then  the sets of the arguments of the perihelions and aphelions of $L^{-1}(K)$ are given by
\begin{equation}\label{eqminimumargument2}
\left \{ 0, \frac{ \pi}{k}, \frac{2\pi}{k}, \cdots, \frac{  (2k-1) \pi}{k}  \right \} 
\end{equation}
and
\begin{equation*}\label{eqmaximumargument2}
\left \{   \frac{ \pi}{2k}, \frac{3\pi}{2k}, \cdots,  \frac{ (4 k-1) \pi}{2k} \right \} ,
\end{equation*}
respectively. If  $k\pm l$ is even, then the two sets are equal and given by (\ref{eqminimumargument2}), see Lemma \ref{lemmaargutime}\vspace{2mm};

\item  if $K$ is direct, for each $\theta_0 \in [0,2\pi )$, there exist precisely $2|k-l|$ points (possibly with multiple points) of $L^{-1}(K)$ on the ray $\theta = \theta_0$, provided that $e < e_{k,l}^{\infty}$. If $\alpha$ is retrograde, then the same property holds true with $2(k+l)$, see Lemma \ref{8h98g97gds}\vspace{2mm};

\item  every double point has  argument $ j \pi /2k$ for some $j=0,1,2,\cdots, 4k-1$, see Lemma \ref{corollarypositionofdouble}\vspace{2mm};

\item given $r \in (\sqrt{r_{\min}}, \sqrt{r_{\max}})$, the number of points on $L^{-1}(K)$ with radius $r$ equals either $4k$ if the points are non-double points or $2k$ if the points are double points, where $r_{\min}$ and $r_{\max}$ is minimal respectively maximal radius.

\end{enumerate}

\;

\textit{Case 1.} Either $k>l$ or $k<l$ and $ I_{\text{direct}}^1$; the winding number of $K$  is odd.\\
Recall that the pulled back trajectory $L^{-1}(K)$ consists of a single curve and its winding number  equals the winding number of $K$, i.e., $l-k$. Bearing the facts (i)-(v) in mind,  one can draw  $L^{-1}(K)$ by an algorithm similar to the one given in Section \ref{sdflsdhfeg}  with $2k(|k-l|-1)$ double points and $(|k-l| -1)$ layers.
see Figure \ref{hoyleroot41}. 
\begin{figure}[t]
\begin{subfigure}{0.45\textwidth}
  \centering
  \includegraphics[width=0.85\linewidth]{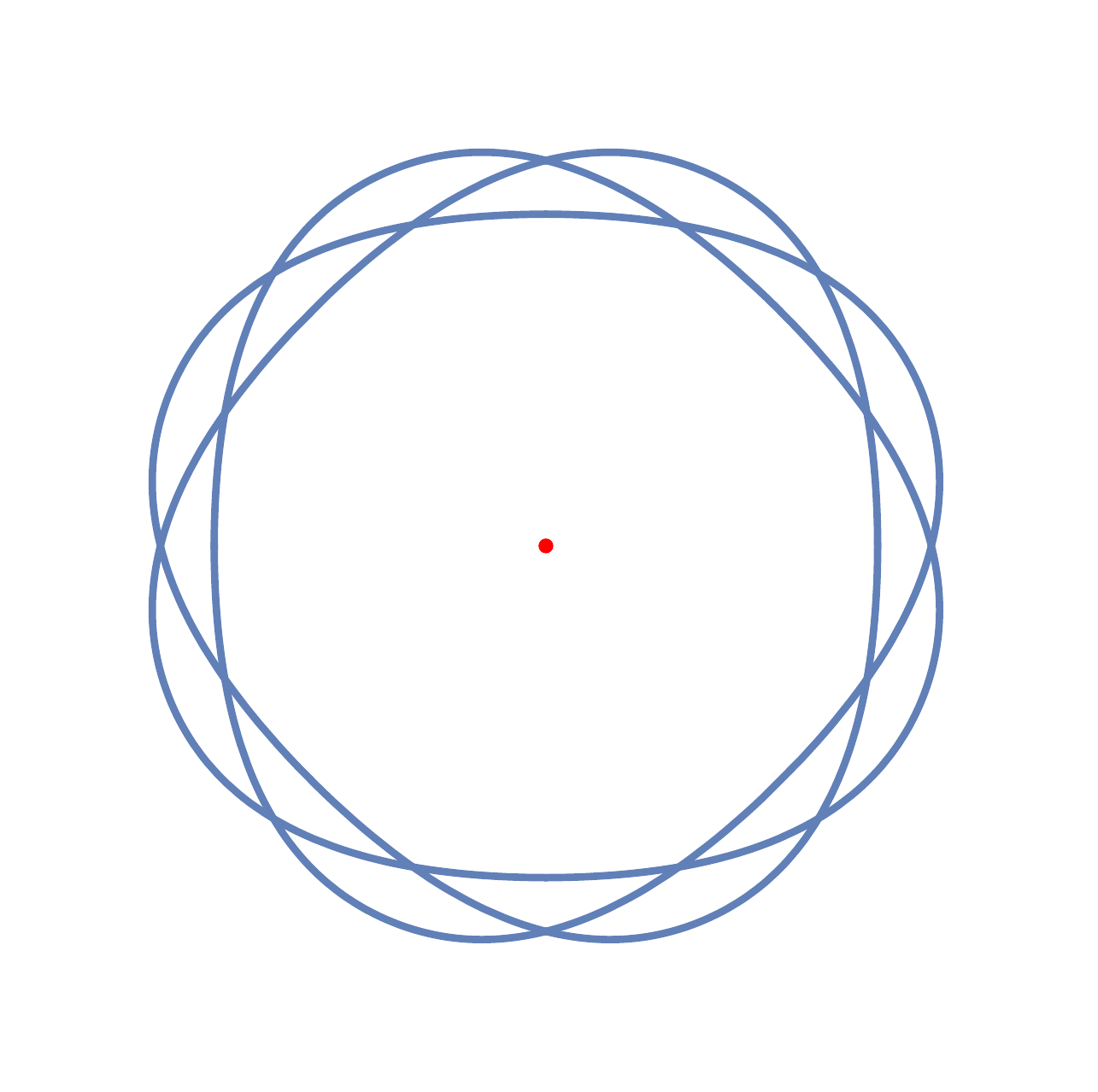}
  \caption{A  pulled back orbit of a direct $T_{4,1}$-type orbit}
 \label{   }
\end{subfigure}
\begin{subfigure}{0.45\textwidth}
  \centering
  \includegraphics[width=0.85\linewidth]{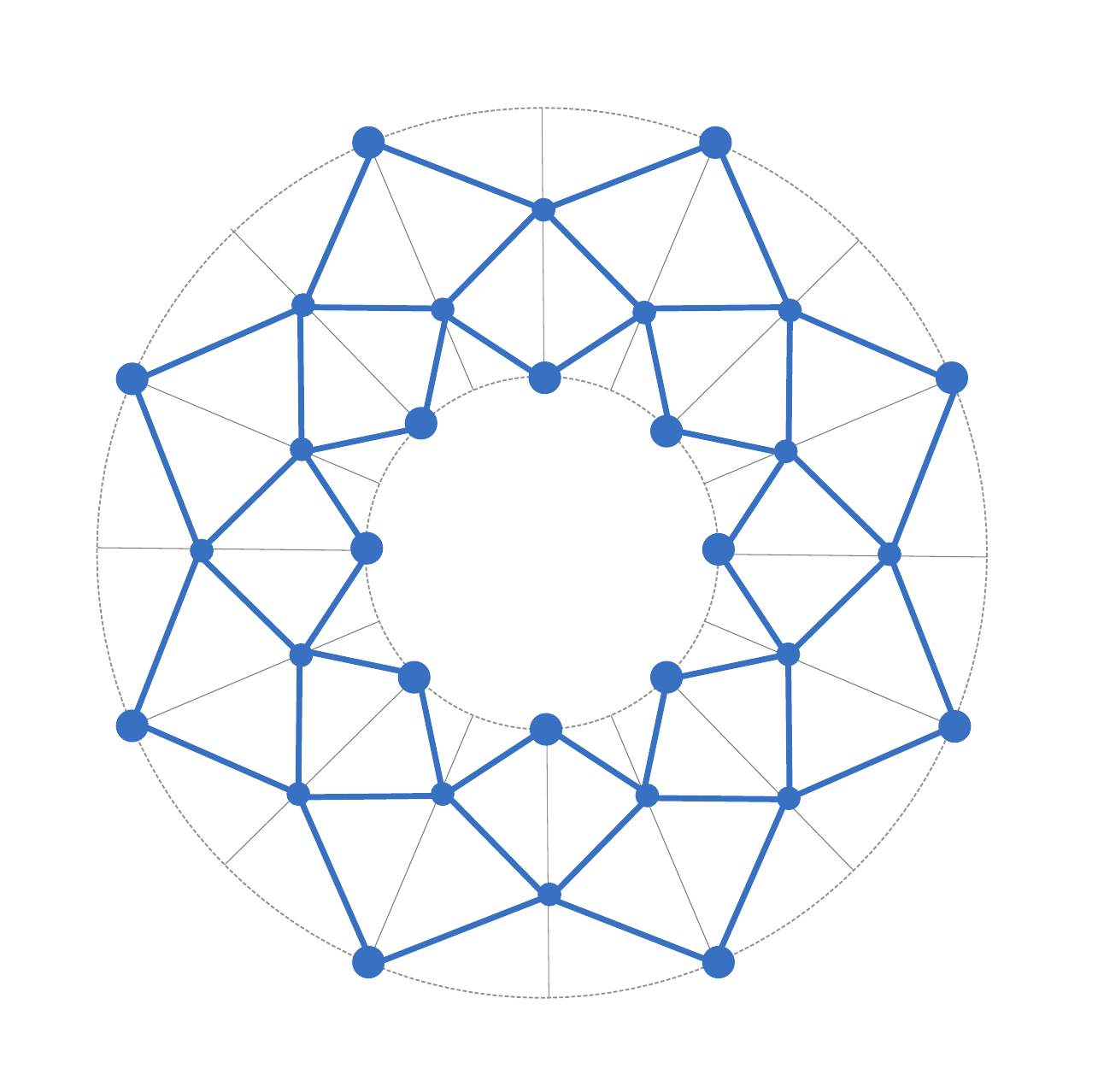}
  \caption{An orbit obtained by the algorithm}
  \label{  }
\end{subfigure}
\caption{ A $L^{-1}(T_{4,1})$-type orbit.  There exist $(2\times 4) \times (4-1-1) = 16$ double points and $4-1-1=2$ layers. Each layer has $2\times 4=8$ connected components. The winding number is $1-4=-3$.}
\label{hoyleroot41}
\end{figure}
It then follows from   Viro's formula  that
\begin{eqnarray*}
\mathcal{J}_2(K) &=& J^+( L^{-1}(K) ) \\
&=& 1 + 2k( |k-l|-1)   - \bigg(  1 \cdot |k-l|^2 + 2k \cdot (|k-l|-1)^2 +2 k \cdot(|k-l|-2)^2 + \cdots + 2k \cdot 1^2   \bigg) \\
&& \;\; \;\;\;\;\;\;\;\; \;\;\; \;\;\; \;\;\; \;\;\; \;\;\; \;\; \;\;+ \bigg(  2k \cdot (|k-l|-1)^2 + 2k \cdot (|k-l|-2)^2 + \cdots + 2k \cdot1^2  \bigg) \\
&=& 1 + 2k(|k-l|-1) -|k-l|^2\\
&=& \begin{cases} (k-1)^2 - l^2 & \; \text{ if $k>l$} \\ 1-2k-3k^2 +4kl - l^2 & \; \text{ if $k<l$ and $I_{\text{direct}}^1$} \end{cases}
\end{eqnarray*}

\;\;\;

\textit{Case 2.}  $k<l$ and $ I_{\text{direct}}^2 \cup I_{\text{retro}}$; the winding number $K$ is odd.\\
We choose a retrograde $T_{k,l}$-type orbit.   Similar to   the case 1, the pulled back trajectory $L^{-1}(K)$ has $2k(k+l-1)$ double points and $(k+l+1)$ layers. Then we compute that
\begin{eqnarray*}
\mathcal{J}_2(K)  &=& 1 + 2k( k+l-1)   - \bigg(  1 \cdot (k+l)^2 + 2k \cdot (k+l-1)^2 + 2k \cdot(k+l-2)^2 + \cdots + 2k \cdot 1^2   \bigg) \\
&& \;\; \;\;\;\;\;\;\;\; \;\;\; \;\;\; \;\;\;\; \;\;\; \;\; \;\;+ \bigg(2  k \cdot (k+l-1)^2 +2 k \cdot (k+l-2)^2 + \cdots + 2k \cdot1^2  \bigg) \\
&=& 1 + 2k(k+l-1) -(k+l)^2\\
&=&   (k-1)^2 - l^2 .
\end{eqnarray*}

\begin{Remark}\rm Note that the   formulas for the previous two cases equal  the ones obtained by Proposition \ref{prorelation}. For example, we compute that if $k>l$
\begin{eqnarray*}
\mathcal{J}_2(K) &=& 2 \mathcal{J}_1 (K) -1\\
&=&  2 \bigg( 1- k + \frac{k^2}{2} - \frac{l^2}{2} \bigg) -1\\
&=& 1 - 2k + k^2 - l^2 \\
&=& (k-1)^2 - l^2.
\end{eqnarray*}
\end{Remark}

\;

\textit{Case 3.}  Either $k>l$ or $k<l$ and $ I_{\text{direct}}^1$; the winding number of $K$ is even. \\
In this case  $L : L^{-1}(K) \rightarrow K$ is also a 2$-$1 covering, but  $L^{-1}(K)$ consists of two generic immersions, say $L^{-1}(K) = \widetilde{K}\cup\widetilde{K}'$.   Fact (ii)  shows that the two curves are related by a $\pi/k$-rotation. Indeed, if $\widetilde{K}$ has a perihelion at $\theta=j \pi/k$ for some $j$, then the reflection fixes the perihelion and hence leaves $\widetilde{K}$ invariant. Since the other component $\widetilde{K}'$ is obtained from $\widetilde{K}$ under $\pi$-rotation, $\widetilde{K}'$ is invariant under the reflection as well. This shows that the $\pi/k$-rotation sends $\widetilde{K}$ to $\widetilde{K}'$. 

We claim that $\widetilde{K}$ has perihelions, aphelions and double points exactly on the rays $\theta=j\pi/k$. To see this, we assume that $\widetilde{K}$ has a perihelion at $\theta=0$. Then the double point of the preimage $L^{-1}(K)$ with the smallest radius on the ray $\theta=\pi/k$ also belongs to $\widetilde{K}$ due to the rotational symmetry given in Fact (i). In view of the reflection symmetry (with respect to the ray $\theta=\pi/k$), we also see that the second double points of $L^{-1}(K)$ on $\theta=0$ lies on $\widetilde{K}$. Proceeding further in a similar way, we find that the $2n$-th, $n\geq1$, double points of $L^{-1}(K)$ on $\theta=0$ belong  to $\widetilde{K}$. Since the perihelion on $\pi/k$ lies on $\widetilde{K}'$, the $(2n-1)$-th, $n\geq 1$, double points on $\theta=0$ belong  to $\widetilde{K}'$. Consequently, the double points on the ray $\theta=\pi/2k$ are single points for $\widetilde{K}$ and $\widetilde{K}'$. It follows that the complement of $\widetilde{K}$ has $(|k-l|/2-1)$ layers and hence there exist $k( |k-l|/2 -1)$ double points by the rotational symmetry (by angle $2\pi/k$), see Figure \ref{hoyleroot51}.
\begin{figure}[t]
\begin{subfigure}{0.45\textwidth}
  \centering
  \includegraphics[width=0.85\linewidth]{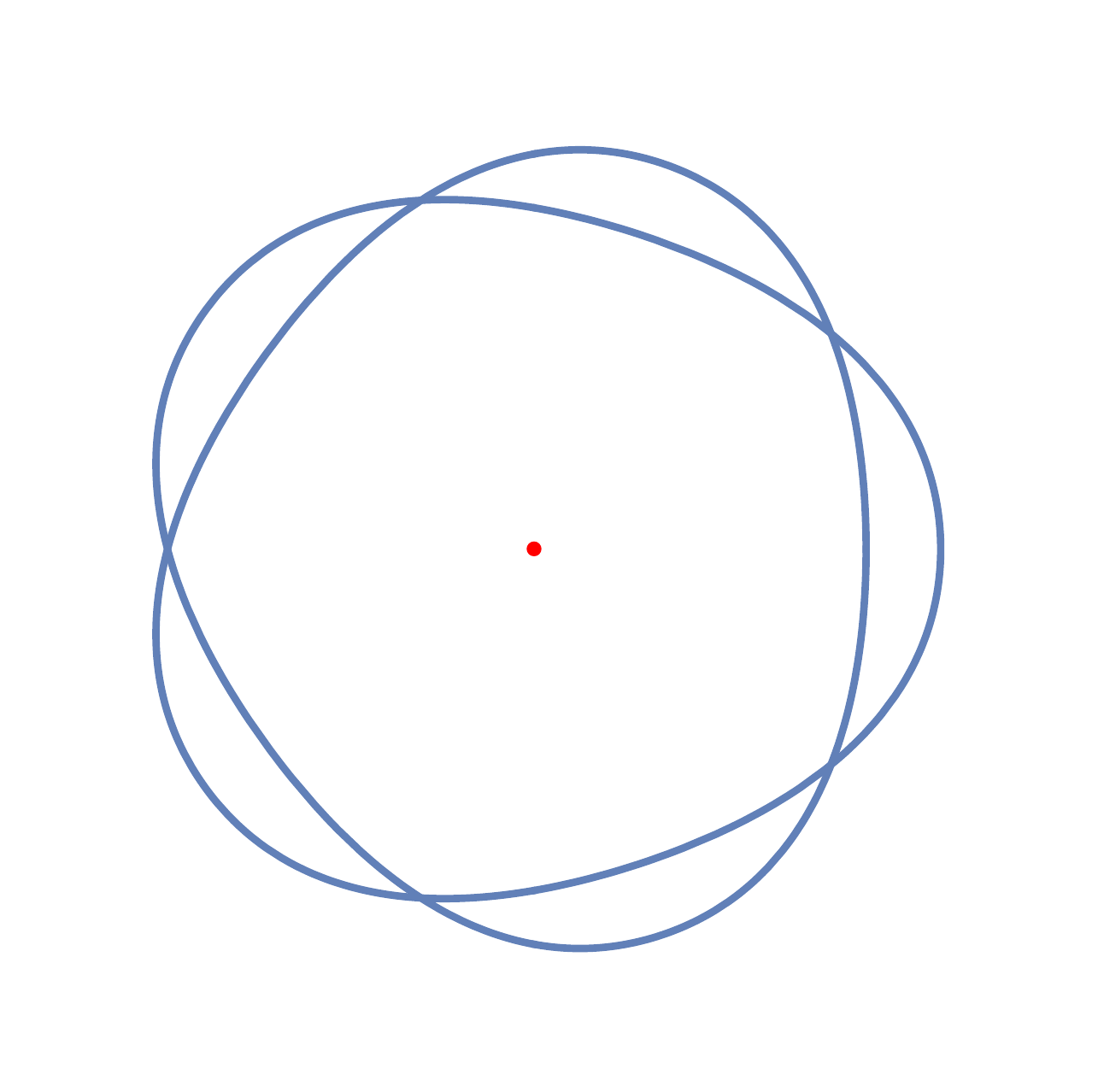}
  \caption{A pulled back orbit of a direct $T_{5,1}$-type orbit}
 \label{   }
\end{subfigure}
\begin{subfigure}{0.45\textwidth}
  \centering
  \includegraphics[width=0.85\linewidth]{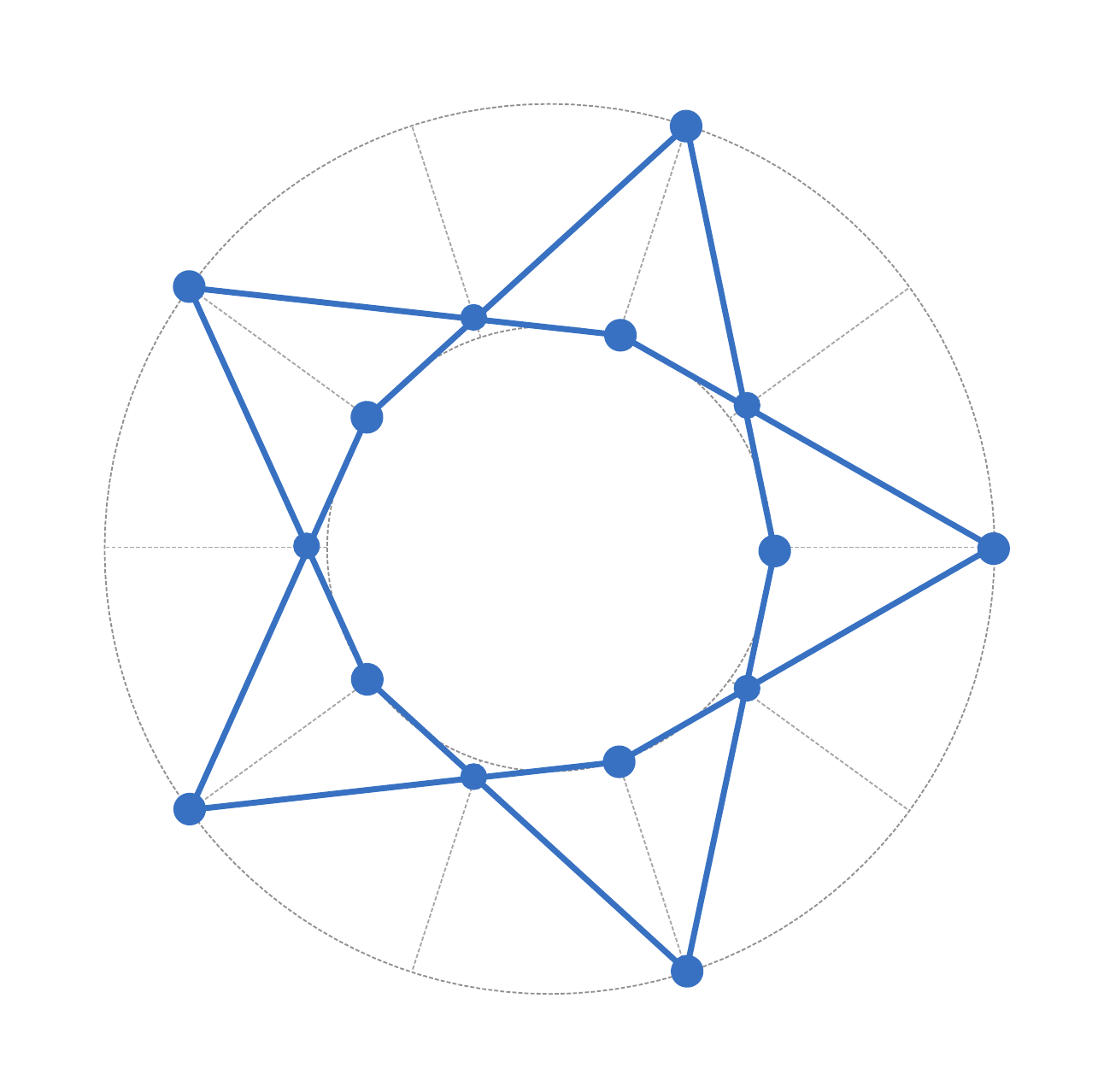}
  \caption{ An orbit obtained by the algorithm}
  \label{  }
\end{subfigure}
\caption{ A $L^{-1}(T_{5,1})$-type orbit.  There exist $5 \times (\frac{5-1}{2}-1) = 5$ double points and $\frac{5-1}{2}-1=1$ layer. Each layer has $5$ connected components. The winding number is $\frac{1-5}{2}=-2$.}
\label{hoyleroot51}
\end{figure}
In view of   $w_0 ( \widetilde{K} ) = (l-k)/2$, using  Viro's formula we find the $\mathcal{J}_2$ invariant  
\begin{eqnarray*}
\mathcal{J}_2(K) &=& J^+( \widetilde{K}  ) \\
&=& 1 + k\bigg( \frac{|k-l|}{2}-1 \bigg)   - \bigg(  1 \cdot \bigg( \frac{|k-l|}{2} \bigg)^2 + k \cdot \bigg( \frac{|k-l|}{2}-1 \bigg)^2   + \cdots + k \cdot 1^2   \bigg) \\
&& \;\; \;\;\;\;\;\;\;\;\;\; \;\;\; \;\;\; \;\;\; \;\; \;\;\; \;\;\; \;\; + \bigg(  k \cdot \bigg( \frac{|k-l|}{2}-1 \bigg)^2 + k \cdot \bigg( \frac{|k-l|}{2}-2 \bigg)^2 + \cdots + k \cdot1^2  \bigg)\\
&=& 1 + k\bigg( \frac{|k-l|}{2}-1 \bigg)   -   \bigg( \frac{|k-l|}{2} \bigg)^2     \\
&=& \begin{cases} 1 - k + \frac{1}{4}(k^2 -l^2) & \; \text{ if  $k>l$,} \\     1 -k -k^2 + \frac{3}{2}kl - \frac{1}{2}l^2           & \; \text{ if $k<l$ and $ I_{\text{direct}}^1$.} \end{cases}
\end{eqnarray*}

\;\;\;

\;\;\; 

\textit{Case 4.}  $k<l$ and $ I_{\text{direct}}^2 \cup I_{\text{retro}}$; the winding number of $K$ is even.\\
Similarly, for a retrograde $T_{k,l}$-type orbit $K$,  we see that $w_0 ( \widetilde{K} ) = (l+k)/2$ and there exist $((k+l)/2-1)$ layers and  $k( (k+l)/2 -1)$ double points. We then compute that
\begin{eqnarray*}
\mathcal{J}_2(K) &=& J^+(\widetilde{K}  ) \\
&=& 1 + k\bigg( \frac{k+l}{2}-1 \bigg)   - \bigg(  1 \cdot \bigg( \frac{k+l}{2} \bigg)^2 + k \cdot \bigg( \frac{k+l}{2}-1 \bigg)^2   + \cdots + k \cdot 1^2   \bigg) \\
&& \;\; \;\;\;\;\;\;\;\;\;\;  \; \;\;\; \;\;\; \;\; \;\;\; \;\;\; \;\; + \bigg(  k \cdot \bigg( \frac{k+l}{2}-1 \bigg)^2 + k \cdot \bigg( \frac{k+l}{2}-2 \bigg)^2 + \cdots + k \cdot1^2  \bigg) \\
&=& 1 + k\bigg( \frac{k+l}{2}-1 \bigg)   -   \bigg( \frac{k+l}{2} \bigg)^2\\
&=& 1 - k - \frac{ kl}{2} - \frac{l^2}{2}.
\end{eqnarray*}

\;\;

We have proven
\begin{Proposition}\label{jplusformeusladddd} Let $\alpha$ be a $T_{k,l}$-type orbit. Its $\mathcal{J}_1$ and $\mathcal{J}_2$  invariants are given by
 \begin{eqnarray*}
\mathcal{J}_1 (\alpha)&=& \begin{cases} 1 - k + \frac{k^2}{2} - \frac{l^2}{2} & \text{ if $k>l$ or if $k<l$ and $\alpha$ is retrograde} \\                1 - k - \frac{3}{2} k^2 + 2kl - \frac{1}{2} l^2         & \text{ if   $k<l$ and $e \in I_{\text{direct}}^1$}  \\               1 - k + \frac{k^2}{2} - \frac{l^2}{2}     & \text{ if   $k<l$ and $ e \in I_{\text{direct}}^2$ }  \end{cases} \\
\mathcal{J}_2 (\alpha)&=& \begin{cases} (k-1)^2  -l^2     & \text{ if  $k>l$ and $w(\alpha)$ is odd}    \\    1 - k + \frac{k^2}{4} - \frac{l^2}{4}     &  \text{ if $k>l$ and  $w(\alpha)$ is even}  \\  1-2k -3k^2 +4kl - l^2  & \text{ if  $k<l$, $ e\in I_{\text{direct}}^1$ and $w(\alpha)$ is odd}  \\ (k-1)^2  -l^2     & \text{ if  $k<l$, $ e \in I_{\text{direct}}^2 \cup I_{\text{retro}}$ and $w(\alpha)$ is odd} \\  1 - k - k^2 + \frac{3}{2}kl - \frac{1}{2}l^2     & \text{ if $k<l$, $ e \in I_{\text{direct}}^1$ and $w(\alpha)$ is even} \\    1 - k - \frac{1}{2}kl - \frac{1}{2}l^2    & \text{ if $k<l$, $ e \in I_{\text{direct}}^2 \cup I_{\text{retro}}$ and $w(\alpha)$ is even},      \end{cases}
\end{eqnarray*}
where $w(\alpha)$ denotes the winding number of $\alpha$ with respect to the origin.
\end{Proposition}

By the argument before Corollary \ref{sglihiwle333} we obtain the following
 
\begin{Corollary}\label{sglihiwle33dddddddd3} Let $\alpha$ be a $T_{k,l}$-type orbit which is a generic immersion. There exists a small $\mu_{k,l} >0$ such that for any $\mu < \mu_{k,l}$ all periodic orbits $\delta$ of the second kind which are obtained from $\alpha$ are generic immersions and have the same  invariants $\mathcal{J}_1$ and $\mathcal{J}_2$ as $\alpha$. 
\end{Corollary}

\;\;\;

\end{document}